\documentclass[leqno,english,letter]{amsart}
\usepackage{hyperref}
\usepackage{amssymb}
\usepackage[square,numbers]{natbib}
\usepackage[all]{xy}

\theoremstyle{plain}

\newtheorem{mainthm}{Theorem}

\newtheorem{mainlemm}{Lemma}

\newtheorem*{maincor}{Corollary}

\newtheorem*{concludingfact}{Fact}

\newtheorem{mainseclemm}{Lemma}
\newtheorem{mainsecprop}[mainseclemm]{Proposition}
\numberwithin{mainseclemm}{subsection}

\swapnumbers
\newtheorem{thm}[subsubsection]{Theorem}
\newtheorem{lemm}[subsubsection]{Lemma}
\newtheorem{prop}[subsubsection]{Proposition}

\newtheorem{fact}[subsubsection]{Fact}
\newtheorem{obsv}[subsubsection]{Observation}

\DeclareMathOperator{\End}{\mathit{End}}
\DeclareMathOperator{\Hom}{\mathit{Hom}}

\DeclareMathOperator{\Simp}{\mathcal{S}}

\DeclareMathOperator{\Op}{\mathcal{O}}

\DeclareMathOperator{\NN}{\mathbb{N}}
\DeclareMathOperator{\ZZ}{\mathbb{Z}}
\DeclareMathOperator{\kk}{\Bbbk}
\DeclareMathOperator{\FF}{\mathbb{F}}
\DeclareMathOperator{\QQ}{\mathbb{Q}}

\DeclareMathOperator{\AOp}{\mathtt{A}}
\DeclareMathOperator{\BOp}{\mathtt{B}}
\DeclareMathOperator{\COp}{\mathtt{C}}
\DeclareMathOperator{\DOp}{\mathtt{D}}
\DeclareMathOperator{\EOp}{\mathtt{E}}
\DeclareMathOperator{\FOp}{\mathtt{F}}
\DeclareMathOperator{\GOp}{\mathtt{G}}
\DeclareMathOperator{\HOp}{\mathtt{H}}
\DeclareMathOperator{\IOp}{\mathtt{I}}
\DeclareMathOperator{\KOp}{\mathtt{K}}
\DeclareMathOperator{\LOp}{\mathtt{L}}
\DeclareMathOperator{\MOp}{\mathtt{M}}
\DeclareMathOperator{\POp}{\mathtt{P}}
\DeclareMathOperator{\QOp}{\mathtt{Q}}

\DeclareMathOperator{\WOp}{\mathtt{W}}

\DeclareMathOperator{\C}{\mathcal{C}}
\DeclareMathOperator{\E}{\mathcal{E}}
\DeclareMathOperator{\I}{\mathcal{I}}
\DeclareMathOperator{\K}{\mathcal{K}}

\DeclareMathOperator{\X}{\mathcal{X}}

\DeclareMathOperator{\Id}{Id}
\DeclareMathOperator{\id}{id}
\DeclareMathOperator{\sgn}{\mathit{sgn}}

\DeclareMathOperator*{\colim}{colim}
\DeclareMathOperator{\cst}{\mathit{cst}}
\DeclareMathOperator{\ev}{\mathit{ev}}
\DeclareMathOperator{\weight}{weight}
\DeclareMathOperator{\ad}{ad}
\DeclareMathOperator{\sk}{sk}
\DeclareMathOperator{\Cst}{\mathit{C}^{st}}
\DeclareMathOperator{\NDeg}{\mathtt{N}}
\DeclareMathOperator{\Span}{Span}

\DeclareMathOperator{\edge}{\eta}

\title[Koszul duality of $\mathrm{E_n}$-operads]{Koszul duality of $\mathbf{E_n}$-operads}
\author{Benoit Fresse}
\date{20 April 2009 (minor writing improvements on 26 April, 30 April, 15 November 2009)}
\address{UMR 8524 de l'Universit\'e Lille 1 - Sciences et Technologies - et du CNRS\\
Cit\'e Scientifique -- B\^atiment M2\\
F-59655 Villeneuve d'Ascq C\'edex (France)}
\email{Benoit.Fresse@math.univ-lille1.fr}
\urladdr{http://math.univ-lille1.fr/\~{ }fresse}
\thanks{Research supported in part by grant ANR-06-JCJC-0042 ``OBTH''}

\begin{document}

\begin{abstract}
The goal of this paper is to prove a Koszul duality result for $E_n$-operads in differential graded modules over a ring.
The case of an $E_1$-operad, which is equivalent to the associative operad, is classical.
For $n>1$, the homology of an $E_n$-operad is identified with the $n$-Gerstenhaber operad
and forms another well known Koszul operad.
Our main theorem asserts that an operadic cobar construction on the dual cooperad of an $E_n$-operad $\EOp_n$
defines a cofibrant model of~$\EOp_n$.
This cofibrant model gives a realization at the chain level
of the minimal model of the $n$-Gerstenhaber operad
arising from Koszul duality.

Most models of $E_n$-operads in differential graded modules come in nested sequences $\EOp_1\subset\EOp_2\subset\cdots\subset\EOp_{\infty}$
homotopically equivalent to the sequence of the chain operads of little cubes.
In our main theorem, we also define a model of the operad embeddings $\EOp_{n-1}\hookrightarrow\EOp_n$
at the level of cobar constructions.
\end{abstract}

\maketitle

\tableofcontents

\section*{Introduction}

The bar duality of operads~\cite{GetzlerJones,GinzburgKapranov} shows that any augmented differential graded operad $\POp$
has a quasi-free model given by a cobar construction $\BOp^c(\DOp)$
on a cooperad $\DOp$ associated to $\POp$.
The notion of a Koszul operad, introduced in~\cite{GinzburgKapranov},
refers to certain good operads $\POp$
together with a minimal model $\BOp^c(\KOp(\POp))\xrightarrow{\sim}\POp$
such that $\KOp(\POp)$ is a cooperad defined by an explicit presentation by generators and relations
which is in a sense dual to a presentation of $\POp$.
The cooperad $\KOp(\POp)$ is called the Koszul dual of $\POp$
and determines the homology of all cooperads $\DOp$
such that $\BOp^c(\DOp)\xrightarrow{\sim}\POp$.
The operad of commutative algebras $\COp$,
which has a desuspension of the cooperad of Lie coalgebras
as a Koszul dual cooperad $\KOp(\COp) = \Lambda^{-1}\LOp^{\vee}$,
gives a classical example of a Koszul operad.
The theory of Koszul operads is established in a characteristic zero setting
in the original reference~\cite{GinzburgKapranov},
but we prove in~\cite{FressePartitions} that the notion of a Koszul operad has a suitable generalization
so that the above assertions make sense in any category of modules over a ring.

\medskip
In this work,
we determine quasi-free models of the form $\BOp^c(\DOp)$
for $E_n$-operads in differential graded modules,
where an $E_n$-operad refers to an operad weakly-equivalent to the chain operad of Boardman-Vogt little $n$-cubes.
The obtained operads $\BOp^c(\DOp_n)$ are cofibrant objects of the model category of operads,
and hence define cofibrant models of $E_n$-operads,
for any choice of ground ring~$\kk$
(for our purpose, we can take $\kk = \ZZ$).

The topological little cubes operads
have been introduced at the origin of the theory of operads in homotopy theory~\cite{BoardmanVogt,May}.
Since then,
many actions of $E_n$-operads have been discovered in algebra
and the idea is now well established that the notion of an $E_n$-operad gives the right device
to understand the degree of commutativity
of a multiplicative structure.
To give only one reference,
we cite~\cite{Kontsevich}
for a comprehensive account of a program which connects the homotopy of $E_n$-operads
to topological field theory and motivic Galois groups.
These developments give strong motivation for understanding the homotopy of cofibrant models of $E_n$-operads.
In fact,
the work of E. Getzler and J. Jones~\cite{GetzlerJones},
at the beginning of the theory of Koszul operads,
included a first attempt to settle this question.
A specific motivating problem, which we plan to study in subsequent work,
comes from the relationship, conjectured in~\cite{Kontsevich},
between the Grothendieck-Teichm\"uller group and the groups of homotopy automorphisms of $E_n$-operads,
whose proper definition in homotopy theory involves the choice of a cofibrant model.

\medskip
According to~\cite{Cohen},
the homology of the operad of little $n$-cubes $\COp_n$, and hence of any $E_n$-operad $\EOp_n$,
forms an operad in graded modules isomorphic to the operad of $n$-Gerstenhaber algebras $\GOp_n$ for $n>1$,
to the associative operad $\AOp$ for $n=1$.
The operads $\GOp_n$ are Koszul (and so does the associative operad $\AOp$),
and their dual $\KOp(\GOp_n) = \Lambda^{-n}\GOp_n^{\vee}$
is an operadic $n$-fold desuspension of the cooperad $\GOp_n^{\vee}$
dual to $\GOp_n$ in $\ZZ$-modules (we apply our work~\cite{FressePartitions} to give a sense to this statement
in the context of $\ZZ$-modules).
This duality result amounts to the existence of a weak-equivalence
$\epsilon: \BOp^c(\Lambda^{-n}\GOp_n^{\vee})\xrightarrow{\sim}\GOp_n$
in the category of differential graded operads,
where $\GOp_n$ is viewed as a differential graded operad equipped with a trivial differential.

Our first purpose is to prove that an $E_n$-operad $\EOp_n$
has a cofibrant model of the form $\BOp^c(\Lambda^{-n}\EOp_n^{\vee})$,
where $\DOp = \Lambda^{-n}\EOp_n^{\vee}$ is an operadic $n$-fold desuspension
of the cooperad $\EOp_n^{\vee}$
dual to $\EOp_n$ in $\ZZ$-modules,
so that we have a weak-equivalence $\epsilon: \BOp^c(\Lambda^{-n}\EOp_n^{\vee})\xrightarrow{\sim}\EOp_n$
which in a sense realizes the morphism $\epsilon: \BOp^c(\Lambda^{-n}\GOp_n^{\vee})\xrightarrow{\sim}\GOp_n$
when we take the homology of $\EOp_n$.

In the characteristic zero context,
our result can be deduced from the formality theorem of~\cite{Kontsevich}
which asserts that an $E_n$-operad $\EOp_n$ is weakly-equivalent to its homology $\GOp_n = H_*(\EOp_n)$.
In the context of $\ZZ$-modules, addressed in this paper,
we have to introduce new methods,
because the formality theorem does not hold anymore (for instance, just because $\EOp_n$ is no more formal as a $\Sigma_*$-object).

\medskip
Usual models of $E_n$-operads come in nested sequences
\begin{equation}\label{eqn:NestedEnOperads}
\EOp_1\subset\EOp_2\subset\cdots\subset\colim_n\EOp_n = \EOp,
\end{equation}
where $\EOp$ is an $E_\infty$-operad,
an operad weakly-equivalent to the operad of commutative algebras $\COp$.
In a previous work (with C. Berger),
we made explicit an operad morphism $\sigma: \EOp\rightarrow\Lambda^{-1}\EOp$ for a particular $E_\infty$-operad $\EOp$,
the Barratt-Eccles operad.
This morphism determines a natural action of $\EOp$
on the suspension of $\EOp$-algebras (where we consider the standard suspension of differential graded modules).
One observes that $\sigma$ maps $\EOp_n$ into $\Lambda^{-1}\EOp_{n-1}$
and yields an operad morphism $\sigma: \EOp_n\rightarrow\Lambda^{-1}\EOp_{n-1}$
for every $n>1$.
In the present article,
we prove that the morphisms $\sigma^*: \BOp^c(\Lambda^{1-n}\EOp_{n-1}^{\vee})\rightarrow \BOp^c(\Lambda^{-n}\EOp_{n}^{\vee})$
induced by $\sigma: \EOp_n\rightarrow\Lambda^{-1}\EOp_{n-1}$
on cobar constructions
fit a commutative diagram
\begin{equation}\label{NestedOperadEquivalences}
\vcenter{\xymatrix{ \BOp^c(\Lambda^{-1}\EOp_{1}^{\vee})\ar@{.>}[r]\ar[d]_{\sim} &
\cdots\ar@{.>}[r] &
\BOp^c(\Lambda^{1-n}\EOp_{n-1}^{\vee})\ar@{.>}[r]\ar@{.>}[d]_{\sim} &
\BOp^c(\Lambda^{-n}\EOp_n^{\vee})\ar@{.>}[r]\ar@{.>}[d]_{\sim} & \cdots \\
\EOp_1\ar@{^{(}->}[]!R+<4pt,0pt>;[r] &
\cdots\ar@{^{(}->}[]!R+<4pt,0pt>;[r] &
\EOp_{n-1}\ar@{^{(}->}[]!R+<4pt,0pt>;[r] &
\EOp_n\ar@{^{(}->}[]!R+<4pt,0pt>;[r] & \cdots }}
\end{equation}
and, hence, represent the operad embeddings $\iota: \EOp_{n-1}\hookrightarrow\EOp_n$
at the level of cofibrant models.

\medskip
This paper is not the first attempt to prove a Koszul duality result
for $E_n$-operads.
Besides~\cite{GetzlerJones},
results close to ours are stated in~\cite{PoHu}.
But several relations used in that reference (along~\cite[\S\S 2-3]{PoHu})
seem valid in homotopy categories only
and this is not enough for constructions of the genuine Koszul duality.

\subsubsection*{Acknowledgements}
I am very grateful to Mike Mandell for a decisive observation on a result of~\cite{BergerFresse}
which gives the starting point of this work.

The results and ideas of this article have been announced in a talk given at the~``Third Arolla Conference on Algebraic Topology'',
held in Arolla (Switzerland), August 2008.
I would like to thank the organizers of this conference for the opportunity
of giving this first announcement.

\subsection*{Contents}
In the prologue,
we review the definition of the Barratt-Eccles operad, of the operadic cobar construction,
and we explain the statement of our main result.

A first step toward the definition of the weak-equivalences of~(\ref{NestedOperadEquivalences})
is the construction of morphisms toward the commutative operad:
\begin{equation}\label{CobarAugmentation}
\colim_n \BOp^c(\Lambda^{-n}\EOp_n^{\vee})\xrightarrow{\epsilon}\COp
\end{equation}
This step is carried out in~\S\ref{FirstStep}.

In the interlude~\S\ref{Interlude},
we revisit the definition of particular cell structures, called $\K$-structures in the paper,
which refine filtration~(\ref{eqn:NestedEnOperads}).
In~\S\ref{SecondStep},
we prove that the operads $\BOp^c(\Lambda^{-n}\EOp_n^{\vee})$
inherit such a $\K$-structure.
Any $E_\infty$-operad $\EOp$
is, by definition,
equipped with a weak-equivalence of operads $\EOp\xrightarrow{\sim}\COp$
and morphism~(\ref{CobarAugmentation})
admits a lifting to $\EOp$.
In~\S\ref{FinalStep},
we use the cell structures of~\S\S\ref{Interlude}-\ref{SecondStep}
for proving that the obtained morphism $\epsilon: \colim_n \BOp^c(\Lambda^{-n}\EOp_n^{\vee})\rightarrow\EOp$
restricts to a sequence of morphisms as in diagram~(\ref{NestedOperadEquivalences}).
Then we use the Koszul duality of the Gerstenhaber operads to conclude that these morphisms are weak-equivalences.

In the epilogue,
we explain applications of our results to the definition of cochain models of spheres $\bar{N}^*(S^m)$
in algebras over operads.

\subsection*{Conventions and background}
Throughout this paper, we deal with operads in the category of differential graded modules.

The ring of integers $\ZZ$
forms our ground ring.
For us a differential graded module (a dg-module for short) refers to a lower $\ZZ$-graded $\ZZ$-module $C$
equipped with a differential, usually denoted by $\delta: C\rightarrow C$,
that decreases gradings by $1$.
The category of dg-modules is denoted by $\C$.
We equip this category with its standard model structure for which a morphism is a weak-equivalence
if it induces an isomorphism in homology, a fibration if it is degreewise surjective.

We adopt conventions of~\cite{FresseCylinder}
and we refer to this article for a survey of the homotopy theory of operads in dg-modules
and further bibliographical references.
To help the reader,
we have included a short glossary of notation
at the end of the article.

We only consider operads $\POp$ such that $\POp(0) = 0$ and $\POp(1) = \ZZ$
-- we then say that the operad $\POp$ is connected.
We use the notation $\Op_1$ to refer to the category of connected operads.
Any connected operad $\POp$ is equipped with an augmentation $\epsilon: \POp\rightarrow I$
just given by the identity of $\ZZ$ in arity $r=1$.
We use the notation $\bar{\POp}$ to refer to the augmentation ideal of $\POp$.
We have obviously $\bar{\POp}(0) = \bar{\POp}(1) = 0$ and $\bar{\POp}(r) = \POp(r)$ for $r\geq 2$.

We apply an extension of the operadic Koszul duality of~\cite{GinzburgKapranov}
to connected operads over rings.
Our reference for this generalization is~\cite{FressePartitions}.

We also use the survey of~\cite{FresseCylinder},
where we study the generalization of the bar duality of operads in the context of unbounded dg-modules
since we deal with operads defined in that category of dg-modules.
The arguments of~\cite{FressePartitions} only work for non-negatively graded objects,
but this will be the case of the operads and cooperads
to which we apply the results of that reference.

\section*{Prologue: statement of the main result}\label{Prelude}
The purpose of this section is to review the definition of objects involved in the statement of our main theorem.
First,
we review the definition of the Barratt-Eccles operad,
the filtration of the Barratt-Eccles by $E_n$-operads,
and
we define the morphisms $\sigma^*: \Lambda^{1-n}\EOp_{n-1}^{\vee}\rightarrow\Lambda^{-n}\EOp_n^{\vee}$
which fit the upper row of diagram~(\ref{NestedOperadEquivalences})
in the introduction.
Then we recall briefly the definition of the operadic cobar construction
and we state our Koszul duality theorem,
the main objective of this work.

\subsection{The Barratt-Eccles operad and the suspension morphism}\label{Prelude:BarrattEccles}
The original Barratt-Eccles operad, defined in~\cite{BarrattEccles},
is an $E_\infty$-operad in simplicial sets
formed by the universal $\Sigma_r$-bundles $E\Sigma_r$,
where $\Sigma_r$ denotes the group of permutations of $\{1,\dots,r\}$.
Throughout this paper,
we use the dg-operad $\EOp$
defined by the normalized complexes
of this simplicial operad:
\begin{equation*}
\EOp(r) = N_*(E\Sigma_r).
\end{equation*}
But,
we take the convention, contrary to the standard definition,
that $\EOp$ forms a connected operad
so that $\EOp(0) = 0$.

The purpose of the next paragraphs is only to recall the main features of the operad $\EOp$
used in the statement of our main theorem.
For more details on this operad,
we refer to the paper~\cite{BergerFresse}
of which we take our conventions.

\subsubsection{The Barratt-Eccles operad}\label{Prelude:BarrattEccles:Definition}
Recall that an $E_\infty$-operad in dg-modules
is a dg-operad $\EOp$
together with a weak-equivalence $\epsilon: \EOp\xrightarrow{\sim}\COp$,
where $\COp$ is the operad of commutative algebras.

For the Barratt-Eccles operad:
\begin{itemize}
\item
The dg-module $\EOp(r)$ is spanned in degree $d$ by non-degenerate $d$-simplices of permutations $(w_0,\dots,w_d)$
-- a simplex $(w_0,\dots,w_d)$ is non-degenerate if we have $w_j \not= w_{j+1}$ for every $j$.
The differential of $\EOp(r)$ is given by the usual formula
\begin{equation*}
\delta(w_0,\dots,w_d) = \sum_{i=0}^{d} \pm(w_0,\dots,\widehat{w_i},\dots,w_d).
\end{equation*}
\item
The composition structure is yielded by an explicit substitution process on permutations
(the definition of this part of the structure not needed untill~\S\ref{SecondStep:OperadKStructure}
and we put off the corresponding recollections untill that section).
\item
The operad weak-equivalence $\epsilon: \EOp\xrightarrow{\sim}\COp$ toward the commutative operad $\COp$
is given by the standard chain augmentations $\epsilon: N_*(E\Sigma_r)\rightarrow\ZZ$,
which are defined by:
\begin{equation*}
\epsilon(w_0,\dots,w_d) = \begin{cases} 1, & \text{if $d=0$}, \\ 0, & \text{otherwise}. \end{cases}
\end{equation*}
\end{itemize}

\subsubsection{The degree $0$ part of the Barratt-Eccles operad}\label{Prelude:BarrattEccles:Degree0}
The Barratt-Eccles operad vanishes in degree $*<0$ and consists of finitely generated free $\ZZ$-modules
in degree $*\geq 0$.

The degree $0$ parts of the chain complexes $\EOp(r)$ form a suboperad of the Barratt-Eccles operad
such that $\EOp(r)_0 = \ZZ[\Sigma_r]$,
the free $\ZZ$-module generated by the permutations of $\{1,\dots,r\}$.
This suboperad can be identified with the operad of associative algebras $\AOp$
for which we also have $\AOp(r) = \ZZ[\Sigma_r]$.
Hence,
the Barratt-Eccles operad sits in a factorization
\begin{equation*}
\xymatrix{ \AOp\ar[dr]\ar[rr] && \COp \\
& \EOp\ar[ur]_{\sim} & }
\end{equation*}
of the morphism $\alpha: \AOp\rightarrow\COp$
which represents the usual embedding $\alpha^*: {}_{\COp}\C\hookrightarrow{}_{\AOp}\C$
from the category of commutative algebras~${}_{\COp}\C$
to the category of associative algebras~${}_{\AOp}\C$.

\subsubsection{The little cubes filtration of the Barratt-Eccles operad}\label{Prelude:BarrattEccles:LittleCubesFiltration}
The Barratt-Eccles chain operad has a filtration by suboperads $\EOp_n\subset\EOp$
which form a nested sequence of operads~(\ref{eqn:NestedEnOperads})
weakly-equivalent
to the sequence of the chain operads of little cubes.
The definition of this filtration, in the simplicial setting,
goes back to~\cite{Smith}.
The existence of weak-equivalences with little cubes operads is established in~\cite{BergerCell}.
For our purpose,
we recall briefly an explicit definition of the filtration
at the chain level.

Any permutation $w\in\Sigma_r$ is represented by the sequence of its values $(w(1),\dots,w(r))$.
For any pair $\{i,j\}\subset\{1,\dots,r\}$,
the notation $w|_{i j}$
refers to the permutation of $\{i,j\}$ formed by the occurrences of $\{i,j\}$
in the sequence $(w(1),\dots,w(r))$.
For instance,
we have $(3,1,2)|_{1 2} = (1,2)$, $(3,1,2)|_{1 3} = (3,1)$, $(3,1,2)|_{2 3} = (3,2)$.
For a simplex $\underline{w} = (w_0,\dots,w_d)$,
we denote by $\mu_{i j}(\underline{w})$
the number of variations in the sequence $\underline{w}|_{i j} = (w_0|_{i j},\dots,w_d|_{i j})$.
For instance, for the simplex $\underline{w} = (1 2 3,3 1 2,3 2 1)$,
we have
\begin{align*}
\underline{w}|_{1 2} = (1 2,1 2,2 1) & \Rightarrow\mu_{1 2}(\underline{w}) = 1, \\
\underline{w}|_{1 3} = (1 3,3 1,3 1) & \Rightarrow\mu_{1 3}(\underline{w}) = 1, \\
\underline{w}|_{2 3} = (2 3,3 2,3 2) & \Rightarrow\mu_{2 3}(\underline{w}) = 1.
\end{align*}
The dg-module $\EOp_n(r)$
is spanned by the simplices of permutations $\underline{w} = (w_0,\dots,w_d)$
such that $\mu_{i j}(\underline{w})<n$
for all pairs $\{i,j\}\subset\{1,\dots,r\}$.
For instance,
we have $(1 2 3,3 1 2,3 2 1)\in\EOp_2(3)$ since we observe that $\mu_{1 2}(\underline{w}) = \mu_{1 3}(\underline{w}) = \mu_{2 3}(\underline{w}) = 1$
for this simplex.

For a simplex $\underline{w} = (w_0,\dots,w_d)$,
we have clearly:
\begin{equation*}
(\mu_{i j}(\underline{w}) = 0\ (\forall i j))\Rightarrow(w_0 = w_1 = \dots = w_d).
\end{equation*}
Hence we obtain:

\begin{obsv}\label{Prelude:BarrattEccles:FirstFiltrationLayer}
The subcomplex $\EOp_1(r)\subset\EOp(r)$ reduces to the degree $0$ part of~$\EOp(r)$.
\end{obsv}

From which we deduce:

\begin{prop}\label{Prelude:BarrattEccles:FirstLayerReduction}
The suboperad $\EOp_1\subset\EOp$ is identified with the operad of associative algebras $\AOp$.
\end{prop}

\subsubsection{The endomorphism operad of normalized chain complexes}\label{Prelude:BarrattEccles:EndomorphismOperads}
Recall that the endomorphism operad of a dg-module $C$
is the operad $\End_C$
such that
\begin{equation*}
\End_C(n) = \Hom_{\C}(C^{\otimes n},C),
\end{equation*}
the dg-module of dg-module homomorphisms $f: C^{\otimes n}\rightarrow C$.
The endomorphism operad $\End_C$ is the universal operad acting on $C$.

The endomorphism operad of a functor $C: \X\rightarrow\C$
on a category $\X$
is defined by the ends
\begin{equation*}
\End_{C(-)}(r) = \int_{X\in\X}\Hom_{\C}(C(X)^{\otimes r},C(X))
\end{equation*}
and forms the universal operad acting on the dg-modules $C(X)\in\C$
functorially in~$X\in\X$.
This definition of the endomorphism~$\End_{C(-)}$
makes sense when the functor $C: \X\rightarrow\C$
is defined on a small category $\X$.

We consider the functor of reduced normalized cochains $\bar{N}^*: \Simp^{op}\rightarrow\C$
on the category of simplicial sets $\Simp$.
We can extend the definition of the endomorphism operad to this functor $\bar{N}^*(X)$
though the category of simplicial sets is not small,
because any simplex $x\in X_n$ in a simplicial set $X$
is represented by a morphism $\sigma_x: \Delta^n\rightarrow X$,
where $\Delta^n$ denotes the standard model of the $n$-simplex in $\Simp$.

The standard multiplicative structure of the cochain complex $\bar{N}^*(X)$, defined by the Alexander-Whitney diagonal,
can be represented by a morphism
\begin{equation*}
\nabla_{a}: \AOp\rightarrow\End_{\bar{N}^*(-)}
\end{equation*}
on the associative operad $\AOp$.
The next result gives a finer structure on $\bar{N}^*(X)$:

\begin{thm}[C. Berger, B. Fresse, see~\cite{BergerFresse}]\label{Prelude:BarrattEccles:CochainMultiplications}
The operad morphism $\nabla_{a}: \AOp\rightarrow\End_{\bar{N}^*(-)}$
admits an extension to the Barratt-Eccles operad
\begin{equation*}
\xymatrix{ \AOp\ar[d]\ar[r]^(0.3){\nabla_{a}} & \End_{\bar{N}^*(-)} \\ \EOp\ar@{.>}[ur]_{\exists\nabla_{\epsilon}} & }
\end{equation*}
so that the reduced normalized cochain complex $\bar{N}^*(X)$
of any simplicial set $X$
inherits the structure of an algebra over $\EOp$.
\end{thm}

The existence of an $E_\infty$-structure on normalized cochain complexes of simplicial sets
is proved by general arguments in~\cite{HinichSchechtman}.
Explicit constructions of such $E_\infty$-structures are given in~\cite{BergerFresse,McClureSmith}.
The action of the Barratt-Eccles operad, used in this paper, is only defined in~\cite{BergerFresse}.

By theorems of M. Mandell~\cite{Mandell,MandellIntegral} (see also~\cite{SmirnovChain}),
the $E_\infty$-algebra $\bar{N}^*(X)$
is sufficient to determine the homotopy type of $X$
provided that $X$
satisfies standard finiteness and completeness assumptions.
This result was a first motivation
for the constructions of~\cite{BergerFresse,McClureSmith}.

\subsubsection{The action of the Barratt-Eccles operad on $S^1$}\label{Prelude:BarrattEccles:CircleCochainMultiplications}
For any fixed simplicial set $X\in\Simp$,
we have a natural operad morphism $\ev_X: \End_{\bar{N}^*(-)}\rightarrow\End_{\bar{N}^*(X)}$
which arises from the definition of $\End_{\bar{N}^*(-)}$
as an end.
The action of the Barratt-Eccles operad on the cochain complex of a fixed space $X$
is determined by the morphism $\nabla_X: \EOp\rightarrow\End_{\bar{N}^*(X)}$
defined by the composite
\begin{equation*}
\EOp\xrightarrow{\nabla_{\epsilon}}\End_{\bar{N}^*(-)}\xrightarrow{\ev_X}\End_{\bar{N}^*(X)},
\end{equation*}
with the morphism $\nabla_{\epsilon}$
of Theorem~\ref{Prelude:BarrattEccles:CochainMultiplications}

For the standard simplicial model of the circle $S^1 = \Delta^1/\partial\Delta^1$,
we have $\bar{N}^*(S^1) = \ZZ[-1]$,
where $\ZZ[d]$ denotes the free $\ZZ$-module of rank $1$
viewed as a dg-module concentrated in degree $d$ (with respect to lower gradings).
Hence, we have
\begin{equation*}
\End_{\bar{N}^*(S^1)}(r) = \Hom_{\C}(\ZZ[-1]^{\otimes r},\ZZ[-1]) = \ZZ[r-1],
\end{equation*}
and the morphism $\nabla_{S^1}: \EOp(r)\rightarrow\End_{\bar{N}^*(S^1)}(r)$
is determined by a cochain $\sgn: \EOp(r)\rightarrow\ZZ$
of degree $r-1$.

According to~\cite{BergerFresse},
the cochains $\sgn: \EOp(r)\rightarrow\ZZ$, $r\in\NN$,
are given by the formula:
\begin{equation*}
\sgn(w_0,\dots,w_d) = \begin{cases}
\pm 1, & \text{if $(w_0(1),\dots,w_d(1))$}\\ & \qquad\text{is a permutation}\\ & \qquad\text{of $(1,\dots,r)$},\\
0, & \text{otherwise}. \end{cases}
\end{equation*}
The sign $\pm$ is determined by the signature of the permutation $(w_0(1),\dots,w_d(1))$.

Note: this definition gives a cochain of degree $r-1$ (as required)
since $(w_0(1),\dots,w_d(1))$ has no chance to form a permutation of $(1,\dots,r)$
when $d\not=r-1$.

\subsubsection{Suspension morphisms}\label{Prelude:BarrattEccles:SuspensionMorphisms}
Recall that the operadic suspension of an operad $\POp$
is an operad $\Lambda\POp$
such that:
\begin{equation*}
\Lambda\POp(r) = \Sigma^{1-r}\POp(r)^{\pm},
\end{equation*}
where $\Sigma$ refers to the standard suspension of the category of dg-modules
and the exponent $\pm$ refers to a twist of the $\Sigma_r$-action
by the signature of permutations
(see~\cite[\S 1.3]{GetzlerJones}).

According to~\cite{BergerFresse},
the cap products with the cochains of~\S\ref{Prelude:BarrattEccles:CircleCochainMultiplications}
\begin{equation*}
\sgn\cap(w_0,\dots,w_d) = \sgn(w_0,\dots,w_{r-1})\cdot(w_{r-1},\dots,w_d)
\end{equation*}
define an operad morphism
\begin{equation*}
\sigma: \EOp\rightarrow\Lambda^{-1}\EOp
\end{equation*}
that we call the suspension morphism (of the Barratt-Eccles chain operad).

\medskip
The next observation has been made to the author by M. Mandell (in Spring~2002):

\begin{obsv}\label{Prelude:BarrattEccles:SuspensionLittleCubesLayers}
The suspension morphism $\sigma: \EOp\rightarrow\Lambda^{-1}\EOp$
admits factorizations
\begin{equation*}
\xymatrix{ \EOp_1\ar@{^{(}->}[]!R+<4pt,0pt>;[r]\ar@{.>}[d]^{\sigma} & \EOp_2\ar@{^{(}->}[]!R+<4pt,0pt>;[r]\ar@{.>}[d]^{\sigma} &
\cdots\ar@{^{(}->}[]!R+<4pt,0pt>;[r] &
\EOp_{n}\ar@{^{(}->}[]!R+<4pt,0pt>;[r]\ar@{.>}[d]^{\sigma} &
\cdots\ar@{^{(}->}[]!R+<4pt,0pt>;[r] &
\EOp\ar[d]^{\sigma} \\
\Lambda^{-1}\IOp\ar@{^{(}->}[]!R+<4pt,0pt>;[r] & \Lambda^{-1}\EOp_1\ar@{^{(}->}[]!R+<4pt,0pt>;[r] &
\cdots\ar@{^{(}->}[]!R+<4pt,0pt>;[r] &
\Lambda^{-1}\EOp_{n-1}\ar@{^{(}->}[]!R+<4pt,0pt>;[r] &
\cdots\ar@{^{(}->}[]!R+<4pt,0pt>;[r] &
\Lambda^{-1}\EOp },
\end{equation*}
where $\IOp$ denotes the composition unit of the category of $\Sigma_*$-objects.
\end{obsv}

This observation follows from a straightforward inspection
of the definition of the filtration in~\S\ref{Prelude:BarrattEccles:LittleCubesFiltration}.

In a sense,
the actual goal of this paper is to give an interpretation of the morphisms $\sigma: \EOp_n\rightarrow\Lambda^{-1}\EOp_{n-1}$
in terms of the homotopy of $E_n$-operads.

\subsection{Operadic cobar constructions and homological Koszul duality}\label{Prelude:KoszulDuality}
Throughout the paper,
we only consider connected cooperads $\DOp$ which,
just like connected operads,
have $\DOp(0) = 0$ and $\DOp(1) = \ZZ$.
According to~\cite[\S\S 3.1-3.3]{FresseCylinder},
this assumptions prevent all difficulties arising in the duality between operads and cooperads.

A connected cooperad $\DOp$ is naturally coaugmented over the composition unit of $\Sigma_*$-objects
and the coaugmentation coideal of $\DOp$
is identified with the $\Sigma_*$-object such that
\begin{equation*}
\bar{\DOp}(r) = \begin{cases} 0, & \text{if $r = 0,1$}, \\ \DOp(r), & \text{otherwise}. \end{cases}
\end{equation*}

The operadic cobar construction is a functor which maps a connected cooperad~$\DOp$ to a quasi-free operad
\begin{equation*}
\BOp^c(\DOp) = (\FOp(\Sigma^{-1}\bar{\DOp}),\partial),
\end{equation*}
where
\begin{equation*}
\partial: \FOp(\Sigma^{-1}\bar{\DOp})\rightarrow\FOp(\Sigma^{-1}\bar{\DOp})
\end{equation*}
is an operad derivation determined by the cooperad structure of~$\DOp$.
The internal differential of the cobar construction $\BOp^c(\DOp)$
is given by the addition of the derivation $\partial$
to the natural internal differential of the free operad $\FOp(\Sigma^{-1}\bar{\DOp})$.
The operadic cobar construction of cooperads is defined in the article~\cite{GetzlerJones}
in a characteristic zero setting.
The application of the cobar construction
to cooperads defined over a ring
is studied in the papers~\cite{FressePartitions,FresseCylinder}
to which we refer in our next recollections.

The goal of this subsection is to study the cobar construction of cooperads $\DOp$
whose homology $H_*(\DOp)$
forms a binary Koszul cooperad in the usual sense (we say that $\DOp$ is homologically Koszul).
The natural idea is to apply a spectral sequence $E^r\Rightarrow H_*(\BOp^c(\DOp))$
such that $E^1 = \BOp^c(H_*(\DOp))$.

For the moment,
we assume that $\DOp$ is any (connected) cooperad.
In the next subsection,
we apply our analysis to $E_n$-operads $\EOp_n$
and their dual cooperads $\DOp = \Lambda^{-n}\EOp_n^{\vee}$.

To begin with,
we review briefly the definition of the free operad $\FOp(M)$ associated to a $\Sigma_*$-object $M$ (in~\S\ref{Prelude:KoszulDuality:FreeOperad})
and of the cobar construction $\BOp^c(\DOp) = (\FOp(\Sigma^{-1}\bar{\DOp}),\partial)$ (in~\S\ref{Prelude:KoszulDuality:CobarConstruction}).
For a more detailed account of these definitions,
we refer to~\cite[\S 1.2]{FresseCylinder} (for the free operad)
and~\cite[\S 3.6]{FresseCylinder} (for the cobar construction).
In~\S\ref{Interlude:GraphOperadModel:FreeKOperad},
we extend the construction
of the free operad to operads equipped with a particular cell structure.

\subsubsection{The free operad}\label{Prelude:KoszulDuality:FreeOperad}
The goal of this paragraph is to review the definition of the free operad $\FOp(M)$.
In the sequel,
we only deal with free operads $\FOp(M)$ associated to $\Sigma_*$-objects $M$ such that $M(0) = M(1) = 0$.
This condition implies that $\FOp(M)$ is connected.
Therefore,
in the course of these recollections,
we also study the implications of the assumption $M(0) = M(1) = 0$
on the structure of the free operad $\FOp(M)$.

First,
we form the category $\Theta(r)$ of abstract oriented trees $\tau$
with $1$ outgoing edge and $r$ ingoing edges
indexed by $1,\dots,r$ (see the representations of Figure~\ref{fig:TreeComposite}).
The morphisms of $\Theta(r)$
are the isomorphisms of oriented trees $f: \sigma\rightarrow\tau$
preserving the indexing of ingoing edges.
The sequence of categories $\Theta(r)$, $r\in\NN$,
inherits an operad structure:
\begin{itemize}
\item
the permutations $w\in\Sigma_r$
act on $\Theta(r)$
by reindexing the ingoing edges of trees;
\item
the operadic composite $\sigma\circ_e\tau\in\Theta(s+t-1)$
of a tree $\tau\in\Theta(t)$ at the $e$th input of $\sigma\in\Theta(s)$
is obtained by plugging the outgoing edge of $\tau$
in the $e$th ingoing edge of $\sigma$ (see Figure~\ref{fig:TreeComposite}).
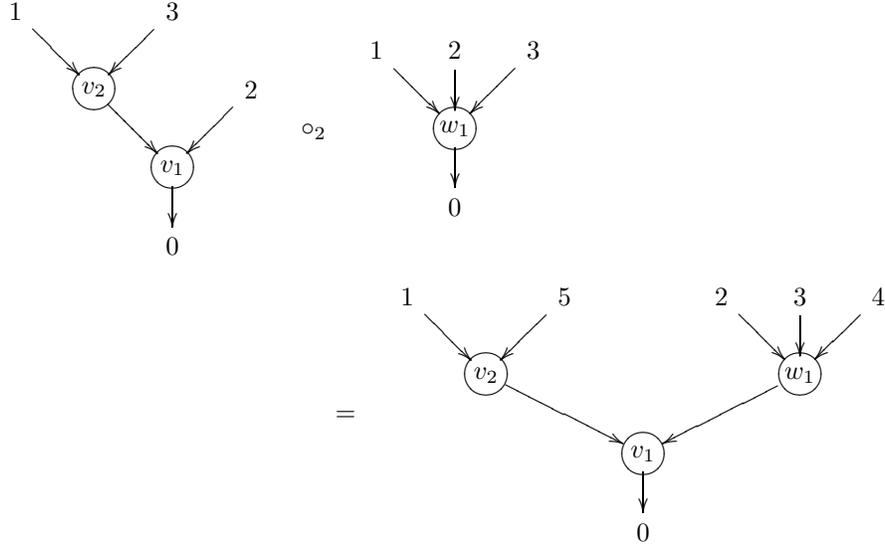
\begin{figure}[t]
\begin{multline*}
\vcenter{\xymatrix@M=0pt@!R=2mm@!C=2mm{ *+<3mm>{1}\ar[dr] && *+<3mm>{3}\ar[dl] & \\
& *+<3mm>[o][F]{v_2}\ar[dr] && *+<3mm>{2}\ar[dl] \\
&& *+<3mm>[o][F]{v_1}\ar[d] & \\
&& *+<3mm>{0} & }}
\quad\circ_2
\quad\vcenter{\xymatrix@M=0pt@!R=2mm@!C=2mm{ *+<3mm>{1}\ar[dr] & *+<3mm>{2}\ar[d] & *+<3mm>{3}\ar[dl] \\
& *+<3mm>[o][F]{w_1}\ar[d] & \\
& *+<3mm>{0} & }} \\
= \quad\vcenter{\xymatrix@M=0pt@!R=2mm@!C=2mm{ *+<3mm>{1}\ar[dr] && *+<3mm>{5}\ar[dl] && *+<3mm>{2}\ar[dr] & *+<3mm>{3}\ar[d] & *+<3mm>{4}\ar[dl] \\
& *+<3mm>[o][F]{v_2}\ar[drr] &&&& *+<3mm>[o][F]{w_1}\ar[dll] & \\
&&& *+<3mm>[o][F]{v_1}\ar[d] &&& \\
&&& *+<3mm>{0} &&& }}
\end{multline*}
\caption{The composition of trees. To produce a tree with $s+t-1$ inputs, we shift
the inputs of $\tau$ by $j\mapsto e+j-1$
and the inputs $i>e$ of $\sigma$ by $i\mapsto i+t-1$.}\label{fig:TreeComposite}\end{figure}
\end{itemize}

Denote the vertex set of a tree $\tau$
by $V(\tau)$.
To each tree $\tau$,
we associate the dg-module $\tau(M)$ of tensors $\otimes_{v\in V(\tau)} x_v$,
where each factor $x_v$, associated to a vertex $v\in V(\tau)$,
is an element of $M(r_v)$
together with a bijection between $\{1,\dots,r_v\}$
and the edges ending at $v$.
The usual convention is to represent the tensor $\otimes_{v\in V(\tau)} x_v$
by a labeling of the vertices of $\tau$,
with the elements $x_v$ as vertex labels.
For instance, in the example of Figure~\ref{fig:TreeComposite},
we label $v_1$ by an element $x_1\in M(2)$,
the vertex $v_2$ by an element $x_2\in M(2)$
and the vertex $w_1$ by an element $x_3\in M(3)$.
We have $x_1,x_2\in M(2)$, $x_3\in M(3)$ and each of these elements $x_k\in M(r_k)$
is given together with a bijection between $\{1,\dots,r_k\}$
and the ingoing edges of the associated vertex.

If $M(0) = M(1) = 0$, then we only have to deal with trees $\tau$, called reduced,
such that $r_v\geq 2$ for every vertex $v\in V(\tau)$,
because the dg-module $\tau(M)$
vanishes when this condition is not satisfied.
Reduced trees have no automorphism
outside the identity (see~\cite[\S 3.6.1]{FressePartitions}).

The dg-modules $\tau(M)$
form a functor in $\tau$
and the free operad $\FOp(M)$
is defined by the direct sums $\FOp(M)(r) = \bigoplus_{\tau\in\Theta(r)}\tau(M)/\equiv$
divided out by the action of tree isomorphisms.
If $M(0) = M(1) = 0$,
then we can restrict the expansion of $\FOp(M)(r)$
to the subcategory of reduced trees.
The triviality of the automorphism groups of reduced trees
implies that no quotient occurs in $\FOp(M)(r)$
when we pick representative of isomorphism classes of trees
to rewrite the expansion of $\FOp(M)(r)$
in a reduced form (see also~\cite[\S 3.6.1]{FressePartitions}).

The dg-module $\FOp(M)(r)$
inherits an action of the symmetric group $\Sigma_r$
from the category of trees $\Theta(r)$.
The operadic composites $\circ_e: \FOp(M)(s)\otimes\FOp(M)(t)\rightarrow\FOp(M)(s+t-1)$
are induced by natural isomorphisms $\sigma(M)\otimes\tau(M)\simeq(\sigma\circ_e\tau)(M)$
associated to all $\sigma\in\Theta(s)$, $\tau\in\Theta(t)$.
The operad unit $1\in\FOp(M)(1)$
is associated to the trivial tree $\downarrow\in\Theta(1)$
with one ingoing-outgoing edge and no vertex.
This gives the composition structure of $\FOp(M)$.

We have a natural isomorphism $M(r)\simeq\psi(M)$
associated to the tree $\psi$
with $1$ single vertex and $r$ ingoing edges (see Figure~\ref{fig:Corolla}).
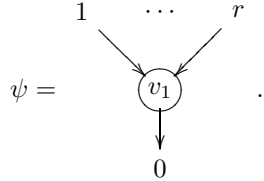
\begin{figure}[t]
\begin{equation*}
\psi = \vcenter{\xymatrix@M=0pt@!R=2mm@!C=2mm{ *+<3mm>{1}\ar[dr] & \cdots & *+<3mm>{r}\ar[dl] \\
& *+<3mm>[o][F]{v_1}\ar[d] & \\ & *+<3mm>{0} & }}.
\end{equation*}
\caption{The corolla.}\label{fig:Corolla}
\end{figure}
The inclusion of the summand $\psi(M)$ in $\FOp(M)(r) = \bigoplus_{\tau\in\Theta(r)}\tau(M)/\equiv$
yields a morphism of $\Sigma_*$-objects $\eta: M\rightarrow\FOp(M)$
and defines the universal morphism of the free operad.

The free operad has a natural weight decomposition $\FOp(M) = \bigoplus_{m=0}^{\infty}\FOp_{m}(M)$,
where $\FOp_{m}(M)$
is the dg-module spanned by tensors $\otimes_{v\in V(\tau)} x_v$
of order $|V(\tau)| = m$.
We have $\FOp_{0}(M) = I$, $\FOp_{1}(M) = M$
and the universal morphism of the free operad $\eta: M\rightarrow\FOp(M)$
is represented by the identity of $M$ with the summand $\FOp_{1}(M)\subset\FOp(M)$.

\subsubsection{Applications of the universal property of free operads}\label{Prelude:KoszulDuality:FreeOperadMonad}
For a connected operad $\POp$,
the universal property implies the existence of a morphism $\lambda_*: \FOp(\bar{\POp})\rightarrow\POp$
which extends the embedding $i: \bar{\POp}\hookrightarrow\POp$
of the augmentation ideal of $\POp$.
Intuitively,
the tensors of $\FOp(\bar{\POp})$
represent formal composites of operations $p_v\in\POp$
and the morphism $\lambda_*: \FOp(\bar{\POp})\rightarrow\POp$
is the operation that performs these composites.

Let $\lambda_2: \FOp_2(\bar{\POp})\rightarrow\POp$
be the restriction of this morphism to the quadratic summand $\FOp_2(\bar{\POp})$.
The components of $\lambda_2$,
defined on each summand $\tau(\bar{\POp})$ associated to a tree with two vertices $\tau$,
are identified with composition operations
\begin{equation*}
\circ_e: \POp(s)\otimes\POp(t)\rightarrow\POp(s+t-1)
\end{equation*}
together with an input permutation (see Figure~\ref{fig:QuadraticComposite}).
\begin{figure}[t]
\begin{equation*}
\vcenter{\xymatrix@M=0pt@!R=2mm@!C=2mm{ && *+<3mm>{j_1}\ar[dr] & \cdots & *+<3mm>{j_t}\ar[dl] & \\
*+<3mm>{i_1}\ar[drr] & \cdots & \cdots & *+<3mm>[F]{q}\ar[dl]|*+<2pt>{e} & \cdots & *+<3mm>{i_s}\ar[dlll] \\
&& *+<3mm>[F]{p}\ar[d] &&& \\
&& *+<3mm>{0} &&& }}
\mapsto\vcenter{\xymatrix@M=0pt@!R=2mm@!C=2mm{ *+<3mm>{i_1}\ar[drrr] & \cdots & *+<3mm>{j_1}\ar[dr] & \cdots &
*+<3mm>{j_t}\ar[dl] & \cdots & *+<3mm>{i_s}\ar[dlll] \\
&&& *+<3mm>[F]{p\circ_e q}\ar[d] &&& \\
&&& *+<3mm>{0} &&& }}
\end{equation*}
\caption{The graphical representation of quadratic composites.
The standard partial composite $p\circ_e q$
corresponds to the input sharing $(i_1,\dots,i_{e-1},j_1,\dots,j_t,i_{e+1},\dots,i_s)
= (1,\dots,e-1,e,\dots,e+t-1,e+t,\dots,s+t-1)$.}\label{fig:QuadraticComposite}
\end{figure}
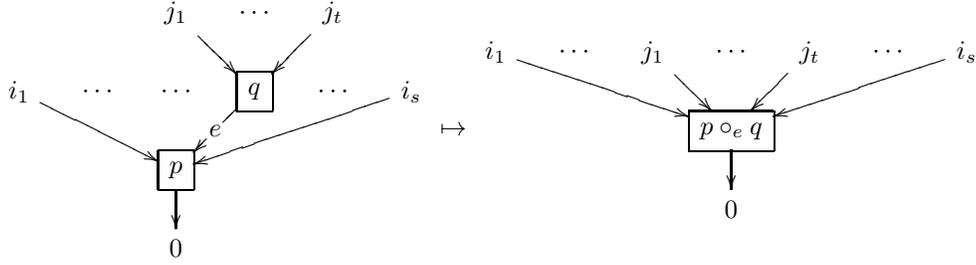
Recall that these operations suffice to specify the composition structure of an operad (we have already used this observation
in the definition of the composition structure of trees).

\subsubsection{Cooperads and the cobar construction}\label{Prelude:KoszulDuality:CobarConstruction}
The $\Sigma_*$-object underlying the free operad $\FOp(M)$
can also be equipped with a cooperad structure
and represents the cofree cooperad associated to $M$.
For detailed explanations on this observation (and other definitions of this paragraph),
we refer to~\cite[\S\S 3.1-3.5]{FresseCylinder}.

Just note that the composition structure of a cooperad $\DOp$
can be defined by a morphism $\rho_*: \DOp\rightarrow\FOp(\bar{\DOp})$.
The projection of this morphism onto the summand $\FOp_{2}(\bar{\DOp})$
gives a morphism $\rho_2: \bar{\DOp}\rightarrow\FOp_{2}(\bar{\DOp})$,
dual to the quadratic composition morphism of an augmented operad,
which also suffices to determine the composition structure of $\DOp$.

The dual dg-modules of an operad $\POp^{\vee}(r) = \Hom_{\C}(\POp(r),\ZZ)$
inherit such a coproduct
$\rho_2: \POp^{\vee}\rightarrow\FOp_{2}(\bar{\POp}^{\vee})$
when each dg-module $\POp(r)$ forms a degreewise finitely generated free $\ZZ$-module,
because we have a natural duality isomorphism $\FOp_{2}(\bar{\POp}^{\vee})\simeq\FOp_{2}(\bar{\POp})^{\vee}$
(see~\cite[Proposition 3.1.4 and \S 3.6.1]{FressePartitions}).
Hence,
we obtain that the $\ZZ$-dual dg-modules of an operad $\POp$ inherit a cooperad structure
under the usual assumption of duality in module categories.

The cobar construction
\begin{equation*}
\BOp^c(\DOp) = (\FOp(\Sigma^{-1}\bar{\DOp}),\partial)
\end{equation*}
is defined by the addition of an operadic derivation of degree $-1$
\begin{equation*}
\partial: \FOp(\Sigma^{-1}\bar{\DOp})\rightarrow\FOp(\Sigma^{-1}\bar{\DOp})
\end{equation*}
to the natural differential of the free operad $\FOp(\Sigma^{-1}\bar{\DOp})$.
The operadic derivation relation, which reads $\partial(p\circ_e q) = \partial(p)\circ_e q + \pm p\circ_e\partial(q)$,
implies that any derivation of a free operad $\partial: \FOp(M)\rightarrow\FOp(M)$
is determined by its restriction to the generating $\Sigma_*$-object $M\subset\FOp(M)$.
The bar differential is defined on generators
\begin{equation*}
\Sigma^{-1}\bar{\DOp}\xrightarrow{\partial|_{\Sigma^{-1}\bar{\DOp}}}\FOp(\Sigma^{-1}\bar{\DOp})
\end{equation*}
by a desuspension of the quadratic coproduct $\rho_2: \bar{\DOp}\rightarrow\FOp_{2}(\bar{\DOp})$
(see for instance~\cite[\S 3.6]{FresseCylinder}).

\subsubsection{The filtration of the cobar construction}\label{Prelude:KoszulDuality:CobarFiltration}
We equip the free operad $\FOp(M)$ with the filtration such that $F_{-s}\FOp(M) = \bigoplus_{r\geq s}\FOp_{r}(M)$.
The composition product of the free operad is homogeneous with respect to this filtration
and the associated graded object $E^0_{-s}\FOp(M) = F_{-s}\FOp(M)/F_{-s-1}\FOp(M)$
satisfies obviously $E^0_{-s}\FOp(M) = F_{-s}\FOp(M)$.
Moreover,
we have clearly $F_{-s}\FOp(M) = \FOp(M)$ for $s\leq 0$.

For the free operad $\FOp(\Sigma^{-1}\bar{\DOp})$
associated to a connected cooperad $\DOp$,
we have
\begin{equation*}
F_{-s}\FOp(\Sigma^{-1}\bar{\DOp})(n) = 0\quad\text{when $s>n-1$},
\end{equation*}
for any fixed arity $n\geq 1$.
Moreover,
the twisting derivation of the cobar construction
satisfies $\partial(F_{-s}\FOp(\Sigma^{-1}\bar{\DOp}))\subset F_{-s-1}\FOp(\Sigma^{-1}\bar{\DOp})$.

Thus,
the cobar construction inherits a filtration
\begin{equation*}
\BOp^c(\DOp) = F_0\BOp^c(\DOp)\supset F_{-1}\BOp^c(\DOp)\supset\cdots\supset F_{-s}\BOp^c(\DOp)\supset\cdots
\end{equation*}
so that:

\begin{fact}[{see~\cite[Lemma 3.6.2]{FressePartitions}}]\label{Prelude:KoszulDuality:CobarSpectralSequence}
We have a strongly convergent spectral sequence of operads
\begin{equation*}
E^r(\BOp^c(\DOp))\Rightarrow H_*(\BOp^c(\DOp)),
\end{equation*}
naturally associated to the cobar construction $\BOp^c(\DOp)$,
so that $E^1 = H_*(\FOp(\Sigma^{-1}\bar{\DOp}))$
is the homology of the free operad on $\Sigma^{-1}\bar{\DOp}$.
\end{fact}

Moreover:

\begin{fact}[{see~\cite[\S\S 3.6.1-2]{FressePartitions}}]\label{Prelude:KoszulDuality:CobarHomology}
If each component $\DOp(r)$ of a cooperad $\DOp$ forms a cofibrant dg-module
and each component $H_*(\DOp(r))$ of the homology of $\DOp$ forms a free graded $\ZZ$-module,
then we have a K\"unneth isomorphism
\begin{equation*}
\FOp(\Sigma^{-1}H_*(\bar{\DOp}))\xrightarrow{\simeq} H_*(\FOp(\Sigma^{-1}\bar{\DOp}))
\end{equation*}
and the $E^1$-term of the spectral sequence can be identified with the cobar construction of the homology of~$\DOp$:
\begin{equation*}
(E^1,d^1) = (\FOp(\Sigma^{-1}H_*(\bar{\DOp})),\partial) = \BOp^c(H_*(\DOp)).
\end{equation*}
\end{fact}

\subsubsection{Homologically Koszul cooperads}\label{Prelude:KoszulDuality:HomologyKoszulCooperads}
The spectral sequence of Fact~\ref{Prelude:KoszulDuality:CobarSpectralSequence}
can be applied to a graded cooperad $\HOp$
equipped with a trivial internal differential $\delta = 0$.
In this case,
the spectral sequence degenerates at $E^2$
and reduces simply to the definition of a natural weight grading on $H_*(\BOp^c(\HOp))$.

We say that a graded cooperad $\HOp$ is Koszul
if we have
\begin{equation*}
E^0_{-s} H_*(\BOp^c(\HOp))(n) = \begin{cases} H_*(\BOp^c(\HOp))(n), & \text{if $s = n-1$}, \\ 0, & \text{otherwise}, \end{cases}
\end{equation*}
for each fixed arity $n\in\NN$,
or, in plain words,
if the weight grading of $H_*(\BOp^c(\HOp))$
collapses onto the arity grading.
(In fact,
we just apply of the general definition of~\cite{FressePartitions} to the cooperad $\HOp$
equipped with the natural arity grading.)
Since we deal with cooperads in $\ZZ$-modules,
we also assume that the components of a Koszul cooperad $\HOp$
are projective as $\ZZ$-modules.

For any graded cooperad $\HOp$,
we have $E^0_{-s} \BOp^c(\HOp)(n) = 0$ for $s>n-1$
since we observe in~\S\ref{Prelude:KoszulDuality:CobarFiltration}
that the free operad filtration vanishes for $s>n-1$.
Hence,
we have an edge morphism
\begin{equation*}
\BOp^c(\HOp)(n)\rightarrow E^0_{1-n}\BOp^c(\HOp)(n)\rightarrow E^0_{1-n} H_*(\BOp^c(\HOp))(n)
\end{equation*}
determined by the spectral sequence filtration
for all $n\in\NN$.
For a Koszul cooperad,
these edge morphisms assemble to a morphism of operads
\begin{equation*}
\edge: \BOp^c(\HOp)\rightarrow H_*(\BOp^c(\HOp)).
\end{equation*}
This morphism is determined on generators of the quasi-free operad $\BOp^c(\HOp) = (\FOp(\Sigma^{-1}\bar{\HOp}),\partial)$
by the natural embedding $\Sigma^{-1}\bar{\HOp}(n)\subset\FOp(\Sigma^{-1}\bar{\HOp})(n)$
in arity $n=2$
and by the null morphism in arity $n\not=2$.
This identity follows from a straightforward inspection of the definition of the filtration.
Naturally,
the morphism $\edge$ defines a weak-equivalence of dg-operads between $\BOp^c(\HOp)$
and the homology operad $H_*(\BOp^c(\HOp))$,
viewed as a dg-operad equipped with a trivial differential.

We say that a dg-cooperad $\DOp$
is homologically Koszul if its homology cooperad $\HOp = H_*(\DOp)$
is Koszul.
The $\ZZ$-modules $H_*(\DOp)(n)$
are implicitely assumed to be projective when we say that $H_*(\DOp)$
is a Koszul cooperad.
For our purpose,
we also assume that the components $\DOp(n)$ of a homologically Koszul cooperad $\DOp$
are cofibrant dg-modules.

\begin{prop}\label{Prelude:KoszulDuality:HomologyKoszulSpectralSequence}
If a cooperad $\DOp$ is homologically Koszul,
then the spectral sequence of Fact~\ref{Prelude:KoszulDuality:CobarSpectralSequence}
degenerates at $E^2$
and we have a natural weak-equivalence of dg-operads
\begin{equation*}
\edge: \BOp^c(H_*(\DOp))\xrightarrow{\sim} H_*(\BOp^c(\DOp))
\end{equation*}
determined on generators of the quasi-free operad $\BOp^c(H_*(\DOp)) = (\FOp(\Sigma^{-1}H_*(\bar{\DOp})),\partial)$
by the homology of the natural embedding $\Sigma^{-1}\bar{\DOp}(n)\subset\FOp(\Sigma^{-1}\bar{\DOp})(n)$
in arity $n=2$
and by the null morphism in arity $n\not=2$.
\end{prop}

\begin{proof}
To simplify notation, we set $\HOp = H_*(\DOp)$.
The spectral sequence $E^r\Rightarrow H_*(\BOp^c(\DOp))$ satisfies $E^1 = \BOp^c(\HOp)$
by Fact~\ref{Prelude:KoszulDuality:CobarHomology}
and $E^2 = E^0 H_*(\BOp^c(\HOp))$.
If the cooperad $\DOp$ is homology Koszul, then $E^0_s H_*(\BOp^c(\HOp))(n)$
is concentrated on the column $s = n-1$,
for each fixed arity $n$.
Hence, we obtain that the spectral sequence degenerates at $E^2$,
giving an isomorphism $E^0 H_*(\BOp^c(\HOp))\simeq E^0 H_*(\BOp^c(\DOp))$,
from which we deduce the existence of a genuine operad isomorphism $H_*(\BOp^c(\HOp))\simeq H_*(\BOp^c(\DOp))$
since, for Koszul operads, the spectral sequence grading reduces to the arity grading.

The homomorphism specified in the proposition is identified with the composite
\begin{equation*}
\BOp^c(H_*(\DOp)) = \BOp^c(\HOp)\xrightarrow{\edge} H_*(\BOp^c(\HOp))\xrightarrow{\simeq} H_*(\BOp^c(\DOp)),
\end{equation*}
where the first arrow is the edge morphism
of~\S\ref{Prelude:KoszulDuality:HomologyKoszulSpectralSequence}.
Hence the definition of the proposition gives a well-defined morphism of dg-operads
$\edge: \BOp^c(H_*(\DOp))\rightarrow H_*(\BOp^c(\DOp))$
which forms obviously a weak-equivalence.
\end{proof}

The morphism of the proposition
$\edge: \BOp^c(H_*(\DOp))\xrightarrow{\sim} H_*(\BOp^c(\DOp))$
can also be identified with the edge morphisms
\begin{equation*}
\BOp^c(H_*(\DOp))(n)\rightarrow E^1_{1-n}(\BOp^c(\DOp))(n)\rightarrow E^0_{1-n} H_*(\BOp^c(\DOp))(n)
\end{equation*}
yielded by the spectral sequence of $\BOp^c(\DOp)$
together with the relation $E^0_{1-n} H_*(\BOp^c(\DOp))(n) = H_*(\BOp^c(\DOp))(n)$
for a homology Koszul cooperad.
For that reason,
we refer to this morphism as the edge morphism of the homologically Koszul cooperad $\DOp$.

\subsubsection{Recollections: Koszul duality of operads}\label{Prelude:KoszulDuality:KoszulDualDefinition}
There are several definitions for the notion of a Koszul operad.
According to one of them (see~\cite[\S 5.2.8]{FressePartitions}),
a graded operad~$\POp$ (equipped with a trivial differential)
is Koszul
if we have a weak-equivalence $\epsilon: \BOp^c(\KOp(\POp))\xrightarrow{\sim}\POp$,
where $\KOp(\POp)$ is a certain graded cooperad (equipped with a trivial differential)
naturally associated to $\POp$,
the Koszul dual cooperad of $\POp$.
In the sequel,
we refer to this weak-equivalence $\epsilon$
as the Koszul duality equivalence
of the operad~$\POp$.

In applications,
we assume tacitely, as usual,
that the components of Koszul operad $\POp$
are projective $\ZZ$-modules.
According to an observation of~\cite[\S 5.2.8]{FressePartitions},
the components of the Koszul dual $\KOp(\POp)$ of a Koszul operad $\POp$
are automatically projective $\ZZ$-modules.

The Koszul dual $\KOp(\POp)$ of a Koszul operad $\POp$
forms a Koszul cooperad
(see~\cite[Lemma 5.2.10]{FressePartitions}).
If $\POp$ is a binary Koszul operad (see~\cite[\S 5.2]{FressePartitions}),
then the dual Koszul cooperad $\KOp(\POp)$
is Koszul in the sense of~\S\ref{Prelude:KoszulDuality:HomologyKoszulCooperads}.

\subsection{Applications to Gerstenhaber operads}\label{Prelude:GerstenhaberOperads}
In this subsection,
we review the structure of the homology operad $H_*(\EOp_n)$.
The articles~\cite{Cohen,Sinha}
are our references for the computation of $H_*(\EOp_n)$.
For a nice inroduction to this topic, we also refer to~\cite[\S I.6]{KrizMay}.

For $n=1$,
we have an identity $H_*(\EOp_1) = \AOp$
with the operad of associative algebras $\AOp$.
For $n>1$,
the operad $H_*(\EOp_n)$ is identified with another operad $\GOp_n$ defined by generators and relations,
to which we refer as the $n$-Gerstenhaber operad.

The associative operad and the Gerstenhaber operads
are Koszul
and we review the definition of the Koszul duality equivalence $\epsilon: \BOp^c(\KOp(\POp))\xrightarrow{\sim}\POp$
associated to these operads.

Recall that the associative operad $\AOp$
can be defined as the operad generated by an operation $\mu = \mu(x_1,x_2)$ in arity $2$
together with the associativity relation $\mu(\mu(x_1,x_2),x_3) = \mu(x_1,\mu(x_2,x_3))$ in arity $3$.
To begin with,
we recall the similar definition of the Gerstenhaber operads by generators and relations
and we review the definition of the isomorphism $\gamma: \GOp_n\xrightarrow{\simeq} H_*(\EOp_n)$
when $\EOp_n$ is the $n$th layer of the Barratt-Eccles operad $\EOp$.

\subsubsection{The Gerstenhaber operads}\label{Prelude:GerstenhaberOperads:Definition}
The generating operations of $\GOp_n$ consist of an operation $\mu = \mu(x_1,x_2)$, of degree $0$,
and an operation $\lambda_{n-1} = \lambda_{n-1}(x_1,x_2)$, of degree $n-1$,
together with the symmetry relations
\begin{equation}\label{eqn:OperationSymmetry}
\mu(x_1,x_2) = \mu(x_2,x_1)\quad\text{and}\quad\lambda_{n-1}(x_1,x_2) = (-1)^{n}\lambda_{n-1}(x_2,x_1).
\end{equation}
For short, we set $\lambda = \lambda_{n-1}$
in this paragraph.
Let $\ZZ\mu\oplus\ZZ\lambda$ represent the $\Sigma_*$-object spanned by the elements $(\mu,\lambda)$
in arity $2$
together with the action of the symmetric group $\Sigma_2$
determined by (\ref{eqn:OperationSymmetry}).
The $n$-Gerstenhaber operad $\GOp_n$ is the quotient of the free operad $\FOp(\ZZ\mu\oplus\ZZ\lambda)$
by the operadic ideal
generated by the associativity relation
\begin{equation}\label{eqn:Associativity}
\mu(\mu(x_1,x_2),x_3)\equiv\mu(x_1,\mu(x_2,x_3)),
\end{equation}
the Jacobi relation
\begin{equation}\label{eqn:Jacobi}
\lambda(\lambda(x_1,x_2),x_3) + \lambda(\lambda(x_2,x_3),x_1) + \lambda(\lambda(x_3,x_1),x_2)\equiv 0, \\
\end{equation}
and the distribution relation
\begin{equation}\label{eqn:Distribution}
\lambda(\mu(x_1,x_2),x_3)\equiv\mu(\lambda(x_1,x_3),x_2) + \mu(x_1,\lambda(x_2,x_3)).
\end{equation}

According to this definition,
a morphism $\phi: \GOp_n\rightarrow\POp$
toward an operad $\POp$
is fully determined by elements $\mu,\lambda\in\POp(2)$ that satisfy the symmetry relations (\ref{eqn:OperationSymmetry})
and relations (\ref{eqn:Associativity}-\ref{eqn:Distribution}) in $\POp$.
The goal of the next paragraph
is to define such elements in the homology of the $n$th filtration layer of the Barratt-Eccles operad.

The Gerstenhaber operad $\GOp_n$ can also be defined as a composite $\GOp_n = \COp\circ\Lambda^{1-n}\LOp$,
where we consider the commutative operad $\COp$
and the $(n-1)$-fold desuspension of the Lie operad $\LOp$.
We only give a brief review of this definition of $\GOp_n$.
We refer to~\cite{FoxMarkl,Markl} for more explanations.

The commutative operad $\COp$ is identified with the suboperad of $\GOp_n$
generated by the operation $\mu\in\GOp_n(2)$.
The $(n-1)$-fold desuspension of the Lie operad $\Lambda^{1-n}\LOp$
is identified with the suboperad of $\GOp_n$
generated by $\lambda_{n-1}\in\GOp_n(2)$.
The distribution relation between the product $\mu$
and the Lie bracket $\lambda_{n-1}$
determines an interchange morphism $\theta: \Lambda^{1-n}\LOp\circ\COp\rightarrow\COp\circ\Lambda^{1-n}\LOp$.
The composition product of the Gerstenhaber operad,
defined abstractly by a morphism of $\Sigma_*$-objects $\psi: \GOp_n\circ\GOp_n\rightarrow\GOp_n$,
is given by the composite of this interchange morphism
with the composition product of the commutative and Lie operads:
\begin{equation*}
(\COp\circ\Lambda^{1-n}\LOp)\circ(\COp\circ\Lambda^{1-n}\LOp)
\xrightarrow{\COp\circ\theta\circ\Lambda^{1-n}\LOp}
(\COp\circ\COp)\circ(\Lambda^{1-n}\LOp\circ\Lambda^{1-n}\LOp)
\xrightarrow{\psi\circ\psi}\COp\circ\Lambda^{1-n}\LOp.
\end{equation*}

\subsubsection{The representative of the product and of the Browder bracket in the Barratt-Eccles operad}\label{Prelude:GerstenhaberOperads:RepresentativeOperations}
In arity $2$,
the Barratt-Eccles operad $\EOp$
is spanned by the alternate simplices
\begin{equation*}
\mu_d = (\underbrace{\id,\tau,\id,\tau,\dots}_{\in\Sigma_2^{\times d+1}})
\quad\text{and}\quad\tau\mu_d = (\underbrace{\tau,\id,\tau,\id,\dots}_{\in\Sigma_2^{\times d+1}}),
\quad d\in\NN,
\end{equation*}
where $\id$ is the identity permutation of $(1,2)$ and $\tau$ is the transposition $\tau = (2,1)$.
Moreover we have $\delta(\mu_d) = \tau\mu_{d-1} + (-1)^d\mu_{d-1}$.
Hence the dg-module $\EOp(2)$
can be identified with the usual free resolution of the trivial $\Sigma_2$-module:
\begin{equation*}
\ZZ\mu_0\oplus\ZZ\tau\mu_0\xleftarrow{\tau-1}\ZZ\mu_1\oplus\ZZ\tau\mu_1
\xleftarrow{\tau+1}\ZZ\mu_2\oplus\ZZ\tau\mu_2\xleftarrow{\tau-1}\cdots.
\end{equation*}

According to the definition of~\S\ref{Prelude:BarrattEccles:LittleCubesFiltration},
we have $\mu_d\in\EOp_n(2)$ if and only if $d<n$.
Hence,
the dg-module $\EOp_n(2)$ is identified with a truncation of $\EOp(2)$
and we have
\begin{equation*}
H_d(\EOp_n(2)) = \begin{cases} \ZZ, & \text{if $d=0,n-1$}, \\ 0, & \text{otherwise}. \end{cases}
\end{equation*}
Thus we retrieve the standard description of the homology of the space $\COp_n(2)$.

The cycles $\mu = \mu_0$ and $\lambda_{n-1} = \mu_{n-1}+(-1)^{n-1}\mu_{n-1}$
define generating homology classes of $H_*(\EOp_n(2))$.

The operation $\mu = \mu_0$, which belongs to $\EOp_1(2)$,
satisfies clearly the relation of an associative product in $\EOp_1$ (recall that $\EOp_1 = \AOp$).
For $n>1$,
the homology class associated to $\mu_0$ in $H_*(\EOp_n)$
satisfies $\tau\mu_0\equiv\mu_0$
since $\tau\mu_0-\mu_0$ is the boundary of $\mu_1$.
Hence we obtain that $\mu = \mu_0$ represents an associative and commutative product in $H_*(\EOp_n)$,
for every $n>1$.

Recall that the homology of the topological little $n$-cubes operad $\COp_n$
is an operad in graded $\ZZ$-modules
which acts on the homology of $n$-fold loop spaces.
The element of $H_0(\COp_n(2),\ZZ)$ which corresponds to $\mu\in H_0(\EOp_n(2))$
represents in fact the Pontrjagyn product,
the multiplication on the homology of loop spaces induced by the composition of loops.
The generating homology class of $H_{n-1}(\COp_n(2),\ZZ)$ which corresponds to $\lambda_{n-1}\in H_{n-1}(\EOp_n(2))$
represents another standard operation on the homology of iterated loop spaces
called the Browder bracket.

The next assertion follows from a standard result about the representatives of these operations
in the homology of little cubes operad (we refer to~\cite{Cohen,Sinha}):

\begin{fact}\label{Prelude:GerstenhaberOperads:RepresentativeIdentities}
The homology classes of $\mu\in\EOp_n(2)$ and $\lambda_{n-1}\in\EOp_n(2)$
satisfy the identities of the generating operations of the $n$-Gerstenhaber operad
in $H_*(\EOp_n)$.
\end{fact}

This assertion can also be checked by direct computations with the representative cycles of $\mu$ and $\lambda_{n-1}$
in $\EOp_n$.

Because of the definition of $\GOp_n$
by generators and relations,
Fact~\ref{Prelude:GerstenhaberOperads:RepresentativeIdentities}
implies the existence of an operad morphism $\gamma: \GOp_n\rightarrow H_*(\EOp_n)$,
for all $n>1$.
We have moreover:

\begin{thm}[F. Cohen, see~\cite{Cohen}]\label{Prelude:GerstenhaberOperads:LittleCubesHomology}
The morphism $\gamma: \GOp_n\rightarrow H_*(\EOp_n)$ which maps the generating operations of $\GOp_n$
to the homology classes $\mu\in H_0(\EOp_n(2))$ and $\lambda_{n-1}\in H_{n-1}(\EOp_n(2))$
is an isomorphism for each $n>1$.\qed
\end{thm}

We refer to the cited reference~\cite{Cohen}
for the original proof of this theorem.
We also refer to~\cite{Sinha}
where this result is revisited with an insightful account of the geometrical picture
underlying the computation of~$H_*(\EOp_n)$.

\medskip
According to this theorem,
any morphism toward an operad $\phi: H_*(\EOp_n)\rightarrow\POp$
is determined by its evaluation on $\mu\in H_0(\EOp_n(2))$ and $\lambda_{n-1}\in H_{n-1}(\EOp_n(2))$
since these operations generate $\GOp_n = H_*(\EOp_n)$.
For the embedding $\iota: \EOp_{n-1}\rightarrow\EOp_n$
and the suspension morphism $\sigma: \EOp_n\rightarrow\Lambda^{-1}\EOp_{n-1}$,
we obtain by a straightforward inspection:

\begin{prop}\label{Prelude:GerstenhaberOperads:LittleCubesMorphismHomology}\hspace*{2mm}
\begin{enumerate}
\item
The morphism $\iota_*: H_*(\EOp_{n-1})\rightarrow H_*(\EOp_{n})$
induced by the embedding $\iota: \EOp_{n-1}\hookrightarrow\EOp_{n}$
satisfies
\begin{equation*}
\iota_*(\mu) = \mu\quad\text{and}\quad\iota_*(\lambda_{n-1}) = 0,\quad\text{for each $n>1$}.
\end{equation*}
For $n=1$, we also have $\iota_*(\mu) = \mu$.
\item
The morphism $\sigma_*: H_*(\EOp_n)\rightarrow H_*(\Lambda^{-1}\EOp_{n-1})$
induced by the suspension morphism $\sigma: \EOp_n\rightarrow\Lambda^{-1}\EOp_{n-1}$
satisfies
\begin{equation*}
\sigma_*(\mu) = 0\quad\text{and}\quad\sigma_*(\lambda_{n-1}) = \lambda_{n-2},\quad\text{for each $n>2$}.
\end{equation*}
To simplify notation, we omit to mark the operadic desuspension in the formula of $\sigma_*(\lambda_{n-1})$.
Note however that this desuspension reverses the parity of the operation $\lambda_{n-2}$.

For $n=2$, we have the same formulas provided we define the degree~$0$ bracket $\lambda_0 = \lambda_0(x_1,x_2)$
by the commutator of the associative product $\mu = \mu(x_1,x_2)$
in $H_*(\EOp_1) = \AOp$.
\qed
\end{enumerate}
\end{prop}

We aim to study the Koszul dual of the operads $\EOp_n$.
At the homology level,
we have the following statement:

\begin{fact}\label{Prelude:GerstenhaberOperads:KoszulDualityStatement}\hspace*{2mm}
\begin{enumerate}
\item\label{KoszulDuality:Associative}
The associative operad $\AOp$ is Koszul with $\KOp(\AOp) = \Lambda^{-1}\AOp^{\vee}$.
Let $\mu^{\vee}\in\AOp^{\vee}(2)$
be the dual element of $\mu\in\AOp(2)$.
The Koszul duality equivalence $\epsilon: \BOp^c(\Lambda^{-1}\AOp^{\vee})\xrightarrow{\sim}\AOp$
is determined by
\begin{equation*}
\epsilon(\mu^{\vee}) = \mu.
\end{equation*}
\item\label{KoszulDuality:Gerstenhaber}
The Gerstenhaber operad $\GOp_n$ is Koszul with $\KOp(\GOp_n) = \Lambda^{-n}\GOp_n^{\vee}$
as a Koszul dual cooperad, for each $n>1$.
Let $\mu^{\vee},\lambda_{n-1}^{\vee}\in\GOp_n^{\vee}(2)$
be the dual basis of $\mu,\lambda_{n-1}\in\GOp_n(2)$.
The Koszul duality equivalence $\epsilon: \BOp^c(\Lambda^{-n}\GOp_n^{\vee})\xrightarrow{\sim}\GOp_n$
is determined by
\begin{equation*}
\epsilon(\mu^{\vee}) = \lambda_{n-1}\quad\text{and}\quad\epsilon(\lambda_{n-1}^{\vee}) = \mu.
\end{equation*}
\end{enumerate}
\end{fact}

Recall that the Koszul duality equivalences
vanish on generators of arity $r>2$ (see~\S\ref{Prelude:KoszulDuality:HomologyKoszulCooperads}).

Again we omit to mark operadic suspensions on elements though this operation modifies degrees
and symmetric group actions.

Assertion~(\ref{KoszulDuality:Associative})
about the associative operad is very classical (see~\cite{GinzburgKapranov})
and holds in the context of modules over a ring (see~\cite[\S 5.2]{FressePartitions}).
Assertion~(\ref{KoszulDuality:Gerstenhaber})
about the Gerstenhaber operad can be deduced from the argument of~\cite{Markl}
and from~\cite[Theorems 6.5-6.7]{FressePartitions}.
To see that the result holds in the context of modules over a ring,
we have to check (according to the definition of~\cite[\S 5.2.8]{FressePartitions})
that the underlying $\ZZ$-modules of the commutative operad are projective (obvious)
as well as the underlying $\ZZ$-modules of the Lie operad (see~\cite{Bourbaki,Reutenauer} or use the argument of~\cite[\S 5.2.8]{FressePartitions}).
An alternative and direct proof of Fact~\ref{Prelude:GerstenhaberOperads:KoszulDualityStatement}
follows from~\cite{Hoffbeck}
(adapt the examples of this reference to check that the associative operad and the Gerstenhaber operads
are all Poincar\'e-Birkhoff-Witt,
for any choice of ground ring).

The next proposition shows that the Koszul duality equivalences of Gerstenhaber operads
fit nicely together:

\begin{prop}\label{Prelude:GerstenhaberOperads:KoszulPatching}
The Koszul duality gives weak-equivalences
\begin{equation*}
\epsilon_n: \BOp^c(\Lambda^{-n}\GOp_n^{\vee})\xrightarrow{\sim}\GOp_n
\end{equation*}
so that the diagram
\begin{equation*}
\xymatrix{ \BOp^c(\Lambda^{-1}\AOp^{\vee})\ar[r]^(0.5){\sigma^*}\ar[d]^{\sim}_{\epsilon_1=\epsilon} &
\BOp^c(\Lambda^{-2}\GOp_2^{\vee})\ar[r]^(0.6){\sigma^*}\ar[d]^{\sim}_{\epsilon_2} &
\cdots\ar[r]^(0.4){\sigma^*} &
\BOp^c(\Lambda^{-n}\GOp_n^{\vee})\ar[r]^(0.65){\sigma^*}\ar[d]^{\sim}_{\epsilon_n} & \cdots \\
\AOp\ar[r]_{\iota_*} &
\GOp_2\ar[r]_{\iota_*} &
\cdots\ar[r]_{\iota_*} &
\GOp_n\ar[r]_{\iota_*} & \cdots }
\end{equation*}
commutes,
where:
\begin{itemize}
\item
the morphisms $\iota_*$ are induced by the operad embeddings $\iota: \EOp_{n-1}\rightarrow\EOp_n$,
\item
the morphisms $\sigma^*$
are defined by the application of the functor $\BOp^c(\Lambda^{-n} H_*(-)^{\vee}) = \BOp^c(H_*(\Lambda^n -)^{\vee})$
to the suspension morphisms $\sigma: \EOp_n\rightarrow\Lambda^{-1}\EOp_{n-1}$.
\end{itemize}
\end{prop}

\begin{proof}
Since the cobar construction is a quasi-free operad $\BOp^c(\DOp) = (\FOp(\Sigma^{-1}\bar{\DOp}),\partial)$,
it is sufficient to check the commutativity of each square on generating elements of $\BOp^c(\DOp)$.
Recall that the Koszul duality equivalences are supposed to vanish on generating elements of arity $r>2$.
Thus we only have to check identities
\begin{equation*}
\iota_*\epsilon_{n-1}(\gamma^{\vee}) = \epsilon_n\sigma^*(\gamma^{\vee})
\end{equation*}
when $\gamma^{\vee}$ is a basis element of $\Lambda^{-1}\AOp^{\vee}(2)$ for $n=2$,
respectively $\Lambda^{1-n}\GOp_{n-1}^{\vee}(2)$ for $n>2$.
This verification is immediate from the formulas of Proposition~\ref{Prelude:GerstenhaberOperads:LittleCubesMorphismHomology}
and Fact~\ref{Prelude:GerstenhaberOperads:KoszulDualityStatement}.
\end{proof}

\subsection{Statement of the main theorems}\label{Prelude:MainTheorems}
The upper and lower rows of the diagram of Proposition~\ref{Prelude:GerstenhaberOperads:KoszulPatching}
are defined at the chain level.
The main task of this paper is to prove
that the vertical morphisms
have a realization at the chain level as well:

\begin{mainthm}\label{Prelude:RealizationTheorem}
We have a sequence of morphisms
\begin{equation*}
\psi_n: \BOp^c(\Lambda^{-n}\EOp_n^{\vee})\rightarrow\EOp_n,\quad n\geq 1,
\end{equation*}
beginning with the Koszul duality equivalence of the associative operad for $n=1$,
and such that:
\begin{enumerate}
\item\label{Prelude:RealizationTheorem:Diagram}
the diagram
\begin{equation*}
\xymatrix{
\BOp^c(\Lambda^{-1}\EOp_1^{\vee})\ar[r]^(0.5){\sigma^*}\ar[d]^{\sim}_{\psi_1=\epsilon_1} &
\BOp^c(\Lambda^{-2}\EOp_2^{\vee})\ar[r]^(0.6){\sigma^*}\ar@{.>}[d]_{\psi_2} &
\cdots\ar[r]^(0.4){\sigma^*} &
\BOp^c(\Lambda^{-n}\EOp_n^{\vee})\ar[r]^(0.65){\sigma^*}\ar@{.>}[d]_{\psi_n} & \cdots \\
\EOp_1\ar@{^{(}->}[]!R+<4pt,0pt>;[r]_{\iota} &
\EOp_2\ar@{^{(}->}[]!R+<4pt,0pt>;[r]_{\iota} &
\cdots\ar@{^{(}->}[]!R+<4pt,0pt>;[r]_{\iota} &
\EOp_n\ar@{^{(}->}[]!R+<4pt,0pt>;[r]_{\iota} & \cdots },
\end{equation*}
commutes;
\item\label{Prelude:RealizationTheorem:Prescription}
the composite of morphism induced by $\psi_n$ in homology
with the edge morphism of Proposition~\ref{Prelude:KoszulDuality:HomologyKoszulSpectralSequence}
reduces to the Koszul duality equivalence
\begin{equation*}
\epsilon_n: \BOp^c(\Lambda^{-n}\GOp_n^{\vee})\xrightarrow{\sim}\GOp_n
\end{equation*}
of propositions~\ref{Prelude:GerstenhaberOperads:KoszulDualityStatement}-\ref{Prelude:GerstenhaberOperads:KoszulPatching}.
\end{enumerate}
\end{mainthm}

In this statement,
the notation $\sigma^*: \BOp^c(\Lambda^{1-n}\EOp_{n-1}^{\vee})\rightarrow\BOp^c(\Lambda^{-n}\EOp_n^{\vee})$
refers to the image of the suspension morphism $\sigma: \EOp_n\rightarrow\Lambda^{-1}\EOp_{n-1}$
under the functor $\BOp^c(\Lambda^{-n}(-)^{\vee}) = \BOp^c((\Lambda^n -)^{\vee})$.

To justify the second assertion of the theorem,
note that the cooperad~$\Lambda^{-n}\EOp_n^{\vee}$ is homologically Koszul,
because, according to Fact~\ref{Prelude:GerstenhaberOperads:KoszulDualityStatement},
the homology operads $\GOp_n = H_*(\EOp_n)$
are Koszul with dual $\Lambda^{-n}\GOp_n^{\vee} = H_*(\Lambda^{-n}\EOp_n^{\vee})$.

The construction implies automatically:

\begin{mainthm}\label{Prelude:BarDualityTheorem}
The morphisms of Theorem~\ref{Prelude:RealizationTheorem}
\begin{equation*}
\psi_n: \BOp^c(\Lambda^{-n}\EOp_n^{\vee})\rightarrow\EOp_n
\end{equation*}
are automatically weak-equivalences.
\end{mainthm}

\begin{proof}
According to the statement of Theorem~\ref{Prelude:RealizationTheorem},
the morphism induced by $\psi_n$
in homology
fits a commutative diagram
\begin{equation*}
\xymatrix{ \BOp^c(\Lambda^{-n}\GOp_n^{\vee})\ar[d]_{\epsilon_n}^{\sim}\ar[r]^(0.43){\edge}_(0.43){\sim} &
H_*(\BOp^c(\Lambda^{-n}\EOp_n^{\vee}))\ar@{.>}[d]^{\psi_n{}_*} \\
\GOp_n\ar[r]_(0.43){\simeq} & H_*(\EOp_n) },
\end{equation*}
where $\edge$ denotes the edge morphism of Fact~\ref{Prelude:KoszulDuality:HomologyKoszulSpectralSequence}.
The conclusion of the theorem is an immediate consequence
of the fact that $\epsilon_n$ and $\edge$
are both weak-equivalences.
\end{proof}

In~\S\ref{FirstStep:Conclusion},
we also check that each morphism $\sigma^*: \BOp^c(\Lambda^{1-n}\EOp_{n-1}^{\vee})\rightarrow\BOp^c(\Lambda^{-n}\EOp_{n}^{\vee})$
is a cofibration, and hence that each operad $\BOp^c(\Lambda^{-n}\EOp_{n}^{\vee})$
is cofibrant.
Thus,
we deduce from our results that $\BOp^c(\Lambda^{-n}\EOp_{n}^{\vee})$
defines a cofibrant replacement of~$\EOp_n$
in the category of operads.

\medskip
The following corollary of theorems~\ref{Prelude:RealizationTheorem}-\ref{Prelude:BarDualityTheorem}
is worth mentioning:

\begin{maincor}
The result of theorems~\ref{Prelude:RealizationTheorem}-\ref{Prelude:BarDualityTheorem}
extends formally to $n=\infty$
when $\Lambda^{-\infty}\EOp_{\infty}^{\vee}$
is defined as the colimit of the cooperads $\Lambda^{-n}\EOp_n^{\vee}$.
\end{maincor}

\begin{proof}
Recall that the forgetful functor from cooperads to $\Sigma_*$-objects creates colimits,
the forgetful functor from operads to $\Sigma_*$-objects create sequential colimits
and the cobar construction, from cooperads to operads,
preserves all colimits.
Therefore, we have
\begin{equation*}
\colim_n\BOp^c(\Lambda^{-n}\EOp_n^{\vee})\simeq\BOp^c(\Lambda^{-\infty}\EOp_{\infty}^{\vee})
\end{equation*}
and the weak-equivalences of Theorem~\ref{Prelude:RealizationTheorem}
give by passing to the colimit $n\rightarrow\infty$
a weak-equivalence:
\begin{equation*}
\BOp^c(\Lambda^{-\infty}\EOp_{\infty}^{\vee})\xrightarrow{\sim}\EOp_{\infty} = \EOp.
\end{equation*}
\end{proof}

In~\S\ref{FirstStep:Conclusion},
we give a short cut to the proof of this corollary.
On the way,
we prove that $\Lambda^{-\infty}\EOp_{\infty}^{\vee}$
is cofibrant as a $\Sigma_*$-object (see Proposition~\ref{FirstStep:Conclusion:ColimitCofibration}).
According to~\cite[Theorem 1.4.12]{FresseCylinder},
this result implies that $\BOp^c(\Lambda^{-\infty}\EOp_{\infty}^{\vee})$
is cofibrant as an operad.
Thus,
we deduce from our results that $\BOp^c(\Lambda^{-\infty}\EOp_{\infty}^{\vee})$
defines a cofibrant replacement of $\EOp$
in the category of operads
(and hence a cofibrant replacement of the commutative operad $\COp$).

\medskip
Theorems~\ref{Prelude:RealizationTheorem}-\ref{Prelude:BarDualityTheorem}
give an answer to the question of~\S\ref{Prelude:BarrattEccles}:
the suspension morphisms $\sigma: \EOp_n\rightarrow\Lambda^{-1}\EOp_{n-1}$
correspond to the embeddings $\iota: \EOp_{n-1}\rightarrow\EOp_n$
under the bar duality of operads.
In the epilogue,
we develop this remark to give an intrinsic representation of the morphism $\nabla_{S^m}: \EOp\rightarrow\End_{\bar{N}^*(S^m)}$
which gives the action of an $E_\infty$-operad on the cochain complex
of the $m$-sphere $S^m$.

\medskip
The next sections \S\S\ref{FirstStep}-\ref{FinalStep}
are entirely devoted to the proof of Theorem~\ref{Prelude:RealizationTheorem}.
From now on,
we use the short notation $\DOp_n = \Lambda^{-n}\EOp_n^{\vee}$
to refer to the dual cooperad
of $\EOp_n$.
In~\S\ref{SecondStep:OperadKStructure} and~\S\ref{FinalStep:KOperadLifting},
we also use the notation $\MOp_n = \Sigma^{-1}\DOp_n$
to refer to the generating $\Sigma_*$-object of $\BOp^c(\DOp_n) = (\FOp(\Sigma^{-1}\DOp_n),\partial)$.

\section{First step: applications of bar duality and obstruction theory}\label{FirstStep}

Since the cobar construction of a cooperad forms a quasi-free operad $\BOp^c(\DOp) = (\FOp(\Sigma^{-1}\bar{\DOp}),\partial)$,
any morphism $\phi: \BOp^c(\DOp)\rightarrow\POp$
toward an operad $\POp$
is fully determined by a homomorphism of degree $-1$
\begin{equation*}
\theta: \bar{\DOp}\rightarrow\POp
\end{equation*}
that represents the restriction of $\phi$ to the generating $\Sigma_*$-object of~$\BOp^c(\DOp)$ (see~\cite[\S 2.3]{GetzlerJones}).
We call this homomorphism $\theta$ the twisting cochain associated to $\phi$.

The commutation of $\phi$ with differentials reduces to a sequence of differential equations
\begin{equation}\label{eqn:TwistingCochains}
\delta(\theta{(r)}) +  \sum_{s+t=r-1} \theta{(s)}\smile\theta{(t)} = 0,
\end{equation}
where $\theta{(r)}: \DOp(r)\rightarrow\POp(r)$ represents the component of $\theta$ in arity $r\geq 2$
and $\smile$ is a certain operation on homomorphisms $f{(r)}: \DOp(r)\rightarrow\POp(r)$
(see also~\cite[\S 2.3]{GetzlerJones}).
To define a morphism $\phi: \BOp^c(\DOp)\rightarrow\POp$,
a natural idea is to construct the homomorphisms $\theta{(r)}: \DOp(r)\rightarrow\POp(r)$
by induction on $r$.
The obstruction to the existence of $\theta{(r)}$
is represented by the homology class of $q(\theta)(r) = \sum_{s+t=r-1} \theta{(s)}\smile\theta{(t)}$
in $\Hom_{\C}(\DOp,\POp)$.

In this section,
we check the vanishing of such obstructions
in order to prove the following result:

\begin{mainlemm}\label{FirstStep:Result}
The composite
\begin{equation*}
\xymatrix{ \BOp^c(\DOp_1)\ar@/^0.5em/@{.>}[]!UR;[rr]!UL^{\phi_1}\ar[r]_(0.7){\epsilon_1} &
\AOp\ar[r]_{\alpha} & \COp },
\end{equation*}
where $\epsilon_1$ is the Koszul duality equivalence
of the associative operad $\EOp_1 = \AOp$,
has extensions
\begin{equation*}
\xymatrix{ \BOp^c(\DOp_1)\ar[r]^(0.5){\sigma^*}\ar[d]_{\phi_1} &
\BOp^c(\DOp_2)\ar[r]^(0.65){\sigma^*}\ar@{.>}[d]_{\exists\phi_2} &
\cdots\ar[r]^(0.4){\sigma^*} &
\BOp^c(\DOp_n)\ar[r]^(0.65){\sigma^*}\ar@{.>}[d]_{\exists\phi_n} & \cdots \\
\COp\ar[r]_{=} &
\COp\ar[r]_{=} &
\cdots\ar[r]_{=} &
\COp\ar[r]_{=} & \cdots }.
\end{equation*}
\end{mainlemm}

In~\S\ref{FirstStep:Equations},
we prove that the construction of twisting cochains $\theta_n: \bar{\DOp}_n\rightarrow\POp$
associated to such morphisms $\phi_n$
amounts to the definition of elements $\omega_n(r)\in\EOp_n(r)_{\Sigma_r}$
satisfying equations equivalent to~(\ref{eqn:TwistingCochains}).
In~\S\S\ref{FirstStep:Homology}-\ref{FirstStep:Conclusion},
we review Cohen's computation of $H_*(\EOp_n(r)_{\Sigma_r}\otimes_{\ZZ}\FF) = H_*(\COp_n(r)_{\Sigma_r},\FF)$,
for $\FF = \QQ$ or $\FF = \FF_p$,
in order to prove the vanishing of homological obstructions to the definition of~$\omega_n(r)$
and we achieve the proof of Lemma~\ref{FirstStep:Result}.

\subsection{The equations of twisting cochains}\label{FirstStep:Equations}
The operadic analogues of the classical twisting cochains of dg-algebras
are introduced in~\cite[\S 2.3]{GetzlerJones}.
The correspondence between twisting cochains $\theta: \DOp\rightarrow\POp$
and morphisms $\phi: \BOp^c(\DOp)\rightarrow\POp$
goes back to this article. We also refer to the papers~\cite{FressePartitions,FresseCylinder},
of which we take the conventions, for a review of this correspondence.
Be careful that the authors of~\cite{GetzlerJones} study applications of operads
in the context of dg-modules over a field of characteristic zero.
Some equations have to be modified in that reference,
but the overall correspondence holds for any choice of ground ring (see~\cite{FressePartitions,FresseCylinder}).

In this subsection,
we make explicit the application of the general theory to the cooperad $\DOp = \DOp_n$ and the operad $\POp = \COp$
involved in our obstruction problem.
Since $\COp(r) = \ZZ$, for every $r\geq 1$,
and the Barratt-Eccles operad is degreewise finite dimensional,
we have immediately:

\begin{fact}\label{FirstStep:Equations:ElementAdjunction}
We have a bijective correspondence between homomorphisms of degree~$-1$
\begin{equation*}
\theta{(r)}: \DOp_n(r)\rightarrow\COp(r)
\end{equation*}
and homogeneous elements of degree $n(r-1)-1$
\begin{equation*}
\omega{(r)}\in\EOp_n(r)_{\Sigma_r},
\end{equation*}
for every $r\in\NN$.
\end{fact}

We use that an element of degree $d$ in $\DOp_n(r) = \Lambda^{-n}\EOp_n^{\vee}(r)$
is equivalent to a homomorphism $c: \EOp_n(r)\rightarrow\ZZ$
of degree $-n(r-1)+d$.
The homomorphism $\theta{(r)} = \theta_{\omega}{(r)}$ associated to $\omega{(r)}\in\EOp_n(r)_{\Sigma_r}$
is simply defined by the formula:
\begin{equation*}
\theta_{\omega}{(r)}(c) = \pm c(\omega{(r)}),
\end{equation*}
for all $c\in\DOp_n(r)$.
The sign comes from the symmetry isomorphism $\EOp_n(r)\otimes\EOp_n^{\vee}(r)\simeq\EOp_n^{\vee}(r)\otimes\EOp_n(r)$
involved in this relation.

\medskip
We have then:

\begin{lemm}\label{FirstStep:Equations:AdjointEquations}
For homomorphisms $\theta{(r)}: \DOp_n(r)\rightarrow\COp(r)$
associated to a collection of elements $\omega{(r)}\in\EOp_n(r)_{\Sigma_r}$,
the equation of twisting cochains~(\ref{eqn:TwistingCochains})
is equivalent to the equation
\begin{equation}\label{eqn:TwistingAdjointElements}
\delta(\omega{(r)}) + \sum_{s+t-1=r}\Bigl\{\sum_{i=1}^{s}\pm\omega{(s)}\circ_i\omega{(t)}\Bigr\} = 0
\end{equation}
in the component of degree $n(r-1)-2$ of~$\EOp_n(r)_{\Sigma_r}$,
where $\circ_i$ refers to the partial composition operations of the operad~$\EOp_n$.
\end{lemm}

\begin{proof}
Direct application of the formula of~\cite[\S 3.7]{FresseCylinder}
(see also \cite[\S 2.3]{GetzlerJones} in the original reference)
for the product $\theta{(s)}\smile\theta{(t)}$
when we take homomorphisms $\theta{(r)}: \DOp_n(r)\rightarrow\COp(r)$
associated to elements $\omega{(r)}\in\EOp_n(r)_{\Sigma_r}$, $r\in\NN$.
\end{proof}

\begin{prop}\label{FirstStep:Equations:AdjointExpression}\hspace*{2mm}
\begin{enumerate}
\item\label{AdjointExpressions:Morphisms}
We have a bijective correspondence between operad morphisms
\begin{equation*}
\phi_n: \BOp^c(\DOp_n)\rightarrow\COp
\end{equation*}
and collections of elements
\begin{equation*}
\omega_n{(r)}\in\EOp_n(r)_{\Sigma_r},\quad r\in\NN,
\end{equation*}
with $\deg(\omega_n{(r)}) = n(r-1)-1$
and such that equation~(\ref{eqn:TwistingAdjointElements})
of Lemma~\ref{FirstStep:Equations:AdjointEquations} holds.
\item\label{AdjointExpressions:TowerCommutation}
The commutativity of the diagrams
\begin{equation*}
\xymatrix{ \BOp^c(\DOp_{n-1})\ar[rr]^{\sigma^*}\ar@{.>}[dr]_{\phi_{n-1}} &&
\BOp^c(\DOp_n)\ar@{.>}[dl]^{\phi_n} \\ & \COp & }
\end{equation*}
amounts to the verification of equations
\begin{equation*}
\omega_{n-1}{(r)} = \sigma(\omega_{n}{(r)}),\quad r\in\NN,
\end{equation*}
for the elements $\omega_n{(r)}$
associated to the morphisms $\phi_n$.
\end{enumerate}
\end{prop}

\begin{proof}
Assertion~(\ref{AdjointExpressions:Morphisms}) is a direct corollary
of Fact~\ref{FirstStep:Equations:ElementAdjunction}
and Lemma~\ref{FirstStep:Equations:AdjointEquations}
since we have a bijective correspondence between operad morphisms
and twisting cochains.

The commutativity of the diagram of Assertion~(\ref{AdjointExpressions:TowerCommutation})
holds if and only if the twisting cochains associated to the morphisms $\phi_{n-1}\sigma^*$ and $\phi_n$ agree.
This relation amounts to the commutativity of the diagram
\begin{equation*}
\xymatrix{ \DOp_{n-1}\ar[rr]^{\sigma^*}\ar@{.>}[dr]_{\theta_{\omega_{n-1}}} &&
\DOp_n\ar@{.>}[dl]^{\theta_{\omega_n}} \\ & \COp & },
\end{equation*}
where $\theta_n$ refers to the twisting cochain associated to $\phi_n$.
(Just take the restriction of these morphisms $\phi_{n-1}\sigma^*$ and $\phi_n$
to the generating $\Sigma_*$-objects
of the cobar constructions.)
The conclusion of Assertion~(\ref{AdjointExpressions:TowerCommutation})
follows readily
since a straightforward calculation gives $\theta_{n}(\sigma^* c) = \pm\sigma(\omega_n{(r)})(c)$
when the twisting cochain $\theta_n = \theta_{\omega_n}$
is associated to a collection of elements $\omega_n(r)$, $r\in\NN$,
and we have $\theta_{n-1}(c) = \pm\omega_{n-1}{(r)}(c)$
by definition of $\theta_{n-1} = \theta_{\omega_{n-1}}$.
\end{proof}

The next lemmas are used in~\S\ref{FirstStep:Conclusion}
to conclude to the existence of collections $\omega_n{(r)}$, $r\in\NN$,
that fulfil the requirements of Proposition~\ref{FirstStep:Equations:AdjointExpression}.

\begin{mainseclemm}\label{FirstStep:Equations:SuspensionSurjectivity}
The morphism $\sigma: \EOp_n(r)\rightarrow\Lambda^{-1}\EOp_{n-1}(r)$
is surjective for each $r\in\NN$ and every $n\geq 2$.
\end{mainseclemm}

\begin{proof}
Let $\underline{w} = (w_0,\dots,w_d)\in\EOp(r)$.
Suppose $w_0 = (i_1,\dots,i_r)$.
Form the sequence of permutations
\begin{multline*}
t_r = w_0 = (i_1,i_2,i_3,i_4,\dots,i_r),\\
\quad t_{r-1} = (i_2,i_1,i_3,i_4,\dots,i_r),
\quad t_{r-2} = (i_3,i_2,i_1,i_4,\dots,i_r),\quad\ldots\\
\ldots\quad t_1 = (i_r,i_{r-1},i_{r-2},\dots,i_1),
\end{multline*}
and the simplex $\underline{t} = (t_1,\dots,t_{r-1},w_0,\dots,w_d)$.
We have $\mu_{i j}(\underline{t}) = \mu_{i j}(\underline{w}) + 1$, $\forall i j$,
and hence ``$\underline{w}\in\EOp_{n-1}(r)\Rightarrow\underline{t}\in\EOp_n(r)$''.
Moreover, we have clearly $\sigma(\underline{t}) = \sgn\cap\underline{t} = \pm\underline{w}$
by definition of the cochain $\sgn: \EOp(r)\rightarrow\ZZ$.
\end{proof}

\begin{mainseclemm}\label{FirstStep:Equations:ObstructionVanishing}
We have $H_{n(r-1)-2}(\ker\{\sigma_*: \EOp_n(r)_{\Sigma_r}\rightarrow\Lambda^{-1}\EOp_{n-1}(r)_{\Sigma_r}\}) = 0$ for all $r>2$.
\end{mainseclemm}

The proof of this lemma is postponed to the end of the next subsection,
because we have to review homology computations of~\cite{Cohen}
in order to reach its conclusion.

\subsection{Applications of Cohen's results}\label{FirstStep:Homology}
In this subsection (and in this subsection only),
we have to deal with dg-modules (and algebras) over extensions of $\ZZ$.
To be explicit,
we take a field $\FF$ which is either the field of rationals $\QQ$, or a finite primary field $\FF_p$,
and we consider the category of dg-modules over~$\FF$
for which we adopt the notation $\C_{\FF}$.
Note that every dg-module of $\C_{\FF}$
is cofibrant
since $\FF$ is a field.
The extension of the Barratt-Eccles operad
to $\FF$
\begin{equation*}
(\EOp_n\otimes_{\ZZ}\FF)(r) = \EOp_n(r)\otimes_{\ZZ}\FF
\end{equation*}
defines an operad in $\C_{\FF}$
and we consider the category of algebras over this operad in $\C_{\FF}$.

The theorems of~\cite{Cohen} gives the homology of $H_*(\COp_n(r)_{\Sigma_r},\FF)$
when we take $\FF = \QQ,\FF_p$
as coefficients.
The result is written in terms of natural operations
acting on the homology of $E_n$-algebras.
We use this representation for computing the action of our suspension morphism
on $H_*((\EOp_n\otimes_{\ZZ}\FF)(r)_{\Sigma_r}) = H_*(\COp_n(r)_{\Sigma_r},\FF)$
and we apply standard arguments of homological algebra
to prove the vanishing property of Lemma~\ref{FirstStep:Equations:ObstructionVanishing}.

First of all,
we review a definition of the homology operations
acting on the homology of algebras over an operad.

\subsubsection{Recollections: free algebras over an operad}\label{FirstStep:Homology:FreeOperadAlgebras}
Recall that the free algebra over an operad $\POp$
associated to an object $E$
in a symmetric monoidal category $\E$
is realized by a generalized symmetric algebra
\begin{equation*}
S(\POp,E) = \bigoplus_{r=0}^{\infty} \POp(r)\otimes_{\Sigma_r} E^{\otimes r},
\end{equation*}
where $\otimes$ refers to the tensor product of $\E$.

In fact,
the generalized symmetric algebra functor $S(\POp,-): \E\rightarrow\E$
associated to an operad $\POp$
defines a monad on the category $\E$.
In certain references on operads,
an algebra over an operad $\POp$
is defined as an algebra over this monad $S(\POp,-): \E\rightarrow\E$
associated to $\POp$ (this definition goes back to the introduction of operads in~\cite{May}).
In this setting,
the structure of a $\POp$-algebra
is determined by a morphism $\lambda: S(\POp,A)\rightarrow A$
satisfying natural associativity and unit relations
with respect to structure morphisms of $S(\POp,-)$.
The assertion that $S(\POp,E)$
represents the free $\POp$-algebra associated to $E$
is a tautological
consequence of this definition.

In the context of a closed symmetric monoidal category $\E$,
the definition of a $\POp$-algebra as an algebra over the monad $S(\POp,-)$
is equivalent to the definition by endomorphism operads (the definition used in~\S\ref{Prelude:BarrattEccles:EndomorphismOperads})
because:
\begin{itemize}
\item
the structure morphism $\lambda: S(\POp,A)\rightarrow A$
amounts to a collection of morphisms
\begin{equation*}
\nabla: \POp(r)\rightarrow\Hom_{\E}(A^{\otimes r},A)
\end{equation*}
commuting with symmetric group actions;
\item
the natural associativity and unit relations of monad actions
amounts to the requirement that $\nabla$
defines an operad morphism from $\POp$ to the endomorphism operad of $A$.
\end{itemize}

\subsubsection{Free algebras over operads and homology operations}\label{FirstStep:Homology:Operations}
Suppose $\POp$ is a $\Sigma_*$-cofibrant operad in the category of dg-modules $\C_{\FF}$.
Then we have a monad~$T(\POp,-)$
on the category of graded modules such that $H_*(S(\POp,E)) = T(\POp,H_*(E))$,
for every $E\in\C_{\FF}$.
Moreover,
the structure morphism of a $\POp$-algebra $\lambda: S(\POp,A)\rightarrow A$
induces a morphism
\begin{equation*}
T(\POp,H_*(A)) = H_*(S(\POp,A))\xrightarrow{\lambda_*} H_*(A)
\end{equation*}
which provides the homology module $H_*(A)$
with the structure of an algebra over this monad $T(\POp,-)$.
Naturally,
the evaluation of the monad $T(\POp,-)$
on a graded $\FF$-module $E$
is defined by the formula $T(\POp,E) = H_*(S(\POp,E))$,
where we identify $E$
with a dg-module equipped with a trivial differential
to form the dg-module of the right hand side $S(\POp,E)$.

This definition of the monad of homological operations associated to an operad
is a reminiscence of the approach of~\cite{FresseSimpAlg,Goerss}
for the definition of homotopical operations
associated to simplicial algebras over monads.
For the purpose of this paper,
we only need the overall idea that elements of $T(\POp,E) = H_*(S(\POp,E))$,
where $E$ runs over graded $\FF$-modules,
represent natural operations
acting on the homology of $\POp$-algebras (we refer to~\cite{MayOperations} for the origin of this idea).

To be explicit,
let $E = \bigoplus_{i=1}^{m}\FF e_i$
be the graded $\FF$-module generated by homogeneous elements $e_i$
of degree $\deg(e_i) = d_i$
together with a trivial differential $\delta(e_i) = 0$.
The elements $\pi\in H_d(S(\POp,E))$
represent homology operations $q_{\pi}: H_{d_1}(A)\times\dots\times H_{d_m}(A)\rightarrow H_{d}(A)$.
Formally,
any choice of representatives $a_i\in A$ of homology classes $c_i\in H_{d_\alpha}(A)$
determines a dg-module morphism $c: E\rightarrow A$
such that $c(e_i) = a_i$.
The evaluation of $q_{\pi}$
on the collection of classes $\{c_i\}_i$
is determined by the image of $\pi\in H_d(S(\POp,E))$
under the morphism
\begin{equation*}
H_*(S(\POp,E))\xrightarrow{S(\POp,c)_*} H_*(S(\POp,A))\xrightarrow{\lambda_*} H_*(A)
\end{equation*}
induced by $c: E\rightarrow A$
and the structure morphism of the $\POp$-algebra $A$.
The homology class $\pi\in H_d(S(\POp,E))$
represents itself the evaluation of this operation $q_{\pi}$
on the generating elements $e_i$
in the homology of the free $\POp$-algebra $H_*(S(\POp,E))$
when we identify these elements $e_i\in E$
with fundamental homology classes of $H_*(S(\POp,E))$.

\subsubsection{Homology operations and weight gradings}\label{FirstStep:Homology:OperationWeight}
For a graded module $E = \FF e_1\oplus\dots\oplus\FF e_m$,
the free $\POp$-algebra $S(\POp,E)$
inherits a natural splitting
\begin{equation}\label{eqn:FreeAlgebraSplitting}
S(\POp,E) = \bigoplus_{(r_1,\dots,r_m)} S_{(r_1,\dots,r_m)}(\POp,E) \\
\end{equation}
such that $S_{(r_1,\dots,r_m)}(\POp,E)
= \POp(r)\otimes_{\Sigma_{r_*}} (\FF e_1)^{\otimes r_1}\otimes\dots\otimes(\FF e_m)^{\otimes r_m}$,
where we set $r = r_1+\dots+r_m$ and $\Sigma_{r_*} = \Sigma_{r_1}\times\dots\times\Sigma_{r_m}$.
Moreover,
we have an obvious isomorphism
\begin{equation*}
S_{(r_1,\dots,r_m)}(\POp,E) = \Sigma^{d}\POp(r)_{\Sigma_{r_*}},
\end{equation*}
where $d = r_1\deg(e_1) + \dots + r_m\deg(e_m)$.
Hence,
any homology class $c\in H_*(\POp(r)_{\Sigma_{r_*}})$
determines a whole collection of classes
$\pi_c\in H_*(S(\POp,\FF e_1\oplus\dots\oplus\FF e_m))$
by suspension
and a whole collection of homology operations
\begin{equation*}
q_{\pi_c}: H_{d_1}(A)\times\dots\times H_{d_m}(A)\rightarrow H_{d}(A)
\end{equation*}
such that $d = r_1 d_1 + \dots + r_m d_m + \deg(c)$,
where $d_1,\dots,d_m\in\ZZ$.

The collection $(r_1,\dots,r_m)\in\NN^{m}$
is a weight
naturally associated to $c$
and the corresponding operations $q_{\pi_c}$.
The splitting of~(\ref{eqn:FreeAlgebraSplitting})
gives a splitting of the module of homology operations $T(\POp,E) = H_*(S(\POp,E))$
into homogeneous weight components.

\medskip
Our next purpose is to describe the homology modules $H_*((\EOp_n\otimes_{\ZZ}\FF)(r)_{\Sigma_r})$
associated to $\FF$-extensions of the operads $\EOp_n$.
For this aim,
we use that $H_*((\EOp_n\otimes_{\ZZ}\FF)(r)_{\Sigma_r})$
is identified with the weight $r$ component
of the module of homology operations $T(\EOp_n\otimes_{\ZZ}\FF,\FF x)$
of one variable $x$ in degree~$0$.

\subsubsection{Rational homology operations associated to $E_n$-algebras}\label{FirstStep:Homology:RationalOperations}
In~\S\S\ref{Prelude:GerstenhaberOperads:RepresentativeOperations}-\ref{Prelude:GerstenhaberOperads:RepresentativeIdentities},
we have already mentioned the existence of operations giving:
\begin{enumerate}
\item
a product $\mu: H_*(A)\otimes H_*(A)\rightarrow H_*(A)$ of degree $0$
\item
and a bracket $\lambda_{n-1}: H_*(A)\otimes H_*(A)\rightarrow H_{*+n-1}(A)$
of degree $n-1$
\end{enumerate}
on the homology of any $\EOp_n$-algebra.
These operations are associated to elements $\mu\in H_0(\EOp_n(2))$
and $\lambda_{n-1}\in H_0(\EOp_{n-1}(2))$
defined over $\ZZ$.
In the representation of~\S\ref{FirstStep:Homology:Operations},
we form the free $\EOp_n$-algebra on two generating elements $S(\EOp_n,\ZZ x_1\oplus\ZZ x_2)$
and we use the natural embedding
\begin{equation*}
\EOp_n(2)\otimes(\ZZ x_1)\otimes(\ZZ x_2)\subset\EOp_n(2)\otimes_{\Sigma_2}(\ZZ x_1\oplus\ZZ x_2)^{\otimes 2}
\end{equation*}
to define the homology classes
\begin{equation*}
\mu(x_1,x_2),\lambda_{n-1}(x_1,x_2)\in H_*(S(\EOp_n,\ZZ x_1\oplus\ZZ x_2))
\end{equation*}
associated to these operations.

If we change the ground ring to the field of rationals $\FF = \QQ$,
then all natural operations on the homology of an $\EOp_n$-algebra are composites of the product
and the Browder bracket,
because we have K\"unneth isomorphisms
\begin{equation*}
H_*(S(\EOp_n\otimes_{\ZZ}\QQ,E))
\simeq S(H_*(\EOp_n\otimes_{\ZZ}\QQ),H_*(E))
\simeq S(H_*(\EOp_n)\otimes_{\ZZ}\QQ,H_*(E)),
\end{equation*}
for every $E\in\C_{\QQ}$,
and we have already mentioned that $\GOp_n = H_*(\EOp_n)$
is the operad generated by these operations.
Consequently:

\begin{fact}\label{FirstStep:Homology:RationalCase}
The homology $H_*((\EOp_n\otimes_{\ZZ}\QQ)(r)_{\Sigma_r})$
is the component of weight $r$
of a free algebra $U(L_{n-1})$,
where:
\begin{itemize}
\item
the notation $L_{n-1}$ represents the free Lie algebra
on a generator $x$ of degree $0$,
together with a bracket $\lambda_{n-1}: L_{n-1}\otimes L_{n-1}\rightarrow L_{n-1}$ of degree $n-1$;
the elements of $L_{n-1}$ are Lie monomials
\begin{equation*}
\gamma_m(x,\dots,x) = \lambda_{n-1}(\cdots\lambda_{n-1}(\lambda_{n-1}(x,x),x),\dots,x)
\end{equation*}
in weight $r = m$ and in degree $d = (n-1)(m-1)$,
where $m$ is the number of occurrences of the variable $x$;
\item
the notation $U(L_{n-1})$ represents the free commutative algebra
generated by $L_{n-1}$;
the weight of a product is defined by $\weight(a b) = \weight(a) + \weight(b)$.
\end{itemize}
\end{fact}

Naturally, the Lie monomials $\gamma_m(x,\dots,x)$ vanish for $m>1$ when $n$ is odd,
for $m>2$ when $n$ is even,
because of the symmetry and Jacobi relations.

\subsubsection{Homology operations associated to $E_n$-algebras mod $p$}\label{FirstStep:Homology:PrimaryOperations}
If we change the ground ring to a finite primary field $\FF = \FF_p$ such that $p>2$,
then we have besides the product and the Browder bracket:
\begin{enumerate}
\item
a Frobenius operation $\xi_{n-1}: H_d(A)\rightarrow H_{p d + (n-1)(p-1)}(A)$,
defined when $d+n-1$ is even,
\item
a Bockstein Frobenius $\zeta_{n-1}: H_d(A)\rightarrow H_{p d + (n-1)(p-1)-1}(A)$,
defined when $d+n-1$ is even too,
\item
a Bockstein $\beta: H_d(A)\rightarrow H_{d-1}(A)$
and Dyer-Lashof operations $Q^i: H_d(A)\rightarrow H_{d + 2i(p-1)}(A)$
defined when $2 i<d+n-1$.
\end{enumerate}

For the sequel,
it is convenient to set
\begin{equation*}
Q^{d+n-1/2}(a) = \xi_{n-1}(a)\quad\text{and}\quad\beta Q^{d+n-1/2}(a) = \zeta_{n-1}(a),
\end{equation*}
for any class $a$ of degree $d$
such that $d+n-1$ is even.
Note however that these top operations are not additive
and, according to the definition of~\cite{Cohen},
the operation $\zeta_{n-1}$ differs from the Bockstein of $\xi_{n-1}$
by a corrective term.

In the representation of~\S\ref{FirstStep:Homology:Operations},
the operations $\beta^{\epsilon} Q^i$
correspond to generating classes of the homology modules $H_*((\EOp_n\otimes_{\ZZ}\FF_p)(p)_{\Sigma_p})
= H_d(\COp_n(p)_{\Sigma_p},\FF_p)$,
for which we have (see~\cite{Cohen}):
\begin{equation*}
H_d((\EOp_n\otimes_{\ZZ}\FF_p)(p)_{\Sigma_p})
= \begin{cases} \FF_p, & \text{if $d$ is of the form $d = 2 i (p-1) - \epsilon$}, \\
& \quad\text{for $i = 0,1,2,\dots$ and $\epsilon = 0,1$} \\
& \quad\text{such that $0\leq d\leq(n-1)(p-1)$}, \\
0, & \text{otherwise}. \end{cases}
\end{equation*}

In the case $p=2$,
we still have:
\begin{enumerate}
\item
a Frobenius operation $\xi_{n-1}: H_d(A)\rightarrow H_{2 d + (n-1)}(A)$,
defined for all $d$,
\item
and Araki-Kudo operations $Q^i: H_d(A)\rightarrow H_{d+i}(A)$,
defined for $i<d+n-1$.
\end{enumerate}
For the sequel,
we also set
\begin{equation*}
Q^{d+n-1}(a) = \xi_{n-1}(a)
\end{equation*}
when $a$ has degree $d$.
The operations $Q^i$
correspond to generating classes of the homology modules:
\begin{equation*}
H_d((\EOp_n\otimes_{\ZZ}\FF_2)(2)_{\Sigma_2})
= \begin{cases} \FF_2, & \text{for $d = 0,1,\dots,n-1$}, \\
0, & \text{otherwise}. \end{cases}
\end{equation*}

The operation $\xi_{n-1}$ is an analogue, with respect to the Browder bracket $\lambda_{n-1}$,
of the Frobenius of restricted Lie algebras.

\begin{fact}[{see~\cite{Cohen}}]\label{FirstStep:Homology:FiniteFieldCase}
In the case $p>2$,
the homology $H_*(\EOp_n(r)_{\Sigma_r}\otimes_{\ZZ}\FF_p)$
is the component of weight $r$
of a free algebra $U(R_{n-1} L_{n-1})$,
where:
\begin{itemize}
\item
the notation $L_{n-1}$ represents a free (unrestricted) Lie algebra
on a generator $x$ of degree $0$
together with a bracket $\lambda_{n-1}: L_{n-1}\otimes L_{n-1}\rightarrow L_{n-1}$ of degree $n-1$,
as in the rational case;
\item
the notation $R_{n-1} L_{n-1}$ represents the free $\FF_p$-module
generated by composites of Dyer-Lashof operations
\begin{equation*}
\beta^{\epsilon_1} Q^{i_1}\beta^{\epsilon_2} Q^{i_2}\cdots \beta^{\epsilon_l} Q^{i_l}(\gamma)
\end{equation*}
in weight $r = p^l\weight(\gamma)$
and in degree $d = \sum_k 2 i_k (p-1) - \epsilon_k + \deg(\gamma)$
satisfying admissibility
\begin{equation*}
i_1\leq p i_2 - \epsilon_2,\ i_2\leq p i_3 - \epsilon_3,\ \dots\ i_{l-1}\leq p i_l - \epsilon_l
\end{equation*}
and excess conditions
\begin{equation*}
\begin{array}{ccccc}
d_l & \leq & 2 i_l (p-1) - \epsilon_l & \leq & d_l+n-1, \\
& \vdots && \vdots & \\
d_2 & \leq & 2 i_2 (p-1) - \epsilon_2 & \leq & d_2+n-1, \\
d_1 & \leq & 2 i_1 (p-1) - \epsilon_1 & \leq & d_1+n-1,
\end{array}
\end{equation*}
where we set $d_k = 2 i_{k+1} (p-1) - \epsilon_{k+1} + \dots + 2 i_l (p-1) - \epsilon_l + \deg(\gamma)$,
for $k = 1,\dots,l$;
\item
the notation $U(R_{n-1} L_{n-1})$ refers to the free commutative algebra generated by $R_{n-1} L_{n-1}$
divided out by relations $Q^{\deg(a)}(a)\equiv a^p$,
for every element $a\in U(R_{n-1} L_{n-1})$
such that $\deg(a)$ is even.
\end{itemize}

In the case $p=2$,
the same result holds provided that we drop Bockstein operations,
we replace terms $2 i_k (p-1) - \epsilon_k$ by $i_k$
in the definition of $R_{n-1} L_{n-1}$,
and we remove the parity condition in the relation $Q^{\deg(a)}(a)\equiv a^2$.
\end{fact}

From Fact~\ref{FirstStep:Homology:RationalCase}
and Fact~\ref{FirstStep:Homology:FiniteFieldCase},
we obtain:

\begin{prop}\label{FirstStep:Homology:TopDegree}\hspace*{2mm}
In each case $\FF = \QQ,\FF_p,\FF_2$,
the homology $H_d((\EOp_n\otimes_{\ZZ}\FF)(r)_{\Sigma_r})$
\begin{enumerate}
\item
vanishes in degree $d>(n-1)(r-1)$;
\item
and is spanned in degree $d = (n-1)(r-1)$ by elements of the form
\begin{equation*}
\xi^l_{n-1}(\gamma_m) = \xi_{n-1}^l(\lambda_{n-1}(\cdots\lambda_{n-1}(\lambda_{n-1}(x,x),x),\dots,x))
\end{equation*}
with $r = p^l m$,
where $m$ denotes the number of occurrences of the variable $x$ in the Lie monomial.
(Just forget the power of the Frobenius $\xi_{n-1}^l$ in the case $\FF = \QQ$.)
\end{enumerate}
\end{prop}

\begin{proof}
We have clearly $\deg(a) = (n-1)(\weight(a)-1)$ for elements of the form $a = \xi^l_{n-1}(\gamma_m)$.
We have $\deg(a)<(n-1)(\weight(a)-1)$
for lower composites of operations
\begin{equation*}
a = \beta^{\epsilon_1} Q^{i_1}\beta^{\epsilon_2} Q^{i_2}\cdots \beta^{\epsilon_l} Q^{i_l}(\gamma),
\end{equation*}
because the degree increase is lower for an operation $\beta^{\epsilon_k} Q^{i_k}$
below the Frobenius $\xi_{n-1}$.
We have similarly $\deg(a b)<(n-1)(\weight(a b)-1)$
for any non-trivial product $a b$.

We obtain from these observations that the largest increase of degree relatively to a weight variation
is reached by the application of brackets and Frobenius operations
and agrees with the bound of the proposition.
We draw our conclusion from this assertion.
\end{proof}

We compute the action of the suspension morphism $\sigma: \EOp_n\rightarrow\Lambda^{-1}\EOp_{n-1}$
on the top homology component of $H_*((\EOp_n\otimes_{\ZZ}\FF)(r)_{\Sigma_r})$.
We rather determine the morphism
\begin{equation*}
\sigma_*: H_*(S(\EOp_n\otimes_{\ZZ}\FF,E))\rightarrow H_*(S(\Lambda^{-1}\EOp_{n-1}\otimes_{\ZZ}\FF,E))
\end{equation*}
induced by $\sigma$
on certain homology classes of the free $\EOp_n\otimes_{\ZZ}\FF$-algebra $A = S(\EOp_n\otimes_{\ZZ}\FF,E)$,
for any choice of $E$.

We already know that $\sigma_*(\lambda_{n-1}) = \lambda_{n-2}$
in $H_*(\EOp_{n-1}(2))$.
We deduce from this result:

\begin{prop}\label{FirstStep:Homology:SuspensionBracket}
Let $\FF$ be any ring over $\ZZ$.
For any pair of elements $(a,b)$
in the homology of a free $\EOp_n\otimes_{\ZZ}\FF$-algebra $A = S(\EOp_n\otimes_{\ZZ}\FF,E)$,
we have the relation
\begin{equation*}
\sigma_*(\lambda_{n-1}(a,b)) = \lambda_{n-2}(\sigma_*(a),\sigma_*(b))
\end{equation*}
in $H_*(S(\Lambda^{-1}\EOp_{n-1}\otimes_{\ZZ}\FF,E))$.
\end{prop}

\begin{proof}
We have $\sigma(p(a_1,\dots,a_r)) = \sigma(p)(\sigma(a_1),\dots,\sigma(a_r))$
for any composite $p(a_1,\dots,a_r)$
in the free algebra $S(\EOp_n\otimes_{\ZZ}\FF,E)$,
because $\sigma$ is an operad morphism.
The proposition follows immediately.
\end{proof}

\begin{prop}\label{FirstStep:Homology:SuspensionFrobenius}
Let $\FF = \FF_p$ be any finite primary field.
For any class $a$
in the homology of a free $\EOp_n\otimes_{\ZZ}\FF$-algebra $A = S(\EOp_n\otimes_{\ZZ}\FF,E)$,
we have the relation
\begin{equation*}
\sigma_*(\xi_{n-1}(a)) = \xi_{n-2}(\sigma_*(a))
\end{equation*}
in $H_*(S(\Lambda^{-1}\EOp_{n-1}\otimes_{\ZZ}\FF,E))$.
\end{prop}

\begin{proof}
Let $x$ be a variable of degree $0$.
Recall that $\xi_{n-1}(x)$
represents a generating element of the homology module $H_d(\EOp_{n}\otimes_{\ZZ}\FF)(p)_{\Sigma_p})$
in degree $d = (p-1)(n-1)$.
The image of $\xi_{n-1}(x)$ under the morphism $\sigma_*$
is defined by an element of the homology module
\begin{equation*}
H_{d}(\Lambda^{-1}(\EOp_{n-1}\otimes_{\ZZ}\FF)(p)_{\Sigma_p})
= H_{d-p+1}((\EOp_{n-1}\otimes_{\ZZ}\FF)(p)_{\Sigma_p}),
\end{equation*}
which is also one dimensional.
Accordingly,
the image of $\xi_{n-1}(x)$
under the suspension morphism
\begin{equation*}
H_*(S(\EOp_{n}\otimes_{\ZZ}\FF,\FF x))
\xrightarrow{\sigma_*} H_*(S(\Lambda^{-1}\EOp_{n-1}\otimes_{\ZZ}\FF,\FF x))
= H_{*+1}(S(\EOp_{n-1}\otimes_{\ZZ}\FF,\FF y))
\end{equation*}
is a multiple of $\xi_{n-2}(y)$,
where the variable $y$ is a one-fold suspension of $x$.
By definition of the operation $\xi_{n-1}(a)$
we have a relation
\begin{equation*}
\sigma_*(\xi_{n-1}(a)) = \Cst\cdot\xi_{n-2}(\sigma_*(a))
\end{equation*}
for any class $a$ in the homology of a free $\EOp_{n}\otimes_{\ZZ}\FF$-algebra $S(\EOp_n\otimes_{\ZZ}\FF,E)$,
where $\Cst$ is a multiplicative constant (possibly null).

According to~\cite{Cohen},
the Frobenius operation $\xi_{n-1}$
satisfies the restriction relation $\lambda_{n-1}(\xi_{n-1}(x),y) = \ad_{n-1}^{p}(x)(y)$,
for every pair of variables $(x,y)$,
where we set $\ad_{n-1}(x) = \lambda_{n-1}(x,-)$.
By applying the morphism $\sigma_*$
to this relation,
we obtain an equation:
\begin{align*}
&& \Cst\cdot\lambda_{n-2}(\xi_{n-2}(\sigma_*(x)),\sigma_*(y)) & = \ad_{n-2}^p(\sigma_*(x))(\sigma_*(y)) \\
\Rightarrow && \Cst\cdot\ad_{n-2}^p(\sigma_*(x))(\sigma_*(y)) & = \ad_{n-2}^p(\sigma_*(x))(\sigma_*(y)),
\end{align*}
from which we deduce the identity $\Cst = 1$.
Hence,
we conclude that $\sigma_*(\xi_{n-1}(a)) = \xi_{n-2}(\sigma_*(a))$,
for every $a$.
\end{proof}

According to Proposition~\ref{FirstStep:Homology:TopDegree},
these propositions imply immediately:

\begin{lemm}\label{FirstStep:Homology:SuspensionAction}
The morphism
\begin{equation*}
\sigma_*: H_d(\EOp_n(r)_{\Sigma_r}\otimes_{\ZZ}\FF)\rightarrow H_{d-r+1}(\EOp_{n-1}(r)_{\Sigma_r}\otimes_{\ZZ}\FF)
\end{equation*}
induced by the suspension $\sigma: \EOp_n\rightarrow\Lambda^{-1}\EOp_{n-1}$
is an isomorphism in degree $d = (r-1)(n-1)$,
for each primary field $\FF = \QQ,\FF_p,\FF_2$.\qed
\end{lemm}

Lemma~\ref{FirstStep:Equations:ObstructionVanishing}
requires to determine the action of the suspension morphism
in degree $d = (r-1)n-2$.
Observe that:

\begin{obsv}\label{FirstStep:Homology:ObstructionBounds}
We have $(r-1)(n-1) = (r-1)n-2$ for $r=3$ and $(r-1)(n-1)<(r-1)n-2$ for $r>3$.
\end{obsv}

Hence,
since all homology groups vanish in degree $d>(r-1)(n-1)$,
we obtain that the suspension induces an isomorphism in degree $d = (r-1)n-2$
for all $r>2$.
We use this observation to deduce:

\begin{lemm}\label{FirstStep:Homology:ObstructionVanishing}
We have
\begin{equation*}
H_{n(r-1)-2}(\ker\{\sigma_*: \EOp_n(r)_{\Sigma_r}\rightarrow\Lambda^{-1}\EOp_{n-1}(r)_{\Sigma_r}\}\otimes_{\ZZ}\FF) = 0
\end{equation*}
when $r>2$,
for each primary field $\FF = \QQ,\FF_p,\FF_2$.
\end{lemm}

\begin{proof}
Since $\sigma_*: \EOp_n(r)_{\Sigma_r}\rightarrow\Lambda^{-1}\EOp_{n-1}(r)_{\Sigma_r}$
forms a surjective morphism of free $\ZZ$-modules in all degrees by Lemma~\ref{FirstStep:Equations:SuspensionSurjectivity},
we have
\begin{multline*}
\ker\{\sigma_*: \EOp_n(r)_{\Sigma_r}\rightarrow\Lambda^{-1}\EOp_{n-1}(r)_{\Sigma_r}\}\otimes_{\ZZ}\FF\\
= \ker\{\sigma_*: (\EOp_n\otimes_{\ZZ}\FF)(r)_{\Sigma_r}\rightarrow\Lambda^{-1}(\EOp_{n-1}\otimes_{\ZZ}\FF)(r)_{\Sigma_r}\}
\end{multline*}
and we obtain a short exact sequence of dg-modules
\begin{equation*}
0\rightarrow\ker\{\sigma_*\}\otimes_{\ZZ}\FF
\rightarrow(\EOp_n\otimes_{\ZZ}\FF)(r)_{\Sigma_r}
\xrightarrow{\sigma_*}(\EOp_n\otimes_{\ZZ}\FF)(r)_{\Sigma_r}
\rightarrow 0.
\end{equation*}
Form the associated homology exact sequence
\begin{multline*}
0 = H_{(n-1)(r-1)-1}((\EOp_{n-1}\otimes_{\ZZ}\FF)(r)_{\Sigma_r})\rightarrow
H_{n(r-1)-2}(\ker\{\sigma_*\}\otimes_{\ZZ}\FF)\\
\rightarrow H_{n(r-1)-2}((\EOp_n\otimes_{\ZZ}\FF)(r)_{\Sigma_r})
\xrightarrow{\sigma_*} H_{(n-1)(r-1)-2}((\EOp_{n-1}\otimes_{\ZZ}\FF)(r)_{\Sigma_r})
\end{multline*}
and use the result of Lemma~\ref{FirstStep:Homology:SuspensionAction}
to conclude.
\end{proof}

Since the vanishing property of Lemma~\ref{FirstStep:Homology:ObstructionVanishing}
holds for every primary field $\FF = \QQ,\FF_p,\FF_2$,
we conclude that the homology of the unreduced dg-module
\begin{equation*}
\ker\{\sigma_*: \EOp_n(r)_{\Sigma_r}\rightarrow\Lambda^{-1}\EOp_{n-1}(r)_{\Sigma_r}\}
\end{equation*}
vanishes in degree $d = n(r-1)-2$
and this achieves the proof of Lemma~\ref{FirstStep:Equations:ObstructionVanishing}.
\qed

\subsection{Conclusion of the first step}\label{FirstStep:Conclusion}
We can now solve the obstruction problems
stated in~\S\ref{FirstStep:Equations}:

\begin{lemm}\label{FirstStep:Conclusion:TwistingElementExistence}
We have a full set of elements $\omega_n(r)\in\EOp_n(r)$, $r\geq 2$, for $n\geq 1$,
satisfying the requirements of Proposition~\ref{FirstStep:Equations:AdjointExpression}:
\begin{align}
\delta(\omega_n(r)) & + \sum_{s+t-1 = r}\Bigl\{\sum_{i=1}^{s}\pm\omega_n(s)\circ_i\omega_n(t)\Bigr\} = 0,
\quad\text{for all $r>2$, $n>1$}, \\
\sigma(\omega_n(r)) & = \omega_{n-1}(r),
\end{align}
with $\deg(\omega_n(r)) = n(r-1)-1$.
\end{lemm}

\begin{proof}
We define the elements $\omega_n(r)$
by an induction on $n$ and $r$
starting with $\omega_n(1) = 0$ for every $n\geq 1$.
For $n=1$,
we take
\begin{equation*}
\omega_1(r) = \begin{cases} \mu\in\EOp_1(2), & \text{if $r=2$}, \\ 0, & \text{otherwise}, \end{cases}
\end{equation*}
since this element corresponds to the usual twisting cochain $\theta_1: \Lambda^{-1}\bar{\AOp}^{\vee}\rightarrow\AOp$.
Suppose we have defined suitable elements $\omega_m(s)$
for $m<n$
and for $m=n$ and $s<r$.
Since $\sigma$ is surjective (by Lemma~\ref{FirstStep:Equations:SuspensionSurjectivity}),
we can pick an element $\omega^0_n(r)$
such that $\sigma(\omega^0_n(r)) = \omega_{n-1}(r)$.

For an element of the form $\omega_n(r) = \omega^0_n(r) + \chi_n(r)$,
with $\chi_n(r)\in\ker\sigma_*$,
Equation~(\ref{eqn:TwistingAdjointElements})
amounts to:
\begin{equation}\label{eqn:TwistingKernelElement}
\delta(\chi_n(r)) = - \sum_{s+t-1 = r}\Bigl\{\sum_{i=1}^{s}\pm\omega_n(s)\circ_i\omega_n(t)\Bigr\}
- \delta(\omega^0_n(r)).
\end{equation}
The suspension morphism maps the right-hand side of this equation
to
\begin{multline*}
\pm\sum_{s+t-1 = r}\Bigl\{\sum_{i=1}^{s}\pm\sigma(\omega_n(s))\circ_i\sigma(\omega_n(t))\Bigr\}
+ \pm\delta(\sigma(\omega^0_n(r))) \\
= \pm\sum_{s+t-1 = r}\Bigl\{\sum_{i=1}^{s}\pm\omega_{n-1}(s)\circ_i\omega_{n-1}(t)\Bigr\}
+ \pm\delta(\omega_{n-1}(r)) = 0.
\end{multline*}
By induction, we also have:
\begin{equation*}
\delta\Bigl\{\sum_{s+t-1 = r}\Bigl\{\sum_{i=1}^{s}\pm\omega_n(s)\circ_i\omega_n(t)\Bigr\}\Bigr\} = 0.
\end{equation*}
Since we prove that the homology of $\ker\sigma_*$
vanishes in degree $n(r-1)-2$,
Equation~(\ref{eqn:TwistingKernelElement})
admits a solution $\chi_n(r)\in\ker\sigma_*$
and this achieves the definition of the element $\omega_n(r)$.
\end{proof}

And from Proposition~\ref{FirstStep:Equations:AdjointExpression}
we conclude:

\begin{thm}\label{FirstStep:Conclusion:Result}
We have a sequence of operad morphisms:
\begin{equation*}
\xymatrix{ \BOp^c(\DOp_1)\ar[r]^(0.5){\sigma^*}\ar[d]_{\phi_1} &
\BOp^c(\DOp_2)\ar[r]^(0.65){\sigma^*}\ar@{.>}[d]_{\exists\phi_2} &
\cdots\ar[r]^(0.4){\sigma^*} &
\BOp^c(\DOp_n)\ar[r]^(0.65){\sigma^*}\ar@{.>}[d]_{\exists\phi_n} & \cdots \\
\COp\ar[r]_{=} &
\COp\ar[r]_{=} &
\cdots\ar[r]_{=} &
\COp\ar[r]_{=} & \cdots }.
\end{equation*}
as asserted in Lemma~\ref{FirstStep:Result}.\qed
\end{thm}

Hence, the proof of Lemma~\ref{FirstStep:Result}
is complete.\qed

\medskip
Lemma~\ref{FirstStep:Equations:SuspensionSurjectivity}
can be improved to:

\begin{lemm}\label{FirstStep:Conclusion:SuspensionSection}
The suspension morphism $\sigma: \EOp_n(r)\rightarrow\Lambda^{-1}\EOp_{n-1}(r)$
is surjective
and its kernel forms a projective $\Sigma_r$-module
in all degrees and for every $r\in\NN$.\qed
\end{lemm}

(Use that $\EOp_{n-1}(r)$ forms a free $\Sigma_r$-module in each degree
to define a $\Sigma_r$-equivariant
section of $\sigma$.)

\medskip
We also have:

\begin{lemm}\label{FirstStep:Conclusion:BarrattEcclesDegreeBounds}
Let $n\in\NN$.
The underlying chain complex of $\EOp_n(r)$ is bounded,
for each $r\in\NN$.
\end{lemm}

\begin{proof}
Let $\underline{w} = (w_0,\dots,w_d)$ be a simplex of $\EOp_n(r)$.
Since $(w_0,\dots,w_d)$ is supposed to be non-degenerate,
we have $w_k|_{i j}\not=w_{k+1}|_{i j}$ for some pair $i j$
(otherwise we would have $w_k = w_{k+1}$).
Hence,
the weights of $\underline{w}$ satisfy the relation $d\leq\sum_{i j}\mu_{i j}(\underline{w})$,
from which we deduce the inequality $d<n r (r-1)/2$
since, by definition of $\EOp_n(r)$, we have $\mu_{i j}(\underline{w})<n$ for all pairs $i j$.
\end{proof}

These statements imply:

\begin{prop}\label{FirstStep:Conclusion:SigmaCofibrations}
The morphism $\sigma^*: \DOp_{n-1}\rightarrow\DOp_n$
defines a cofibration of $\Sigma_*$-objects.
\end{prop}

\begin{proof}
Use the characterization of~\cite[\S\S 2.3.6-9]{Hovey}
for cofibrations in a category of dg-modules
over a ring.
\end{proof}

Hence, according to~\cite[Proposition 1.4.13]{FresseCylinder}, we obtain:

\begin{prop}\label{FirstStep:Conclusion:CobarCofibrations}
The morphism $\sigma^*: \BOp^c(\DOp_{n-1})\rightarrow\BOp^c(\DOp_n)$
induced by the suspension morphisms
defines a cofibration of operads.\qed
\end{prop}

From this proposition,
we deduce the existence of a lifting $\tilde{\phi_n}$ in the diagram:
\begin{equation*}
\xymatrix{ \BOp^c(\DOp_1)\ar@{>->}[]!D-<0pt,4pt>;[d]\ar[r] &
\EOp_1\ar@{^{(}->}[]!R+<4pt,0pt>;[r] & \cdots\ar@{^{(}->}[]!R+<4pt,0pt>;[r] &
\EOp_n\ar@{^{(}->}[]!R+<4pt,0pt>;[r] & \cdots\ar@{^{(}->}[]!R+<4pt,0pt>;[r] & \EOp\ar@{->>}[ddd]^{\sim} \\
\ar@{.}[d] &&&&& \\
\ar@{>->}[]!D-<0pt,4pt>;[d] &&&&& \\
\BOp^c(\DOp_n)\ar@{.>}[]!UR;[uuurrr]^{\exists ? \psi_n}
\ar@{.>}[]!UR;[uuurrrrr]_{\exists\tilde{\phi}_n}\ar[rrrrr]_{\phi_n} &&&&& \COp }.
\end{equation*}
But the desired morphism is $\psi_n$.

\subsubsection{Why do not we try to prove the existence of $\psi_n$ by the same method as $\phi_n$?}
The morphisms $\psi_n$
are associated to twisting cochains
\begin{equation*}
\eta_n(r): \DOp_n(r)\rightarrow\EOp_n(r),\ r\geq 2,
\end{equation*}
such that $\iota\eta_{n-1} = \eta_n\sigma^*$,
or equivalently to $\Sigma_r$-coinvariant elements of degree $-1$
\begin{equation*}
\nu_n(r)\in\{\Lambda^n\EOp_n(r)\otimes\EOp_n(r)\}_{\Sigma_r}
\end{equation*}
such that $\id\otimes\iota(\nu_{n-1}(r)) = \sigma\otimes\id(\nu_n(r))$.
The obstructions to the existence of $\nu_n(r)$
are represented by homology classes of degree $-2$
in
\begin{equation*}
H_*(\ker\{(\sigma\otimes\id)_*: \{\Lambda^n\EOp_n(r)\otimes\EOp_{n}(r)\}_{\Sigma_r}
\rightarrow\{\Lambda^{n-1}\EOp_{n-1}(r)\otimes\EOp_n(r)\}_{\Sigma_r}\}).
\end{equation*}
The problem is that this homology does not vanish in degree $-2$.
Hence, obstructions might occur in the construction of our morphism $\psi_n$.

The idea is to refine the lifting argument
in order to go through these obstructions.
For this purpose,
we use cell structures that refine filtration~(\ref{eqn:NestedEnOperads})
of the Barratt-Eccles operad $\EOp$.
The objective of the next subsection is to give a global abstract definition
of these cell structures
in a form suitable for our application.

\medskip
To conclude this section,
observe that propositions~\ref{FirstStep:Conclusion:SigmaCofibrations}-\ref{FirstStep:Conclusion:CobarCofibrations}
give as a corollary:

\begin{mainsecprop}\label{FirstStep:Conclusion:ColimitCofibration}
The cooperad $\DOp_{\infty} = \colim_n\DOp_n$
is cofibrant as a $\Sigma_*$-object
and the associated cobar construction $\BOp^c(\DOp_{\infty})$
is cofibrant as an operad.\qed
\end{mainsecprop}

In the prologue,
we observe that the main results of the paper, theorems~\ref{Prelude:RealizationTheorem}-\ref{Prelude:BarDualityTheorem},
imply that $\BOp^c(\DOp_{\infty})$
forms an $E_\infty$-operad,
but this assertion can already be gained from the results of this section:
we have $H_*(\DOp_{\infty}) = \colim_n H_*(\DOp_n) = \Lambda^{-1}\LOp^{\vee}$
because the verifications of Proposition~\ref{Prelude:GerstenhaberOperads:KoszulPatching}
imply that the morphism $\sigma_*: H(\DOp_n)\rightarrow H_*(\DOp_{n-1})$ reduces to a composite
\begin{equation*}
\Lambda^{-n}\COp^{\vee}\circ\Lambda^{-1}\LOp^{\vee}
\rightarrow\Lambda^{-1}\LOp^{\vee}
\hookrightarrow\Lambda^{1-n}\COp^{\vee}\circ\Lambda^{-1}\LOp^{\vee}
\end{equation*}
when we take the representation of $\GOp_n$ as a composite $\Sigma_*$-object (see~\S\ref{Prelude:GerstenhaberOperads:Definition});
take the colimit $n\rightarrow\infty$
of the morphisms $\phi_n$
of Theorem~\ref{FirstStep:Conclusion:Result} (Lemma~\ref{FirstStep:Result})
to define a morphism $\phi_{\infty}: \BOp^c(\DOp_{\infty})\rightarrow\COp$;
we can check readily that the composite of $\phi_{\infty}$
with the edge morphism of Proposition~\ref{Prelude:KoszulDuality:HomologyKoszulSpectralSequence}
reduces to the Koszul duality equivalence for the commutative operad;
hence,
the argument of Theorem~\ref{Prelude:BarDualityTheorem}
can be applied to prove directly that $\phi_{\infty}$ is a weak-equivalence of operads.

\section{Interlude: operads shaped on complete graph posets}\label{Interlude}

The cell structures we are going to use have been introduced by C. Berger in~\cite{BergerCell,BergerSummary}
as a device to compare nested sequences of operads
to the sequence of little cubes operads.
The existence of a homotopy equivalence between the topological realization of Smith's filtration of the Barratt-Eccles operad
and little cubes operads has arisen as a significant success
of Berger's article.

The cell structure is defined for operads in topological spaces in the original reference.
The input data consists of a topological operad $\POp$
whose term $\POp(2)$ is a cellular model of the infinite dimensional sphere $S^{\infty}$.
Berger's idea was to use restriction operations, included in the composition structure of $\POp$,
in order to derive a cell decomposition of each $\POp(r)$
from the cell structure of $\POp(2)$.
The universal poset $\K(r)$ underlying the decomposition
\begin{equation}\label{eqn:CellUnion}
\POp(r) = \cup_{\kappa\in\K(r)}\POp(\kappa)
\end{equation}
is defined by complete edge-label graphs with $r$ vertices.
Such an operad inherits a natural filtration
by subcomplexes $\POp_n(r) = \cup_{\kappa\in\K_n(r)}\POp(\kappa)$,
where $\K_n(r)$ is a subposet of $\K$
defined by bounding edge-labels.
The existence of a homotopy equivalence between filtration layers $\POp_n$
and little cubes operads $\COp_n$
is ensured when:
\begin{enumerate}
\item\label{Interlude:Intro:Contractible}
the cells $\POp(\kappa)$
are contractible,
\item\label{Interlude:Intro:CellInclusion}
the cell inclusions $\POp(\alpha)\subset\POp(\beta)$
satisfy standard cofibrancy conditions.
\end{enumerate}
The idea of~\cite{BergerCell} is to arrange the situation when bad cells are unessential and decomposition~(\ref{eqn:CellUnion})
reduces to a union of good cells along a poset equivalent to $\K(r)$.
This refinement is necessary to handle the cell structure of the genuine little cubes operads (see~\cite{BergerCell,BrunFiedVogt}).

The applications of this paper demand to introduce an opposite idea,
because we use cell decompositions~(\ref{eqn:CellUnion}) for defining morphisms
and not for determining the homotopy of operads.
For that reason,
we completely revisit definitions of~\cite{BergerCell}.
In summary,
we use that the complete graph posets form an operad in posets, called the complete graph operad $\K$,
and we define a category of diagram operads, called $\K$-operads,
whose objects are diagram sequences $\POp(r) = \{\POp(\kappa)\}_{\kappa\in\K(r)}$, $r\in\NN$,
equipped with a composition structure
shaped on the composition structure of $\K$.
Our main example of a $\K$-operad, before the Barratt-Eccles operad,
is the commutative operad which can be viewed as a constant $\K$-diagram $\COp(\kappa) = \COp(r)$.
Our idea is to interpret condition~(\ref{Interlude:Intro:Contractible})
in terms of a model structure
within the category of $\K$-operads.

The notion of a $\K$-operad makes sense in any surrounding symmetric monoidal category.
This comprises the category of topological spaces, used in~\cite{BergerCell},
but also the category of dg-modules.
For our purpose, we focus on applications to operads in dg-modules.

In~\S\ref{Interlude:CompleteGraphOperad},
we review the definition of the complete graph operad $\K$
and we give the definition of our category of $\K$-operads.
In~\S\ref{Interlude:GraphOperadModel},
we define the model structure of the category of $\K$-operads.

\subsection{Operads shaped on complete graph posets}\label{Interlude:CompleteGraphOperad}
The complete graph operad~$\K$ is defined in~\cite{BergerCell}.
The name refers to a nice representation of the elements of this operad
by complete edge-label graphs.
Since the structure of this operad
is essential for our purpose,
we prefer to recall the definition of~$\K$ first.
We give the definition of a $\K$-operad afterwards.

As we only deal with connected operads,
we will assume $\K(0) = \emptyset$
and this convention differs from~\cite{BergerCell}.

\subsubsection{Complete graph posets}\label{Interlude:CompleteGraphOperad:Poset}
The $r$th term of the complete graph operad $\K$
is the set of pairs $\kappa = (\mu,\sigma)$
where
$\mu = \{\mu_{i j}\}_{i j}$
is a collection of non-negative integers $\mu_{i j}\in\NN$,
indexed by pairs $\{i,j\}\subset\{1,\dots,r\}$,
and $\sigma$ is a permutation
of $\{1,\dots,r\}$.
Recall that $\sigma|_{i j}$ denotes the permutation of $\{i,j\}$
defined by the occurrences of $\{i,j\}$
in the sequence $\sigma = (\sigma(1),\dots,\sigma(r))$.

The pair $\kappa = (\mu,\sigma)$
is represented by an edge-label graph, with $\{1,\dots,r\}$ as a vertex set
and one edge for each pair $\{i,j\}$,
each edge being equipped with a weight, defined by the non-negative integer $\mu_{i j}\in\NN$,
and an orientation, defined by the permutation $\sigma|_{i j}\in\{(i,j),(j,i)\}$.
In the paper, we also refer to the elements $\kappa = (\mu,\sigma)$
as oriented weight systems.

Note that the collection of orientations $\{\sigma|_{i j}\}_{i j}$
is sufficient to determine the permutation $\sigma$,
but only the collections $\{\theta_{i j}\in(i,j),(j,i)\}_{i j}$
which assemble to a global ordering of the set $\{1,\dots,r\}$
are associated to permutations $\sigma\in\Sigma_r$.
Figure~\ref{fig:EdgeLabelGraph} gives an example of a complete edge-label graph
that defines an element of~$\K(r)$ for $r=4$.
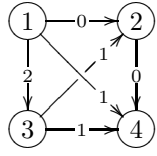
\begin{figure}[t]%
\[\vcenter{\xymatrix@M=0pt@!R=6mm@!C=6mm{ *+<3mm>[o][F]{1}\ar[r]|*+<2pt>{\scriptstyle 0}
\ar[d]|*+<2pt>{\scriptstyle 2}
\ar[dr]|(0.7)*+<2pt>{\scriptstyle 1}
& *+<3mm>[o][F]{2}
\ar[d]|*+<2pt>{\scriptstyle 0} \\
*+<3mm>[o][F]{3}\ar[r]|*+<2pt>{\scriptstyle 1}
\ar[ur]|*+<1mm>{\hole}|(0.7)*+<2pt>{\scriptstyle 1} &
*+<3mm>[o][F]{4} }}\]
\caption{This figure represents the edge-label graph associated to the oriented weight-system $(\mu,\sigma)\in\K(4)$
such that $\mu_{1 2} = \mu_{2 4} = 0$, $\mu_{1 4} = \mu_{2 3} = \mu_{3 4} = 1$, $\mu_{1 3} = 2$
and $\sigma = (1,3,2,4)$.}\label{fig:EdgeLabelGraph}
\end{figure}

The comparison relation $(\mu,\sigma)\leq(\nu,\tau)$ in $\K(r)$
is defined by the requirement
that:
\begin{equation*}
(\mu_{i j}<\nu_{i j})\qquad\text{or}\qquad(\mu_{i j},\sigma|_{i j})=(\nu_{i j},\tau|_{i j}),
\end{equation*}
for all pairs $\{i,j\}\subset\{1,\dots,r\}$.

\subsubsection{The operad structure of complete graph posets}\label{Interlude:CompleteGraphOperad:CompositionStructure}
The symmetric group $\Sigma_r$
acts on $\K(r)$
by poset morphisms.
The element $w\kappa\in\K(r)$ returned by the action of a permutation $w\in\Sigma$
on a pair $\kappa = (\mu,\sigma)\in\K(r)$
is defined by the pair $w\kappa = (w\mu,w\sigma)$,
where $w\mu$ is the collection such that $w\mu_{i j} = \mu_{w^{-1}(i) w^{-1}(j)}$
and $w\sigma$
is the composite of $\sigma$ and $w$
in the symmetric group.
This definition amounts to applying the permutation $w$
to the vertices of the edge-label graph $\kappa$.

The collection of complete graph posets $\K(r)$
is also equipped with partial composition products
\begin{equation*}
\K(s)\times\K(t)\xrightarrow{\circ_e}\K(s+t-1),\quad e = 1,\dots,s,
\end{equation*}
and hence has the full structure of an operad in posets.
To be explicit,
let $\alpha = (\mu,\sigma)\in\K(s)$ and $\beta = (\nu,\tau)\in\K(t)$,
the composite $\alpha\circ_e\beta\in\K(s+t-1)$
is defined by a pair $(\mu\circ_e\nu,\sigma\circ_e\tau)$
such that
\begin{equation*}
(\mu\circ_e\nu)_{i j} = \begin{cases} \mu_{i\,j}, & \text{for $i = 1,\dots,e-1$, $j = 1,\dots,e-1$}, \\
\mu_{i\,e}, & \text{for $i = 1,\dots,e-1$, $j = e,\dots,e+t-1$}, \\
\mu_{i\,j-t+1}, & \text{for $i = 1,\dots,e-1$, $j = e+t,\dots,s+t-1$}, \\
\nu_{i-e+1\,j-e+1}, & \text{for $i = e,\dots,e+t-1$, $j = e,\dots,e+t-1$}, \\
\mu_{e\,j-t+1}, & \text{for $i = e,\dots,e+t-1$, $j = e+t,\dots,s+t-1$},
\end{cases}
\end{equation*}
and $\sigma\circ_e\tau$
is the composite of $\sigma$ and $\tau$
in the permutation operad.

The idea is to plug $\beta$ into the $e$th vertex of $\alpha$.
The edges from a vertex of $\beta$ to the $i$th vertex of $\alpha$
in the composite edge-label graph $\alpha\circ_e\beta$
are just copies of the edge $\{e,i\}$
of the edge-label graph $\alpha$.
The other edges of $\alpha\circ_e\beta$ are copies of the internal edges of $\beta$
and of the edges $\{i,j\}$, $i,j\not=e$,
inside $\alpha$.
We perform the standard operadic index shifts $(1,\dots,t)\mapsto(e,\dots,e+t-1)$ on vertices of $\beta$
and $(e+1,\dots,s)\mapsto(e+t,\dots,s+t-1)$ on vertices of $\alpha$
in order to index vertices of $\alpha\circ_e\beta$
by the set $\{1,\dots,s+t-1\}$.
An example is represented in Figure~\ref{fig:GraphComposite}.
To help understanding, we have marked the substitution array by a dotted frame in the composite graph.
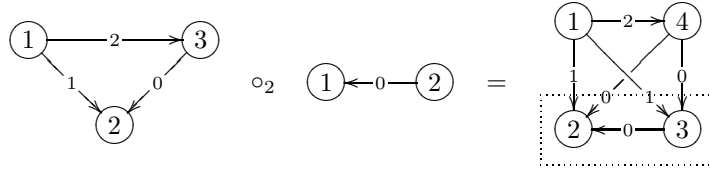
\begin{figure}[t]%
\[\vcenter{\xymatrix@M=0pt@!R=6mm@!C=3mm{ *+<3mm>[o][F]{1}\ar[dr]|*+<2pt>{\scriptstyle 1}\ar[rr]|*+<2pt>{\scriptstyle 2} &&
*+<3mm>[o][F]{3}\ar[dl]|*+<2pt>{\scriptstyle 0} \\
& *+<3mm>[o][F]{2} & }}
\quad\circ_2\quad\vcenter{\xymatrix@M=0pt@!R=6mm@!C=6mm{ *+<3mm>[o][F]{1} & *+<3mm>[o][F]{2}\ar[l]|*+<2pt>{\scriptstyle 0} }}
\quad=\quad\vcenter{\xymatrix@M=0pt@!R=6mm@!C=6mm{ *+<3mm>[o][F]{1}\ar[r]|*+<2pt>{\scriptstyle 2}
\ar[d]|*+<2pt>{\scriptstyle 1}
\ar[dr]|(0.7)*+<2pt>{\scriptstyle 1}
& *+<3mm>[o][F]{4}\ar[dl]|*+<1mm>{\hole}|(0.7)*+<2pt>{\scriptstyle 0}
\ar[d]|*+<2pt>{\scriptstyle 0} \\
*+<3mm>[o][F]{2} &
*+<3mm>[o][F]{3}\ar[l]|*+<2pt>{\scriptstyle 0}\save[]!C.[l]!C *+<4mm>[F.]\frm{}\restore }}\]
\caption{A composite in the complete graph operad}\label{fig:GraphComposite}
\end{figure}

\subsubsection{The category of $\K$-operads}\label{Interlude:CompleteGraphOperad:KOperads}
We call $\K$-diagram of dg-modules
a collection of functors $M: \K(r)\rightarrow\C$
on the posets $\K(r)$, $r\in\NN$.
We use the notation $M(\kappa)$ for the image of an element $\kappa\in\K(r)$ under $M: \K(r)\rightarrow\C$
and the notation $i_*: M(\alpha)\rightarrow M(\beta)$
for the morphism of dg-modules associated to an order relation $\alpha\leq\beta$
in $\K(r)$.

A $\K$-operad is a $\K$-diagram $\POp$
equipped with:
\begin{enumerate}
\item
actions of the symmetric groups $\Sigma_r$, $r\in\NN$,
defined by collections of morphisms
\begin{equation*}
\POp(\kappa)\xrightarrow{w_*}\POp(w\kappa)
\end{equation*}
associated to each permutation $w\in\Sigma_r$,
so that the diagram
\begin{equation*}
\xymatrix{ \POp(\alpha)\ar[r]^{w_*}\ar[d]_{i_*} & \POp(w\alpha)\ar[d]^{i_*} \\
\POp(\beta)\ar[r]_{w_*} & \POp(w\beta) }
\end{equation*}
commutes whenever we have a relation $\alpha\leq\beta$,
\item
and partial composition products
\begin{equation*}
\POp(\alpha)\otimes\POp(\beta)\xrightarrow{\circ_e}\POp(\alpha\circ_e\beta),
\end{equation*}
defined for $e = 1,\dots,s$ when $\alpha\in\K(r)$,
so that the diagram
\begin{equation*}
\xymatrix{ \POp(\alpha)\otimes\POp(\beta)\ar[r]^{\circ_e}\ar[d]_{i_*\otimes i_*} &
\POp(\alpha\circ_e\beta)\ar[d]^{i_*} \\
\POp(\gamma)\otimes\POp(\delta)\ar[r]_{\circ_e} & \POp(\gamma\circ_e\delta) }
\end{equation*}
commutes whenever we have relations $\alpha\leq\gamma$ and $\beta\leq\delta$.
\end{enumerate}
We also assume the existence of a unit element $1\in\POp(1)$
and that a natural extension of the usual axioms of operads
holds for this structure.
We say that a $\K$-operad is connected if we have $\POp(0) = 0$
and $\POp(1) = \ZZ$.

We adopt the notation $\Op_1^{\K}$ for the category of connected $\K$-operads.
A morphism of $\K$-operads $\phi: \POp\rightarrow\QOp$ is obviously a morphism of $\K$-diagrams
which preserves symmetric group actions and composition structures.

\subsubsection{Constant $\K$-operads and colimits}\label{Interlude:CompleteGraphOperad:ConstantKOperads}
Let $\POp\in\Op_1$ be a usual operad.
The constant $\K(r)$-diagrams
\begin{equation*}
\POp(\kappa) = \POp(r),\quad\kappa\in\K(r),
\end{equation*}
inherit an obvious $\K$-operad structure.
Hence
we have a constant diagram functor $\cst: \Op_1\rightarrow\Op_1^{\K}$
from the usual category of operads $\Op_1$
to the category of $\K$-operads $\Op_1^{\K}$.

In the other direction,
any $\K$-operad $\POp$
has an associated usual operad $\colim_{\K}\POp$
defined by the colimits
of its underlying $\K(r)$-diagrams:
\begin{equation*}
(\colim_{\K}\POp)(r) = \colim_{\kappa\in\K(r)}\POp(\kappa).
\end{equation*}
The action of symmetric groups and the composition structure of $\colim_{\K}\POp$
is defined by patching the action of symmetric groups and the composition structure of $\POp$
on colimits.
Hence,
we also have a functor $\colim_{\K}: \Op_1^{\K}\rightarrow\Op_1$
from the category of $\K$-operads $\Op_1^{\K}$
to the usual category of operads $\Op_1$.

In the sequel,
we say that an operad in dg-modules $\POp$ has a $\K$-structure if we have a $\K$-operad $\POp(\kappa)$
such that $\POp(r) = \colim_{\kappa\in\K(r)}\POp(\kappa)$.
Similarly,
we say that an operad morphism $f: \POp\rightarrow\QOp$ is realized by a morphism of $\K$-operads
if the operads $\POp$ and $\QOp$
have a $\K$-structure and $f$ is the colimit
of a given morphism of $\K$-operads.

Throughout the paper,
we take the same notation for the underlying $\K$-operad $\POp(\kappa)$
of an operad $\POp(r) = \colim_{\kappa\in\K(r)}\POp(\kappa)$
equipped with a $\K$-structure.
The letter $\POp$ can denote either the one or the other object.
Usually,
the structure to which we refer is clearly determined
by the context.
Otherwise,
we simply use letters $r,s,t,\dots\in\NN$ and $\alpha,\beta,\kappa,\dots\in\K(r)$
as dummy variables to mark the distinction.

The usual adjunction between colimits and constant diagrams gives:

\begin{prop}\label{Interlude:CompleteGraphOperad:ColimitKOperadAdjunction}
The colimit functor $\colim_{\K}: \Op_1^{\K}\rightarrow\Op_1$
is left adjoint to the constant functor $\cst: \Op_1\rightarrow\Op_1^{\K}$.
\end{prop}

\begin{proof}
Straightforward: check that the augmentation $\epsilon: \colim_{\K(r)}(\POp(r))\rightarrow\POp(r)$
and the unit $\eta: \POp(\kappa)\rightarrow\colim_{\kappa\in\K(r)}\POp(\kappa)$
yielded by the usual adjunction of colimits preserve operad structures.
\end{proof}

\subsubsection{The nested sequence associated to a $\K$-operad}\label{Interlude:CompleteGraphOperad:NestedSuboperads}
The complete graph operad
has a nested sequence of suboperads
\begin{equation*}
\K_1\subset\dots\subset\K_n\subset\dots\subset\colim_n\KOp_n = \K
\end{equation*}
whose terms $\K_n(r)$ consist of oriented weight-systems $\kappa = (\mu,\sigma)\in\K(r)$
satisfying $\max_{i j}(\mu_{i j})<n$.
By convention, we also set $\K_{\infty} = \K$.

The construction of~\S\ref{Interlude:CompleteGraphOperad:ConstantKOperads}
can be applied to produce, from the structure of a $\K$-operad $\POp$,
a sequence of operads
\begin{equation*}
\colim_{\K_1}\POp\rightarrow\dots\rightarrow\colim_{\K_n}\POp\rightarrow\dots\rightarrow\colim_n\{\colim_{\K_n}\POp\} = \colim_{\K}\POp
\end{equation*}
so that $(\colim_{\K_n}\POp)(r) = \colim_{\kappa\in\K_n(r)}\POp(\kappa)$.
The morphisms $\colim_{\K_{n-1}}\POp\rightarrow\colim_{\K_{n}}\POp$
are yielded by the poset embeddings $\K_{n-1}(r)\subset\K_{n}(r)$.

\medskip
In the next paragraph,
we revisit the definition of the filtration layers of the Barratt-Eccles operad
in order to prove that they arise from a $\K$-operad structure.
In this case,
the identity $\EOp_n(r) = \colim_{\kappa\in\K_n(r)}\EOp(\kappa)$
is an abstract reformulation of an observation of~\cite{BergerCell}.

\subsubsection{The example of the Barratt-Eccles operad}\label{Interlude:CompleteGraphOperad:BarrattEccles}
Recall that we associate to each simplex of permutations $\underline{w} = (w_0,\dots,w_d)\in\Sigma_r^{\times d+1}$
a collection of weights $\mu(\underline{w}) = \{\mu_{i j}(\underline{w})\}_{i j}$
defined by the variation numbers of the sequences $\underline{w}|_{i j} = (w_0|_{i j},\dots,w_d|_{i j})$.
Let $\kappa(\underline{w}) = (\mu(\underline{w}),\sigma(\underline{w}))$
be the element of $\K(r)$
defined by this weight collection $\mu(\underline{w})$
and by the last permutation $\sigma(\underline{w}) = w_d$
of the simplex $\underline{w} = (w_0,\dots,w_d)$.

To each $\kappa\in\K(r)$
we associate the module $\EOp(\kappa)\subset\EOp(r)$ spanned by the simplices $\underline{w}$
such that $\kappa(\underline{w})\leq\kappa$.

In~\cite{BergerCell},
an analogous collection of subobjects $\WOp(\kappa)$
is defined at the level of the simplicial Barratt-Eccles operad $\WOp$.
In this original article,
the subobjects $\WOp(\kappa)$ are defined from the skeletal filtration of $\WOp(2)$
by an intersection
\begin{equation*}
\WOp(\kappa) = \bigcap_{i j} r_{i j}^{-1}(\sk_{\mu_{i j}}\WOp(2)),
\end{equation*}
where $r_{i j}: \WOp(\kappa)\rightarrow\WOp(2)$
is a restriction operation deduced from the operad structure of $\WOp$.
The module $\EOp(\kappa)$
is just the submodule of $\EOp(r) = N_*(\WOp(r))$
spanned by the non-degenerate simplices of~$\WOp(\kappa)\subset\WOp(r)$
(we use this observation next, in~\S\ref{SecondStep:DualModuleKStructure:Definition}).

The next statement can be proved by an easy inspection
of definitions in \emph{loc. cit.}:

\begin{obsv}[see~\cite{BergerCell} and~\cite{BergerFresse}]\label{Interlude:CompleteGraphOperad:BarrattEcclesWeightRelations}\hspace*{2mm}
\begin{enumerate}
\item
For any simplex $\underline{w} = (w_0,\dots,w_d)$
and $k = 0,\dots,d$,
we have the relation
\begin{equation*}
\kappa(w_0,\dots,\widehat{w_k},\dots,w_d)\leq\kappa(w_0,\dots,w_d).
\end{equation*}
\item
For any simplex $\underline{w}\in\EOp(r)$ and any permutation $s\in\Sigma_r$,
we have the relations $\mu_{i j}(s\cdot\underline{w}) = \mu_{s^{-1}(i) s^{-1}(j)}(\underline{w})$, $\forall i j$,
and $\sigma(s\cdot\underline{w}) = s\cdot\sigma(\underline{w})$,
from which we deduce
\begin{equation*}
\kappa(s\cdot\underline{w}) = s\cdot\kappa(\underline{w}),
\end{equation*}
where $s\cdot\kappa(\underline{w})$
refers to the action of $s$ on the oriented weight-system associated to $\underline{w}$.
\item
For any pair of simplices $\underline{u}\in\EOp(s)$ and $\underline{v}\in\EOp(t)$ and any $e = 1,\dots,r$,
we have the identity
\begin{equation*}
\kappa(\underline{u}\circ_e\underline{v}) = \kappa(\underline{u})\circ_e\kappa(\underline{v}).
\end{equation*}
\end{enumerate}
\end{obsv}

A first application of these relations gives:

\begin{lemm}[see~\cite{BergerCell} and~\cite{BergerFresse}]\label{Interlude:CompleteGraphOperad:BarrattEcclesKStructureVerifications}\hspace*{2mm}
\begin{enumerate}
\item
The modules $\EOp(\kappa)\subset\EOp(r)$ are preserved by the differential of the Barratt-Eccles operad $\delta: \EOp(r)\rightarrow\EOp(r)$
and form a collection of dg-submodules of~$\EOp(r)$.
\item
The morphism $w: \EOp(r)\rightarrow\EOp(r)$ defined by the action of a permutation $w\in\Sigma_r$ on $\EOp(r)$
maps the submodule $\EOp(\kappa)\subset\EOp(r)$
into $\EOp(w\kappa)\subset\EOp(r)$,
for all $\kappa\in\K(r)$.
\item
The partial composition products $\circ_e: \EOp(s)\otimes\EOp(t)\rightarrow\EOp(s+t-1)$
maps the submodule $\EOp(\alpha)\otimes\EOp(\beta)\subset\EOp(s)\otimes\EOp(t)$
into $\EOp(\alpha\circ_e\beta)\subset\EOp(s+t-1)$,
for all $\alpha\in\K(s)$, $\beta\in\K(t)$.\qed
\end{enumerate}
\end{lemm}

From which we deduce:

\begin{prop}\label{Interlude:CompleteGraphOperad:BarrattEcclesKStructure}
The collection of dg-modules $\{\EOp(\kappa)\}_{\kappa\in\K(r)}$, $r\in\NN$,
inherits a $\K$-operad structure.\qed
\end{prop}

We have moreover:

\begin{prop}\label{Interlude:CompleteGraphOperad:BarrattEcclesLayers}
We have a natural isomorphism
of operads
\begin{equation*}
\colim_{\kappa\in\K_n(r)}\EOp(\kappa)\xrightarrow{\simeq}\EOp_n(r),
\end{equation*}
for all $n$, including $n = \infty$.
\end{prop}

\begin{proof}
We have clearly $\EOp(\kappa)\subset\EOp_n(r)$, for all $\kappa\in\K_n(r)$,
and we have by definition $\underline{w}\in\EOp(\kappa)$
if and only if $\kappa(\underline{w})\leq\kappa$, for all $\underline{w}\in\EOp(r)$.
We have a map $\EOp_n(r)\rightarrow\colim_{\kappa\in\K_n(r)}\EOp(\kappa)$
sending a basis element $\underline{w}\in\EOp_n(r)$
to the same element in the summand $\EOp_n(\kappa(\underline{w}))$
of the colimit $\colim_{\kappa\in\K_n(r)}\EOp(\kappa)$.
It is easy to check that this mapping gives an inverse bijection
of the natural morphism $\colim_{\kappa\in\K_n(r)}\EOp(\kappa)\rightarrow\EOp_n(r)$
yielded by the embeddings $\EOp(\kappa)\subset\EOp_n(r)$.
\end{proof}

In the case $n = \infty$,
the proposition asserts that the colimit $\colim_{\kappa\in\K(r)}\EOp(\kappa)$
is isomorphic to $\EOp(r) = \EOp_{\infty}(r)$.
Hence,
we obtain that the Barratt-Eccles operad $\EOp$ comes equipped with a $\K$-structure
such that the operads of~\S\ref{Interlude:CompleteGraphOperad:NestedSuboperads}
\begin{equation*}
\colim_{\K_1}\EOp\rightarrow\dots\rightarrow\colim_{\K_n}\EOp\rightarrow\dots\rightarrow\colim_n\{\colim_{\K_n}\EOp\} = \colim_{\K}\EOp
\end{equation*}
are identified with the layers $\EOp_n\subset\EOp$
of the filtration of~\S\ref{Prelude:BarrattEccles:LittleCubesFiltration}.

The simplicial analogue of the identity $\EOp_n(r) = \colim_{\kappa\in\K_n(r)}\EOp(\kappa)$
is used in the definition of the homotopy equivalence between the filtration layers of the Barratt-Eccles operad
and the operads of little $n$-cubes.
The proof of the existence of these homotopy equivalences in~\cite{BergerCell}
involves a homotopical study of the $\K$-diagram
underlying the Barratt-Eccles operad.

For our purpose, we just record:

\begin{fact}[see~\cite{BergerCell,BergerFresse}]\label{Interlude:CompleteGraphOperad:BarrattEcclesAcyclicity}
Each dg-module $\EOp(\kappa)\subset\EOp(r)$, $\kappa\in\K(r)$, is contractible so that the augmentation $\epsilon: \EOp(r)\xrightarrow{\sim}\ZZ$
restricts to a weak-equivalence on~$\EOp(\kappa)$.
\end{fact}

\subsection{Model structures}\label{Interlude:GraphOperadModel}
We apply a usual process to provide the category of $\K$-operads
with a model structure.
We study first a category of symmetric $\K$-diagrams $\C_1^{\K\Sigma_*}$
which only retain the action of symmetric groups and the $\K$-diagram structure of a connected $\K$-operad.
We have an obvious forgetful functor $U: \Op_1^{\K}\rightarrow\C_1^{\K\Sigma_*}$.
We adapt the definition of the free operad to prove that this functor has a left adjoint $F: \C_1^{\K\Sigma_*}\rightarrow\Op_1^{\K}$.
We check that symmetric $\K$-diagrams form a cofibrantly generated model category
and we use the adjunction $F: \C_1^{\K\Sigma_*}\rightleftarrows\Op_1^{\K} :U$
to transport this model structure to the category of $\K$-operads $\Op_1^{\K}$.

\subsubsection{The category of symmetric $\K$-diagrams}\label{Interlude:GraphOperadModel:SymmetricKDiagrams}
A symmetric $\K$-diagram is just a sequence of $\K(r)$-diagrams $\{M(\kappa)\}_{\kappa\in\K(r)}$, $r\in\NN$,
together with symmetric group actions defined by collections of morphisms
\begin{equation*}
M(\kappa)\xrightarrow{w_*} M(w\kappa)
\end{equation*}
associated to each permutation $w\in\Sigma_r$,
so that the diagram of~\S\ref{Interlude:CompleteGraphOperad:KOperads}
commutes.
A morphism of symmetric $\K$-diagrams $f: M\rightarrow N$
is a collection of $\K(r)$-diagram morphisms $f: M(\kappa)\rightarrow N(\kappa)$
preserving symmetric group actions.

Let $\K\Sigma_*$ be the category formed by oriented weight-systems $\kappa\in\K(r)$
as objects
and the composites of order relations $\alpha\xrightarrow{\leq}\beta$
and permutations $\kappa\xrightarrow{w_*} w\kappa$
as morphisms
with the convention that the diagram
\begin{equation*}
\xymatrix{ \alpha\ar[r]^{w_*}\ar[d]_{\leq} & w\alpha\ar[d]^{\leq} \\ \beta\ar[r]_{w_*} & w\beta }
\end{equation*}
commutes in $\K\Sigma_*$.
The next assertion
is obvious from this definition:

\begin{fact}\label{Interlude:GraphOperadModel:SymmetricKDiagramsAsFunctors}
The category of symmetric $\K$-diagrams
is isomorphic to the category of functors $M: \K\Sigma_*\rightarrow\C$.
\end{fact}

We adopt the notation $\C_1^{\K\Sigma_*}$ for the category of symmetric $\K$-diagrams such that $M(0) = M(1) = 0$.
We only consider symmetric $\K$-diagrams of this category in the sequel.

By~\cite[Theorem 11.6.1]{Hirschhorn},
any category of functors $F: \I\rightarrow\C$
from a small category $\I$
toward a cofibrantly generated model category $\C$
inherits a cofibrantly generated model structure.
In the case of symmetric $\K$-diagrams,
this statement returns:

\begin{prop}\label{Interlude:GraphOperadModel:SymmetricKDiagramModel}
The category of symmetric $\K$-diagrams in dg-modules
inherit a model structure so that:
\begin{itemize}
\item
the weak-equivalences, respectively fibrations, are the morphisms of symmetric $\K$-diagrams $f: M\rightarrow N$
such that $f(\kappa): M(\kappa)\rightarrow N(\kappa)$ forms a weak-equivalence, respectively a fibration,
of dg-modules, for all $\kappa\in\K$;
\item
the cofibrations are the morphisms of symmetric $\K$-diagrams
which have the right lifting property with respect to acyclic fibrations.
\end{itemize}
This model category also inherits a set of generating (acyclic) cofibrations.\qed
\end{prop}

\subsubsection{Latching objects}\label{Interlude:GraphOperadModel:LatchingObjects}
The cofibrations of symmetric $\K$-diagrams can be characterized effectively as retracts of relative cell complexes,
as in any cofibrantly generated category.
But, in this paper,
we use another characterization of cofibrations of symmetric $\K$-diagrams
which arises from a generalization of the notion of a Reedy cofibration
to categories with isomorphisms.

Any symmetric $\K$-diagram $M$ has a latching object $LM$
defined by the collections of dg-modules
\begin{equation*}
LM(\kappa) = \colim_{\alpha\lneqq\kappa} M(\alpha)
\end{equation*}
and the morphisms $M(\alpha)\rightarrow M(\kappa)$
assemble to a natural latching morphism $\lambda: LM(\kappa)\rightarrow M(\kappa)$,
for each $\kappa$.
Observe that $LM$
inherits symmetric group actions
and structure morphisms $i_*: LM(\alpha)\rightarrow LM(\beta)$, for every relation $\alpha\leq\beta$,
so that $LM$ forms itself a symmetric $\K$-diagram and the latching morphisms $\lambda: LM(\kappa)\rightarrow M(\kappa)$
define a morphism of symmetric $\K$-diagrams $\lambda: LM\rightarrow M$.

For a morphism of symmetric $\K$-diagrams $f: M\rightarrow N$,
we form the pushout morphism
\begin{equation*}
\xymatrix{ LM\ar[r]^{Lf}\ar[d]_{\lambda} & LN\ar@{.>}[d]\ar@/^5mm/[ddr]^{\lambda} & \\
M\ar@{.>}[r]\ar@/_5mm/[rrd]_{f} & M\bigoplus_{LM} LN\ar@{.>}[dr]_(0.4){(f,\lambda)} & \\
&& N }.
\end{equation*}
We have:

\begin{prop}\label{Interlude:GraphOperadModel:SymmetricKDiagramCofibrations}
Let $i: M\rightarrow N$
be a morphism of symmetric $\K$-diagrams.
If the pushout morphisms
\begin{equation*}
(i,\lambda): M(\kappa)\bigoplus_{L M(\kappa)} L N(\kappa)\rightarrow N(\kappa)
\end{equation*}
are cofibrations of dg-modules for all $\kappa\in\K(r)$ and all $r\in\NN$,
then $i$ is a cofibration of symmetric $\K$-diagrams.
\end{prop}

\begin{proof}
We prove that $i$ has the right-lifting-property with respect
to acyclic fibrations:
\begin{equation}\label{eqn:ReedyLifting}
\xymatrix{ M\ar[d]_{i}\ar[r]^{f} & P\ar@{->>}[d]_{\sim}^{p} \\
N\ar[r]_{g}\ar@{.>}[ur]^{h} & Q }.
\end{equation}

Define the degree of an oriented weight-system $\kappa = (\mu,\sigma)$
by $\deg(\kappa) = \sum_{i j}\mu_{i j}$.
Observe that $\deg(w\kappa) = \deg(\kappa)$, for every permutation $w$,
and $\alpha\lneqq\beta$
implies $\deg(\alpha)<\deg(\beta)$.

Let $m\in\NN$.
Suppose we have morphisms $h: N(\alpha)\rightarrow P(\alpha)$
commuting with the action of permutations and with the structure morphisms of $\K$-diagrams
for every $\alpha$ such that $\deg(\alpha)<m$.

Let $\kappa$ be an element of the form $\kappa = (\mu,\sigma)$,
where $\sigma = \id$
is the identity permutation
and $\deg(\kappa) = m$.
Since $\alpha\lneqq\kappa$ implies $\deg(\alpha)<\deg(\kappa) = m$,
the already defined morphisms $h: N(\kappa)\rightarrow P(\kappa)$
assemble to a morphism $\lambda\cdot Lh: LN(\kappa)\rightarrow P(\kappa)$
on the latching object $LN$
so that we have a commutative diagram:
\begin{equation*}
\xymatrix{ M(\kappa)\bigoplus_{LM(\kappa)} LN(\kappa)\ar[d]_{(i,\lambda)}\ar[rr]^(0.6){(f,\lambda\cdot Lh)} &&
P(\kappa)\ar@{->>}[d]_{\sim}^{p} \\
N(\kappa)\ar[rr]_{g}\ar@{.>}[urr]^{h} && Q(\kappa) }.
\end{equation*}
The right lifting property in dg-modules implies the existence of a fill-in morphism
$h: N(\kappa)\rightarrow P(\kappa)$
that makes this diagram commute.
The commutativity of the diagram implies that this morphism commutes with the structure morphisms of $\K$-diagrams
associated to relations $\alpha\lneqq\kappa$.

By equivariance,
we have a whole collection of morphisms $h: N(\kappa)\rightarrow P(\kappa)$
commuting with the action of permutations,
for all pairs $\kappa = (\mu,\sigma)$ such that $\deg(\mu) = m$.
These morphisms $h: N(\mu,\sigma)\rightarrow P(\mu,\sigma)$
still commute with the structure morphisms of $\K$-diagrams.

By induction on $m = \deg(\mu)$,
we obtain a whole collection of morphisms $h: N(\mu,\sigma)\rightarrow P(\mu,\sigma)$
commuting with the structure morphisms of symmetric $\K$-diagrams.
Hence,
we have a morphism of symmetric $\K$-diagrams $h: N\rightarrow P$
which solves the lifting problem~(\ref{eqn:ReedyLifting}).

This construction proves that the morphism $i$ has the right lifting property with respect to acyclic fibrations
and hence forms a cofibration in the category of symmetric $\K$-diagrams.
\end{proof}

\subsubsection{The construction of free $\K$-operads}\label{Interlude:GraphOperadModel:FreeKOperad}
We adapt the construction of~\S\ref{Prelude:KoszulDuality:FreeOperad}
to define a free object functor $F: \C_1^{\K\Sigma_*}\rightarrow\Op_1^{\K}$
left adjoint to the forgetful functor $U: \Op_1^{\K}\rightarrow\C_1^{\K\Sigma_*}$.

First,
the construction of the free operad can be applied to the complete graph operad $\K$
and returns an operad in posets $\FOp(\K)$
together with an operad morphism $\lambda_*: \FOp(\K)\rightarrow\K$.
In this context,
the summand $\tau(\K)$
of the expansion of the free operad
\begin{equation*}
\FOp(\K)(r) = \coprod_{\tau\in\Theta(r)}\tau(\K)/\equiv,
\end{equation*}
is the set of collections $\alpha_* = \{\alpha_v\}_{v\in V(\tau)}$,
where each $\alpha_v\in\K(r_v)$ is an oriented weight-system associated to a vertex $v\in V(\tau)$
together with a bijection between $\{1,\dots,r_v\}$
and the ingoing edges of $v$.
Each summand $\tau(\K)$, defined by a cartesian product of posets $\K(r_v)$,
has an internal poset structure
and elements of $\FOp(\K)$ are comparable only if they belong to the same summand $\tau(\K)$
(up to tree isomorphisms).

Let $M$ be a symmetric $\K$-diagram.
To a tree $\tau\in\Theta(r)$
and a collection $\alpha_* = \{\alpha_v\}_{v\in V(\tau)}\in\tau(\K)$,
we associate the dg-module
\begin{equation*}
\tau(M,\alpha_*) = \bigotimes_{v\in V(\tau)} M(\alpha_v).
\end{equation*}
For comparable elements $\alpha_*\leq\beta_*$
we have a morphism $i_*: \tau(M,\alpha_*)\rightarrow\tau(M,\beta_*)$
defined by the tensor product of the morphisms $i_*: M(\alpha_v)\rightarrow M(\beta_v)$
associated to the relations $\alpha_v\leq\beta_v$.
Hence,
the collection $\{\tau(M,\alpha_*)\}_{\alpha_*\in\tau(\K)}$
defines a functor on the poset $\tau(\K)$.

To each oriented weight-system $\kappa\in\K(r)$,
we associate the dg-module
\begin{equation*}
\tau(M)(\kappa) = \colim_{\lambda_*(\alpha_*)\leq\kappa} \tau(M,\alpha_*),
\end{equation*}
where the colimit ranges over the subposet of collections $\alpha_*\in\tau(\K)$
such that~$\lambda_*(\alpha_*)\leq\kappa$.
This construction is clearly functorial in $\tau\in\Theta(r)$.
The free $\K$-operad $\FOp(M)$
is defined at $\kappa\in\K(r)$
by the sum
\begin{equation*}
\FOp(M)(\kappa) = \bigoplus_{\tau\in\Theta(r)}\tau(M)(\kappa)/\equiv
\end{equation*}
divided out by the action of automorphisms.
Again~(see~\S\ref{Prelude:KoszulDuality:FreeOperad}),
the assumption $M(0) = M(1) = 0$
implies that we can restrict the expansion of $\FOp(M)(r)$
to the subcategory of reduced trees,
and by triviality of automorphism groups of reduced trees,
we can rewrite this expansion in a reduced form in which no quotient occurs.

For the tree with one vertex of~\S\ref{Prelude:KoszulDuality:FreeOperad},
we have an isomorphism $M(\kappa)\simeq\psi(M)(\kappa)$,
for each $\kappa\in\K(r)$,
and the identity of $M(\kappa)$
with this summand yields a natural morphism
\begin{equation*}
\eta: M(\kappa)\rightarrow\FOp(M)(\kappa).
\end{equation*}

Our purpose is to check:

\begin{prop}\label{Interlude:GraphOperadModel:FreeOperadStructure}
The collection of dg-modules $\{\FOp(M)(\kappa)\}_{\kappa\in\K(r)}$, $r\in\NN$, inherits the structure of a $\K$-operad
and represents the free $\K$-operad associated to $M$
together with the universal morphism $\eta: M\rightarrow\FOp(M)$
defined in~\S\ref{Interlude:GraphOperadModel:FreeKOperad}.
\end{prop}

\begin{proof}
The operad structure of $\FOp(M)$
is defined by a natural extension of the constructions of~\S\ref{Prelude:KoszulDuality:FreeOperad}.

First,
for comparable elements $\alpha\leq\beta$,
the relation $\lambda_*(\gamma_*)\leq\alpha\leq\beta$
implies that each summand of $\tau(M)(\alpha)$
defines a summand of $\tau(M)(\beta)$.
Hence,
we have morphisms $i_*: \tau(M)(\alpha)\rightarrow\tau(M)(\beta)$, for each $\tau\in\Theta(r)$,
and an induced morphism $i_*: \FOp(M)(\alpha)\rightarrow\FOp(M)(\beta)$
at the level of $\FOp(M)$,
so that $\FOp(M)$ inherits the structure of a $\K$-diagram.

Recall that the action of a permutation on a tree reduces to a reindexing of the ingoing edges.
For any permutation $w\in\Sigma_r$,
we have a natural isomorphism $w_*: \tau(M,\alpha_*)\xrightarrow{\simeq}(w\tau)(M,w_*(\alpha_*))$,
where $w_*(\alpha_*)\in(w\tau)(\K)$
is the same as $\{\alpha_v\}_{v\in V(\tau)}$
in $V(w\tau) = V(\tau)$.
Since $\lambda_*(\alpha_*)\leq\kappa$ implies $\lambda_*(w_*(\alpha_*))\leq w\kappa$,
these isomorphisms yield an isomorphism $w_*: \tau(M)(\kappa)\xrightarrow{\simeq}(w\tau)(M)(w\kappa)$ on each $\tau(M)$,
and an isomorphism $w_*: \FOp(M)(\kappa)\rightarrow\FOp(M)(w\kappa)$
at the level of the $\K$-diagram $\FOp(M)$.
Hence,
we obtain that $\FOp(M)$ inherits the structure of a symmetric $\K$-diagram.

For trees $\sigma\in\Theta(s),\tau\in\Theta(t)$
and collections $\alpha_*\in\sigma(\K),\beta_*\in\tau(\K)$,
we have a natural isomorphism
$\sigma(M,\alpha_*)\otimes\tau(M,\beta_*)\simeq(\sigma\circ_e\tau)(M,\alpha_*\circ_e\beta_*)$,
where $\alpha_*\circ_e\beta_*\in\sigma\circ_e\tau(\K)$
represents the composite of $\alpha_*$ and $\beta_*$ in the free operad $\FOp(\K)$.
Since $\lambda_*(\alpha_*)\leq\gamma$ and $\lambda_*(\beta_*)\leq\delta$
implies $\lambda_*(\alpha_*\circ_e\beta_*)\leq\gamma\circ_e\delta$,
these isomorphisms assemble to an isomorphism
$\sigma(M)(\gamma)\otimes\tau(M)(\delta)\simeq\sigma\circ_e\tau(M)(\gamma\circ_e\delta)$,
for each $\gamma\in\K(s)$ and $\delta\in\K(t)$,
where $\gamma\circ_e\delta\in\K(s+t-1)$
represents the composite of $\gamma$ and $\delta$ in $\K$,
from which we deduce the existence of morphisms
$\circ_e: \FOp(M)(\gamma)\otimes\FOp(M)(\delta)\rightarrow\FOp(M)(\gamma\circ_e\delta)$
which provide $\FOp(M)$
with the composition structure of a $\K$-operad.

The proof that $\FOp(M)$ represents the free $\K$-operad associated to $M$
follows from a straightforward generalization of the case of usual operads
for which we refer to~\cite{GetzlerJones,GinzburgKapranov}
(see also~\cite[\S\S 1.2.4-1.2.10]{FresseCylinder}).
\end{proof}

The forgetful functor $U: \Op_1^{\K}\rightarrow\C_1^{\K\Sigma_*}$
creates limits in the category of $\K$-operads.
The forgetful functor creates coequalizers and filtered colimits as well,
but not all colimits.
Nevertheless,
any colimit can be identified with a reflexive coequalizer of free $\K$-operads.
Hence,
we obtain that the category of $\K$-operads
has all colimits.

We have moreover:

\begin{thm}\label{Interlude:GraphOperadModel:KOperadModel}
The category of connected $\K$-operads in dg-modules
inherits a model structure
so that:
\begin{itemize}
\item
a morphism $f: \POp\rightarrow\QOp$
is a weak-equivalence, respectively a fibration, of $\K$-operads
if $f$ defines a weak-equivalence, respectively a fibration, of symmetric $\K$-diagrams
(we say that the forgetful functor creates weak-equivalences and fibrations);
\item
the cofibrations of $\K$-operads
are the morphisms which have the right-lifting-property with respect to acyclic fibrations.
\end{itemize}
This model category also inherits a set of generating (acyclic) cofibrations
defined by the morphism of free $\K$-operads $\FOp(i): \FOp(M)\rightarrow\FOp(N)$
such that $i: M\rightarrow N$ runs over the generating (acyclic) cofibrations
of the category of symmetric $\K$-diagrams.\qed
\end{thm}

\begin{proof}
Straightforward generalization of the analysis of~\cite{HinichHomotopy}.
The elegant argument of~\cite{BergerMoerdijk}
can also be extended to the context of $\K$-operads.
\end{proof}

The assertion of Fact~\ref{Interlude:CompleteGraphOperad:BarrattEcclesAcyclicity}
has the following interpretation:

\begin{prop}\label{Interlude:GraphOperadModel:BarrattEcclesAugmentation}
The augmentation morphism of the Barratt-Eccles operad $\epsilon: \EOp\rightarrow\COp$
defines an acyclic fibration
of $\K$-operads
when we equip $\EOp$ with the $\K$-structure
of~\S\S\ref{Interlude:CompleteGraphOperad:BarrattEccles}-\ref{Interlude:CompleteGraphOperad:BarrattEcclesLayers}
and the commutative operad $\COp$ is viewed as a constant $\K$-operad.\qed
\end{prop}

One proves further:

\begin{prop}\label{Interlude:GraphOperadModel:BarrattEcclesCofibrancy}
The Barratt-Eccles operad, though not cofibrant in the category of $\K$-operads,
is cofibrant as a symmetric $\K$-diagram.\qed
\end{prop}

Proposition~\ref{Interlude:GraphOperadModel:BarrattEcclesCofibrancy}
is stated as a remark and will not be proved in this article.

\medskip
The arguments of~\cite{BergerCell}
can be adapted to prove that the nested sequence of~\S\ref{Interlude:CompleteGraphOperad:NestedSuboperads}
associated to any $\K$-operad
satisfying propositions~\ref{Interlude:GraphOperadModel:BarrattEcclesAugmentation}-\ref{Interlude:GraphOperadModel:BarrattEcclesCofibrancy}
is equivalent to the sequence of the chain little cubes operads
as a sequence of operads in dg-modules.
But we do not use $\K$-structures that way.
Instead, we are going to use structures,
like constant $\K$-operads,
for which Proposition~\ref{Interlude:GraphOperadModel:BarrattEcclesCofibrancy} fails.

\subsubsection{Quasi-free objects}\label{Interlude:GraphOperadModel:QuasiFreeKOperads}
The last purpose of this subsection is to give an effective construction of cofibrations in the category of $\K$-operads.
For this aim, we use a natural generalization
of the notion of a quasi-free object
in the context of $\K$-operads.
The structure of a quasi-free $\K$-operad is defined explicitly by a pair $\POp = (\FOp(M),\partial)$,
where $\FOp(M)$ is a free $\K$-operad and $\partial$ is a derivation of the free $\K$-operad $\FOp(M)$
which is added to the natural differential of~$\FOp(M)$
to produce the differential of $\POp$.
This derivation $\partial$ consists of a collection of homomorphisms of degree $-1$
\begin{equation*}
\partial: \FOp(M)(\kappa)\rightarrow\FOp(M)(\kappa),\quad\kappa\in\K(r),
\end{equation*}
commuting with the structure morphisms $i_*: \FOp(M)(\alpha)\rightarrow\FOp(M)(\beta)$,
for every pair $\alpha\leq\beta$,
with the action of permutations $w_*: \FOp(M)(\kappa)\rightarrow\FOp(M)(w\kappa)$
and so that we have the derivation relation $\partial(p\circ_e q) = \partial(p)\circ_e q+\pm p\circ_e\partial(q)$
for every composite in the free operad $\FOp(M)$.
The sum $\delta+\partial$
gives a well defined differential on $\FOp(M)$
if and only if the derivation satisfies the relation $\delta(\partial) + \partial^2 = 0$,
where $\delta(-)$ refers to the differential of homomorphisms in the dg-modules $\Hom_{\C}(\FOp(M)(\kappa),\FOp(M)(\kappa))$.

The derivation relation implies that $\partial$
is uniquely determined by its restriction to the generating symmetric $\K$-diagram $M\subset\FOp(M)$.

The free $\K$-operad associated to a symmetric $\K$-diagram
inherits a splitting $\FOp(M) = \bigoplus_{m=0}^{\infty}\FOp_{m}(M)$
like the usual free operad of $\Sigma_*$-objects.
Moreover,
we have $\FOp_{0}(M) = I$, $\FOp_{1}(M) = M$
and $\bar{\FOp}(M) = \bigoplus_{m=1}^{\infty}\FOp_{m}(M)$
represents the augmentation ideal of $\FOp(M)$.
In general (see explanations in~\cite[\S 1.4.9]{FresseCylinder}),
we assume that the restriction $\partial|_{M}: M\rightarrow\FOp(M)$
satisfies the relation $\partial(M)\subset\bigoplus_{m\geq 2}\FOp_{m}(M)$.

\subsubsection{Morphism on quasi-free operads}\label{Interlude:GraphOperadModel:QuasiFreeMorphisms}
For a quasi-free $\K$-operad $\POp = (\FOp(M),\partial)$,
a morphism of $\K$-operads $\phi: (\FOp(M),\partial)\rightarrow\QOp$
is uniquely determined by its restriction to $M\subset\FOp(M)$,
just like the derivation $\partial$.
In fact,
any homomorphism of degree $0$
\begin{equation*}
f: M\rightarrow\QOp
\end{equation*}
gives rise to a unique homomorphism $\phi_f: \FOp(M)\rightarrow\POp$
commuting with composition structures.
This homomorphism defines a morphism of $\K$-operads $\phi_f: (\FOp(M),\partial)\rightarrow\QOp$
if and only if we have the commutation relation $\delta\phi_f = \phi_f\delta + \phi_f\partial$
with respect to the differential of $\POp = (\FOp(M),\partial)$
and the internal differential of $\QOp$.

In the next section,
we consider morphisms of quasi-free $\K$-operads
\begin{equation*}
\phi_f: (\FOp(M),\partial)\rightarrow(\FOp(N),\partial)
\end{equation*}
induced by morphisms of symmetric $\K$-diagrams
\begin{equation*}
M\xrightarrow{f} N\subset\FOp(N).
\end{equation*}
In that case,
we have an identity $\phi_f = \FOp(f)$, where $\FOp(f)$ is the morphism of free $\K$-operads associated to $f$,
and the commutation of $\phi_f$ with differentials amounts to the relation $\partial\phi_f = \phi_f\partial$,
because $\phi_f = \FOp(f)$
commutes automatically with the internal differential of free $\K$-operads
when $f$ is a genuine morphism of symmetric $\K$-diagrams.

\subsubsection{Quasi-free operad filtrations}\label{Interlude:GraphOperadModel:QuasiFreeFiltration}
By~\cite[Lemma 1.4.11]{FresseCylinder},
the quasi-free operad $\POp = (\FOp(M),\partial)$
associated to a connected $\Sigma_*$-object $M$
inherits a canonical filtration by quasi-free operads $\POp_{\leq r} = (\FOp(M_{\leq r}),\partial)$
such that $M_{\leq r}\subset M$ is the $\Sigma_*$-object formed by the components $M(n)$
of arity $n\leq r$ of $M$.

This observation has a straightforward generalization in the context of $\K$-operads:
to a symmetric $\K$-diagram $M$
we associate the subobject $M_{\leq r}\subset M$
such that
\begin{equation*}
M_{\leq r}(\kappa) = \begin{cases} M(\kappa), & \text{for all $\kappa\in\K(n)$ when $n\leq r$}, \\
0, & \text{otherwise}. \end{cases}
\end{equation*}
Under the assumption $M(0) = 0$,
the derivation of $\POp = (\FOp(M),\partial)$
satisfies automatically $\partial(M_{\leq r})\subset\FOp(M_{\leq r})$ (compare with~\cite[Lemma 1.4.11]{FresseCylinder})
so that we have a quasi-free $\K$-operad $\POp_{\leq r} = (\FOp(M_{\leq r}),\partial)$
such that $\POp_{\leq r}\subset\POp$.
Under the assumption $M(0) = M(1) = 0$ and $\partial(M)\subset\bigoplus_{m\geq 2}\FOp_{m}(M)$,
we have better, namely: $\partial(M_{\leq r})\subset\FOp(M_{\leq r-1})$.
This assumption is not necessary for the definition of the quasi-free operad $\POp_{\leq r}$
but is needed for the proof of the next proposition.

The filtration of a quasi-free $\K$-operad is canonical
in the sense that a morphism
\begin{equation*}
\phi_f: \underbrace{(\FOp(M),\partial)}_{\POp}\rightarrow\underbrace{(\FOp(N),\partial)}_{\QOp}
\end{equation*}
induced by a morphism of symmetric $\K$-diagrams $f: M\rightarrow N$
satisfies $\phi_f(\POp_{\leq r})\subset\QOp_{\leq r}$.
Thus we have a commutative diagram of $\K$-operad morphisms
\begin{equation*}
\xymatrix{ \POp_{\leq r}\ar[d]_{i_r}\ar@{.>}[r]^{\phi_f} &
\QOp_{\leq r}\ar[d]^{j_r} \\
\POp\ar[r]_{\phi_f} & \QOp }
\end{equation*}
for each $r\in\NN$.
For our purpose, we record the following result:

\begin{prop}\label{Interlude:GraphOperadModel:CofibrantKOperads}
Suppose that the symmetric $\K$-diagrams $M$ and $N$ are reduced ($M(0) = M(1) = 0$ and $N(0) = N(1) = 0$)
and the derivations of the quasi-free operads $\POp$ and $\QOp$
satisfy $\partial(M)\subset\subset\bigoplus_{m\geq 2}\FOp_{m}(M)$
and $\partial(N)\subset\bigoplus_{m\geq 2}\FOp_{m}(N)$.
If $f$ is a cofibration of symmetric $\K$-diagrams,
then every pushout morphism
\begin{equation*}
\POp\bigvee_{\POp_{\leq r}}\QOp_{\leq r}\xrightarrow{(\phi_f,j_r)}\QOp,\ r\in\NN,
\end{equation*}
forms a cofibration of $\K$-operads.
\end{prop}

\begin{proof}
Easy extension of the arguments of~\cite[Proposition 1.4.13]{FresseCylinder}
to $\K$-operads.
\end{proof}

\section{Second step: applications of $\K$-structures}\label{SecondStep}

The goal of this section is to prove:

\begin{mainlemm}\label{SecondStep:Result}\hspace*{2mm}
\begin{enumerate}
\item
Each quasi-free operad
\begin{equation*}
\BOp^c(\DOp_n) = (\FOp(\Sigma^{-1}\DOp_n),\partial)
\end{equation*}
arises from a quasi-free $\K$-operad
such that
\begin{equation*}
\colim_{\kappa\in\K_n(r)}\BOp^c(\DOp_n)(\kappa)\xrightarrow{\simeq}\cdots
\xrightarrow{\simeq}\colim_{\kappa\in\K(r)}\BOp^c(\DOp_n)(\kappa)
\xrightarrow{\simeq}\BOp^c(\DOp_n)(r).
\end{equation*}
\item
The morphism $\sigma^*: \BOp^c(\DOp_{n-1})\rightarrow\BOp^c(\DOp_n)$
deduced from the suspension morphism of the Barratt-Eccles operad
is realized by a morphism of quasi-free $\K$-operads
\begin{equation*}
(\FOp(\Sigma^{-1}\DOp_{n-1})(\kappa),\partial)\xrightarrow{\phi_{\sigma^*}}(\FOp(\Sigma^{-1}\DOp_n)(\kappa),\partial)
\end{equation*}
associated to a morphism of symmetric $\K$-diagrams $\sigma^*: \DOp_{n-1}(\kappa)\rightarrow\DOp_n(\kappa)$.
\end{enumerate}
\end{mainlemm}

Then we have immediately:

\begin{maincor}
The diagram of Lemma~\ref{FirstStep:Result}
\begin{equation*}
\xymatrix{ \BOp^c(\DOp_1)\ar[r]^(0.5){\sigma^*}\ar[d]_{\phi_1} &
\BOp^c(\DOp_2)\ar[r]^(0.65){\sigma^*}\ar@{.>}[d]_{\exists\phi_2} &
\cdots\ar[r]^(0.4){\sigma^*} &
\BOp^c(\DOp_n)\ar[r]^(0.65){\sigma^*}\ar@{.>}[d]_{\exists\phi_n} & \cdots \\
\COp\ar[r]_{=} &
\COp\ar[r]_{=} &
\cdots\ar[r]_{=} &
\COp\ar[r]_{=} & \cdots }.
\end{equation*}
is equivalent to an adjoint diagram, in the category of $\K$-operads,
the commutative operad $\COp$ being equipped with the structure of a constant $\K$-operad.
\end{maincor}

The first step toward the proof of Lemma~\ref{SecondStep:Result}
is to define $\K$-structures on the generating $\Sigma_*$-objects $\DOp_n$
of the quasi-free operads $\BOp^c(\DOp_n)$.
This step is carried out in~\S\ref{SecondStep:DualModuleKStructure}.
Besides,
we check in that section that the morphism of $\Sigma_*$-objects $\sigma^*: \DOp_{n-1}\rightarrow\DOp_n$
arises from a morphism of symmetric $\K$-diagrams.

From the result of~\S\ref{SecondStep:DualModuleKStructure},
we obtain that the underlying free operad of $\BOp^c(\DOp_n)$
arises from a free $\K$-operad.
We check in~\S\ref{SecondStep:OperadKStructure} that the twisting derivation of the cobar construction arises
from a derivation of this free $\K$-operad.
We conclude from this observation that the quasi-free operad $\BOp^c(\DOp_n)$
has a $\K$-structure as claimed by the first assertion
of Lemma~\ref{SecondStep:Result}.
We also check that the morphism of symmetric $\K$-diagrams $\sigma^*: \DOp_{n-1}\rightarrow\DOp_n$
defined in~\S\ref{SecondStep:DualModuleKStructure},
induce morphisms of quasi-free $\K$-operads
on cobar constructions.
We deduce the second assertion of Lemma~\ref{SecondStep:Result}
from this verification.

\subsection{Definition of $\K$-structures on generating $\Sigma_*$-objects}\label{SecondStep:DualModuleKStructure}
Our first purpose is to define the $\K$-structure of $\DOp_n$.
The idea beyond our construction is to use a natural complement $\kappa^{\vee}$
associated to each oriented weight system $\kappa$ of $\K_n$
and to associate to $\kappa$
the submodule of $\DOp_n$
orthogonal to the complement of $\EOp_n(\kappa^{\vee})$ in $\EOp_n(r)$.
The precise definition of this submodule $\DOp_n(\kappa)\subset\DOp_n(r)$
is given next.
Then we check that the obtained collection $\DOp_n(\kappa)$
has the required structure to form a symmetric $\K$-diagram
and is preserved by the dual of the suspension morphism of the Barratt-Eccles operad.

\subsubsection{The $\K$-structure on generating $\Sigma_*$-objects}\label{SecondStep:DualModuleKStructure:WeightComplement}
Let $\kappa = (\mu,\sigma)$ be an oriented weight system.
The complement of $\kappa$ in $\K_n$
is the oriented weight system $\kappa^{\vee} = (n-1-\mu,\sigma)$,
where $n-1-\mu$ refers to the collection $\{n-1-\mu_{i j}\}_{i j}$.
If $\kappa\in\K_n(r)$,
then we have by definition $0\leq\mu_{i j}<n$, $\forall i j$,
from which we deduce $0\leq n-1-\mu_{i j}<n$, $\forall i j$.
Hence, we still have $\kappa^{\vee}\in\K_n(r)$.
Otherwise we can assume that $\kappa^{\vee} = (n-1-\mu,\sigma)$
has negative weights.
Observe simply that the order relation of $\K(r)$
has a natural extension
to oriented weight-systems with negative weights.

For our purpose,
it is useful to note:

\begin{obsv}\label{SecondStep:DualModuleKStructure:WeightComplementOrder}
For every pair of oriented weight-systems,
we have $(\mu,\sigma)\leq(\nu,\tau)\Leftrightarrow(n-1-\nu,\tau)\leq(n-1-\mu)$.
\end{obsv}

\subsubsection{The $\K$-structure on generating $\Sigma_*$-objects}\label{SecondStep:DualModuleKStructure:Definition}
In our study,
we use that $\EOp_n$ is defined as a free dg-module over the non-degenerate simplices of the simplicial Barratt-Eccles operad
and that the dg-modules $\DOp_n(r)$, dual to $\EOp_n(r)$,
inherit a basis.
The simplicial Barratt-Eccles operad $\WOp$
is defined by the simplicial sets $\WOp(r)$, $r\in\NN$,
such that $\WOp(r)_d = \Sigma^{\times d+1}$,
for each dimension $d\in\NN$.
The subset of nondegenerate simplices of $\WOp(r)$,
denoted by $\NDeg\WOp(r)\subset\WOp(r)$,
consists of simplices $(w_0,\dots,w_d)\in\WOp(r)$
such that $w_j\not=w_{j+1}$ for $j = 0,\dots,d-1$.
The Barratt-Eccles chain operad $\EOp(r)$
is defined by the free $\ZZ$-module
\begin{equation*}
\EOp(r) = \ZZ[\NDeg\WOp(r)].
\end{equation*}

Recall that we have an oriented weight-system $\kappa(\underline{w}) = (\mu(\underline{w}),\sigma(\underline{w}))$
associated to each simplex $\underline{w} = (w_0,\dots,w_d)\in\WOp(r)$,
such that $\sigma(\underline{w}) = w_d$, the last permutation of $\underline{w}$,
and the term $\mu_{i j}(\underline{w})$
of the collection $\mu(\underline{w}) =\{\mu_{i j}(\underline{w})\}_{i j}$
is defined as the number of variations of the sequence $\underline{w}|_{i j} = (w_0|_{i j},\dots,w_d|_{i j})$.

Smith's filtration of the simplicial Barratt-Eccles operad
is defined by the simplicial sets
\begin{equation*}
\WOp_n(r) = \bigl\{\underline{w}\in\WOp(r)\ \text{such that $\mu_{i j}(\underline{w})<n$ for all pairs $i j$}\bigr\}
\end{equation*}
(see~\cite{Smith}).
We have clearly:
\begin{equation*}
\EOp_n(r) = \ZZ[\NDeg\WOp_n(r)]
\quad\text{and}
\quad\DOp_n(r) = \bigl\{ c: \NDeg\WOp_n(r)\rightarrow\ZZ\bigr\},
\end{equation*}
where $\NDeg\WOp_n(r) = \WOp_n(r)\cap\NDeg\WOp(r)$ denotes the subset of non-degenerate simplices of~$\WOp_n(r)$.
We set
\begin{equation*}
\DOp_n(\mu,\sigma) = \bigl\{ c: \NDeg\WOp_n(r)\rightarrow\ZZ
\quad\text{such that $c(\underline{w})\not= 0\Rightarrow\kappa(\underline{w})\geq(n-1-\mu,\sigma)$}\bigr\},
\end{equation*}
where $\kappa^{\vee} = (n-1-\mu,\sigma)$
represents the complement weight system, defined in~\S\ref{SecondStep:DualModuleKStructure:WeightComplement},
of $\kappa = (\mu,\sigma)$.

We have the easy observations:

\begin{lemm}\label{SecondStep:DualModuleKStructure:KStructureVerification}\hspace*{2mm}
\begin{enumerate}
\item\label{DualCellInclusion}
If $\alpha\leq\beta$, then we have $\DOp_n(\alpha)\subset\DOp_n(\beta)$.
\item\label{DualCellDifferential}
The differential of $\DOp_n(r)$
preserves each subobject $\DOp_n(\kappa)\subset\DOp_n(r)$
so that $\DOp_n(\kappa)$ forms a dg-submodule of $\DOp_n(r)$.
\item\label{DualCellPermutation}
For each permutation $w\in\Sigma_r$, the action of $w$ on $\DOp_n(r)$
maps the submodule $\DOp_n(\kappa)$
into $\DOp_n(w\kappa)$.
\end{enumerate}
\end{lemm}

\begin{proof}
Assertion~(\ref{DualCellInclusion}) is immediate from Observation~\ref{SecondStep:DualModuleKStructure:WeightComplementOrder}.
Assertions~(\ref{DualCellDifferential}-\ref{DualCellPermutation})
are immediate consequences of Observation~\ref{Interlude:CompleteGraphOperad:BarrattEcclesWeightRelations}.
\end{proof}

From which we deduce:

\begin{prop}\label{SecondStep:DualModuleKStructure:DefinitionStatement}
The dg-modules $\{\DOp_n(\kappa)\}_{\kappa\in\K(r)}$, $r\in\NN$,
form a symmetric $\K$-diagram.\qed
\end{prop}

\subsubsection{Dual basis and associated weights}\label{SecondStep:DualModuleKStructure:DualBasis}
In our next verifications,
we use the dual basis in $\DOp_n(r)$
of the basis of non-degenerate simplices $\underline{w}\in\NDeg\WOp_n(r)$
in $\EOp_n(r)$.
The basis element $\underline{s}^{\vee}$
dual to~$\underline{s}\in\NDeg\WOp_n(r)$
is characterized as a map $\underline{s}^{\vee}: \NDeg\WOp_n(r)\rightarrow\ZZ$
by the relation:
\begin{equation*}
\underline{s}^{\vee}(\underline{w}) = \begin{cases} 1, & \text{if $\underline{w} = \underline{s}$}, \\
0, & \text{otherwise}, \end{cases}
\end{equation*}
for each $\underline{w}\in\NDeg\WOp_n(r)$.

To a dual basis element $\underline{s}^{\vee}$,
we associate the complement in $\K_n$
of the oriented weight-system associated to $\underline{s}$:
\begin{equation*}
\kappa^{\vee}(\underline{s}) = (n-1-\mu(\underline{s}),\sigma(\underline{s})).
\end{equation*}
Note again that the relations $0\leq\mu_{i j}(\underline{s})<n$,
which characterize the simplices of $\EOp_n(r)$,
imply $\kappa^{\vee}(\underline{s})\in\K_n(r)$.

The next observation follows from a straightforward inspection of definitions:

\begin{obsv}\label{SecondStep:DualModuleKStructure:DualCellBasis}
For a dual basis element $\underline{s}^{\vee}\in\DOp_n(r)$,
we have $\underline{s}^{\vee}\in\DOp_n(\kappa)$
if and only if $\kappa^{\vee}(\underline{s})\leq\kappa$.
\end{obsv}

The proof of the next proposition is an application of this observation:

\begin{prop}\label{SecondStep:DualModuleKStructure:Colimits}
The embeddings $\DOp_n(\kappa)\subset\DOp_n(r)$
induce isomorphisms of $\Sigma_*$-objects:
\begin{equation*}
\colim_{\kappa\in\K_n(r)}\DOp_n(\kappa)\xrightarrow{\simeq}\cdots
\xrightarrow{\simeq}\colim_{\kappa\in\K(r)}\DOp_n(\kappa)
\xrightarrow{\simeq}\DOp_n(r).
\end{equation*}
\end{prop}

\begin{proof}
Let
\begin{equation*}
\DOp_n(r)\xrightarrow{\Psi}\colim_{\kappa\in\K_m(r)}\DOp_n(\kappa)
\end{equation*}
be the morphism of $\ZZ$-modules mapping a dual basis element $\underline{s}^{\vee}$
to the same element $\underline{s}^{\vee}$
in the summand $\DOp_n(\kappa^{\vee}(\underline{s}))$
of the colimit.
This mapping is well defined for all $m\geq n$
since, for every basis element $\underline{s}^{\vee}\in\DOp_n(r)$,
we have $\kappa^{\vee}(\underline{s})\in\K_n(r)$
(see~\S\ref{SecondStep:DualModuleKStructure:DualBasis}).

This map $\Psi$
is clearly right inverse to the natural morphism
\begin{equation*}
\colim_{\kappa\in\K_m(r)}\{\DOp_n(\kappa)\}
\xrightarrow{\Phi}\DOp_n(r).
\end{equation*}
In the other direction,
observe that a dual basis element $\underline{s}^{\vee}\in\DOp_n(\kappa)$
in the summand associated to any weight-system $\kappa$
is identified in the colimit $\colim_{\kappa\in\K_m(r)}\{\DOp_n(\kappa)\}$
with the same element $\underline{s}^{\vee}$
coming from the summand $\DOp_n(\kappa^{\vee}(\underline{s}))$.
Hence, we also have $\Psi\Phi = \Id$, from which we conclude that $\Psi$ is both right and left inverse to $\Phi$
so that $\Phi$ is necessarily a bijection.
\end{proof}

To complete the results of this subsection,
we study the mapping induced by the suspension morphism
on $\K$-structures.
We observe that:

\begin{lemm}\label{SecondStep:DualModuleKStructure:DualSuspensionFiltration}
The morphism $\sigma^*: \DOp_{n-1}(r)\rightarrow\DOp_n(r)$
maps the submodule $\DOp_{n-1}(\kappa)\subset\DOp_{n-1}(r)$ associated to a weight-system $\kappa\in\K(r)$
into $\DOp_n(\kappa)\subset\DOp_n(r)$.
\end{lemm}

\begin{proof}
Let $c\in\DOp_{n-1}(\mu,\sigma)$.

We have by definition
\begin{equation*}
\sigma^*(c)(w_0,\dots,w_d)
= c(\sgn\cap(w_0,\dots,w_d))
= \sgn(w_0,\dots,w_{r-1})\cdot c(w_{r-1},\dots,w_d),
\end{equation*}
for any simplex $\underline{w} = (w_0,\dots,w_d)\in\NDeg\WOp_{n}(r)$,
for a cochain $\sgn: \EOp(r)\rightarrow\ZZ$
defined in~\S\ref{Prelude:BarrattEccles:CircleCochainMultiplications}.

We have clearly $(w_{r-1},\dots,w_d)\in\NDeg\WOp(r)$,
and we see easily from the definition of $\sgn: \EOp(r)\rightarrow\ZZ$
that $\sgn(w_0,\dots,w_{r-1})\not=0$
implies
\begin{equation*}
\mu_{i j}(w_0,\dots,w_d)\geq\mu_{i j}(w_{r-1},\dots,w_d)+1\quad(\forall i j).
\end{equation*}
When this condition is satisfied,
we have the implications
\begin{multline*}
c(w_{r-1},\dots,w_d)\not=0 \\
\begin{aligned} & \Rightarrow\kappa(w_{r-1},\dots,w_d)\geq(n-2-\mu,\sigma) \\
& \Rightarrow\mu_{i j}(w_{r-1},\dots,w_d)>n-2-\mu_{i j} \\
& \qquad\qquad\text{or}\ (\mu_{i j}(w_{r-1},\dots,w_d),w_d|_{i j})=(n-2-\mu_{i j},\sigma|_{i j})
\quad(\forall i j) \\
& \Rightarrow\mu_{i j}(w_0,\dots,w_d)>n-1-\mu_{i j} \\
& \qquad\qquad\text{or}\ (\mu_{i j}(w_{r-1},\dots,w_d),w_d|_{i j})=(n-1-\mu_{i j},\sigma|_{i j})
\quad(\forall i j) \\
& \Rightarrow\kappa(w_0,\dots,w_d)\geq(n-1-\mu,\sigma),
\end{aligned}
\end{multline*}
from which we conclude:
$c\in\DOp_{n-1}(\mu,\sigma)\Rightarrow\sigma^*(c)\in\DOp_n(\mu,\sigma)$.
\end{proof}

This lemma gives as a corollary:

\begin{prop}\label{SecondStep:DualModuleKStructure:DualSuspensionRealization}
The morphism $\sigma^*: \DOp_{n-1}(r)\rightarrow\DOp_n(r)$ is realized by a morphism of symmetric $\K$-diagrams
$\sigma^*: \DOp_{n-1}(\kappa)\rightarrow\DOp_n(\kappa)$.\qed
\end{prop}

\subsection{Extension of $\K$-structures to quasi-free operads}\label{SecondStep:OperadKStructure}
In order to define the $\K$-operad underlying the cobar construction $\BOp^c(\DOp_n)$,
we check how the $\K$-structure of $\DOp_n$
extends to the quasi-free operad
\begin{equation*}
\BOp^c(\DOp_n) = (\FOp(\Sigma^{-1}\DOp_n),\partial).
\end{equation*}
Throughout this subsection,
we adopt the short notation $\MOp_n = \Sigma^{-1}\DOp_n$
for the generating $\Sigma_*$-object of this quasi-free operad
and its underlying symmetric $\K$-diagram.

By adjunction,
we have clearly:

\begin{obsv}\label{SecondStep:OperadKStructure:FreeOperadColimit}
For any symmetric $\K$-diagram $M$,
we have a natural isomorphism
\begin{equation*}
\colim_{\K}\FOp(M)\simeq\FOp(\colim_{\K} M),
\end{equation*}
between:
\begin{itemize}
\item
on the left-hand-side,
the colimit $\colim_{\K}\FOp(M)$
of the free $\K$-operad $F(M)$ associated to the symmetric $\K$-diagram $M(\kappa)$,
\item
on the right-hand side,
the usual free operad on the $\Sigma_*$-object $(\colim_{\K} M)(r) = \colim_{\kappa\in\K(r)} M(\kappa)$
associated to the symmetric $\K$-diagram $M(\kappa)$.
\end{itemize}
\end{obsv}

Hence,
in the case $M = \MOp_n$,
we obtain that the underlying free operad of the cobar construction $\BOp^c(\DOp_n)$
has a $\K$-structure.
The main task of this subsection
is to prove that the twisting differential of the cobar construction
is realized at the $\K$-operad level.
For this aim,
we need at first to analyze the differential of dual basis elements $\underline{s}^{\vee}$
in the usual cobar construction of the cooperad $\DOp_n$.

\subsubsection{The differential of dual basis elements in the usual cobar construction}\label{SecondStep:OperadKStructure:GenuineBasisDifferential}
According to the description of~\S\S\ref{Prelude:KoszulDuality:FreeOperad}-\ref{Prelude:KoszulDuality:CobarConstruction},
the cobar differential maps a cooperad element $c\in\DOp(r)$
to a sum of terms $\partial_{\tau}(c)\in\tau(\Sigma^{-1}\DOp)(r)$,
where $\tau$ ranges over trees with two vertices.
For the $\ZZ$-dual cooperad of a dualizable operad $\DOp = \POp^{\vee}$,
we have a duality pairing
\begin{equation*}
\langle-,-\rangle: \tau(\Sigma^{-1}\DOp)\otimes\tau(\Sigma\POp)\rightarrow\ZZ.
\end{equation*}
The elements of the dg-module $\tau(\Sigma\POp)(r)$
represent formal operadic composites $p(i_1,\dots,i_e,\dots,e_s)\circ_{i_e} q(j_1,\dots,j_t)$, $p\in\POp(s)$, $q\in\POp(t)$
together with an input sharing $\{1,\dots,r\} = \{i_1,\dots,\widehat{i_e},\dots,i_s\}\amalg\{j_1,\dots,j_t\}$
determined by the structure of the tree $\tau$.
The homomorphism $\partial_{\tau}: \Sigma^{-1}\DOp\rightarrow\tau(\Sigma^{-1}\DOp)$
is dual to the natural morphism $\lambda_{\tau}: \tau(\Sigma\POp)\rightarrow\Sigma\POp$
which forgets suspensions and performs the partial composite $p(i_1,\dots,i_e,\dots,i_s)\circ_{i_e} q(j_1,\dots,j_t)$ in $\POp$.
To forget suspension,
we have to move suspensions on the left first and this produces a sign,
but this process does not create any actual problem.

In this proof,
we use the literal expression $p(i_1,\dots,i_e,\dots,e_s)\circ_{i_e} q(j_1,\dots,j_t)$
to represent an element of $\tau(\Sigma\POp)$ (or $\tau(\Sigma^{-1}\DOp)$)
rather than the graphical representation of Figure~\ref{fig:QuadraticComposite}.
To shorten notation,
we also set $p(i_*) = p(i_1,\dots,i_e,\dots,e_s)$ and $q(j_*) = q(j_1,\dots,j_t)$.
Recall that the index sharing $\{1,\dots,r\} = \{i_1,\dots,\widehat{i_e},\dots,i_s\}\amalg\{j_1,\dots,j_t\}$
is fixed by the structure of the tree.
Note that the index $i_e$ in the expression $p(i_*)\circ_{i_e} q(j_*)$
is a dummy variable.

The element $w(i_1,\dots,i_r)$
associated to a permutation $w$
is equivalent to an ordering $(i_{w(1)},\dots,i_{w(r)})$
of the set $\{i_1,\dots,i_r\}$.
The element $\underline{w}(i_1,\dots,i_r)$
associated to a simplex of permutations $\underline{w} = (w_0,\dots,w_d)$
is just the $d+1$-tuple of orderings $\underline{w}(i_*) = (w_0(i_*),\dots,w_d(i_*))$
associated to the permutations of $\underline{w}$.

The composite $u(i_1,\dots,i_e,\dots,i_s)\circ_{i_e} v(j_1,\dots,j_t)$
of permutations is defined by the substitution of the value $i_e$
in the ordering $u(i_*) = (i_{u(1)},\dots,i_e,\dots,i_{u(s)})$
by the sequence $v(j_*) = (j_{v(1)},\dots,j_{v(t)})$.
For instance,
we have
\begin{equation*}
(1,i_e,4)\circ_{i_e} (5,2,3) = (1,5,2,3,4).
\end{equation*}
For the Barratt-Eccles chain operad $\EOp(r) = N_*(\WOp(r))$
and its suboperad $\EOp_n\subset\EOp$,
a partial composite of simplices
\begin{equation*}
\underline{u}(i_*) = (u_0(i_*),\dots,u_m(i_*))
\quad\text{and}
\quad\underline{v}(j_*) = (v_0(j_*),\dots,v_n(j_*))
\end{equation*}
is a signed sum of simplices of the form
\begin{equation*}
\underline{w} = (u_{k_0}(i_*)\circ_{i_e} v_{l_0}(j_*),\dots,u_{k_{m+n}}(i_*)\circ_e v_{l_{m+n}}(j_*)),
\end{equation*}
where the index sequence
\begin{equation*}
(k_0,l_0)\rightarrow(k_1,l_1)\rightarrow\dots\rightarrow(k_{m+n},l_{m+n})
\end{equation*}
ranges over paths
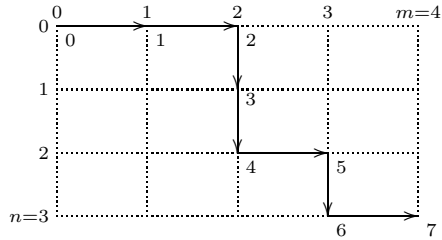
\begin{equation*}
\xymatrix@M=0pt@C=12mm{\ar@{.}[d]_<{0}\ar@{->}[r]^<{0}_(0.15){0} &
\ar@{.}[d]\ar@{->}[r]^<{1}_(0.15){1} &
\ar@{->}[d]\ar@{.}[r]^<{2}_(0.15){2} &
\ar@{.}[d]\ar@{.}[r]^<{3}^>{m=4} &
\ar@{.}[d] \\
\ar@{.}[d]_<{1}\ar@{.}[r] & \ar@{.}[d]\ar@{.}[r] &
\ar@{.}[d]\ar@{->}[d]^(0.15){3}\ar@{.}[r] & \ar@{.}[d]\ar@{.}[r] & \ar@{.}[d] \\
\ar@{.}[d]_<{2}_>{n=3}\ar@{.}[r] & \ar@{.}[d]\ar@{.}[r] &
\ar@{.}[d]\ar@{->}[r]_(0.15){4} & \ar@{->}[d]\ar@{.}[r]_(0.15){5} & \ar@{.}[d] \\
\ar@{.}[r] & \ar@{.}[r] & \ar@{.}[r] & \ar@{->}[r]_(0.15){6}_(1.15){7} & \\ }
\end{equation*}
and each $u_{k_*}(i_*)\circ_{i_e} v_{l_*}(j_*)$
is given by the substitution of permutations.
We use the notation $\underline{w}\in\underline{u}(i_*)\circ_{i_e}\underline{v}(j_*)$
to mean that a simplex $\underline{w}$
occurs in the expansion of $\underline{u}(i_*)\circ_{i_e}\underline{v}(j_*)$.

The definition of the substitution process for permutations
implies immediately that:

\begin{obsv}\label{SecondStep:OperadKStructure:SubstitutionPermutationMapping}
The mapping $(u(i_*),v(j_*))\mapsto u(i_*)\circ_{i_e} v(i_*)$
is injective.
\end{obsv}

Indeed,
the ordering $v(j_*)$ is identified with the connected subsequence of $w = u(i_*)\circ_{i_e} v(i_*)$
formed by the occurrences of $\{j_1,\dots,j_t\}$
and we recover $u(j_*)$ by replacing this subsequence by the variable $i_e$.
By an easy generalization of this argument, we obtain:

\begin{obsv}\label{SecondStep:OperadKStructure:SubstitutionSimplexMapping}
For any fixed sharing $\{1,\dots,r\} = \{i_1,\dots,\widehat{i_e},\dots,i_s\}\amalg\{j_1,\dots,j_t\}$
a non-degenerate simplex $\underline{w}\in\WOp(r)$
occurs in at most one partial composite of non-degenerate simplices $\underline{u}(i_*)\circ_{i_e}\underline{v}(j_*)$
and only once if so.
\end{obsv}

In light of this analysis,
we obtain from the definition of $\partial_{\tau}$
a result of the form:

\begin{fact}\label{SecondStep:OperadKStructure:BasisDifferentialMapping}
Let $\tau\in\Theta(r)$ be a tree with two vertices.
Let $\{1,\dots,r\} = \{i_1,\dots,\widehat{i_e},\dots,i_s\}\amalg\{j_1,\dots,j_t\}$
be the index sharing associated to $\tau$.
For a basis element $\underline{s}^{\vee}$,
we have:
\begin{equation*}
\partial_{\tau}(\underline{s}^{\vee})
= \begin{cases} \pm\underline{u}^{\vee}(i_*)\circ_{i_e}\underline{v}^{\vee}(j_*),
& \text{whenever $\underline{s}\in\underline{u}(i_*)\circ_{i_e}\underline{v}(j_*)$} \\
& \qquad\qquad\qquad\text{for some $\underline{u}\in\WOp(s)$, $\underline{v}\in\WOp(t)$}, \\
0, & \text{otherwise}. \end{cases}
\end{equation*}
\end{fact}

Moreover:

\begin{fact}\label{SecondStep:OperadKStructure:BasisDifferentialWeight}
If $\underline{s}^{\vee}$
has a non-trivial cobar differential in $\tau(\MOp_n)$,
then we have the relation
$\kappa^{\vee}(\underline{s}) = \kappa^{\vee}(\underline{u})(i_*)\circ_{i_e}\kappa^{\vee}(\underline{v})(j_*)$.
\end{fact}

This assertion is an immediate consequence of the identity
$\kappa(\underline{u}\circ_{e}\underline{v}) = \kappa(\underline{u})\circ_{e}\kappa(\underline{v})$
of Proposition~\ref{Interlude:CompleteGraphOperad:BarrattEcclesWeightRelations}.

\subsubsection{The differential of dual basis elements in the $\K$-diagram version of the cobar construction}\label{SecondStep:OperadKStructure:BasisDifferentialDefinition}
For a fixed $\kappa\in\K(r)$,
the summand $\tau(M)(\kappa)$ of a free operad $\FOp(M)$
is represented by a colimit
\begin{equation*}
\tau(M)(\kappa) = \colim_{\alpha(i_*)\circ_{i_e}\beta(j_*)\leq\kappa} \tau(M,\alpha(i_*)\circ_{i_e}\beta(j_*)),
\end{equation*}
where each $\tau(M,\alpha(i_*)\circ_{i_e}\beta(j_*))$
consists of formal composites $p(i_*)\circ_{i_e} q(j_*)$
such that $p\in M(\alpha)$ and $q\in M(\beta)$.
Define the image of a basis element $\underline{s}^{\vee}\in\DOp_n(\kappa)$
in the summand
$\tau(\MOp_n)(\kappa)\subset\FOp(\MOp_n)(\kappa)$
as the element of the colimit represented by the composite
\begin{equation*}
\underline{u}(i_*)\circ_{i_e}\underline{v}(j_*)
\in\tau(\MOp_n,\alpha(i_*)\circ_{i_e}\beta(j_*)),
\end{equation*}
where $\alpha = \kappa^{\vee}(\underline{u})$, $\beta = \kappa^{\vee}(\underline{v})$,
whenever we have $\underline{s}\in\underline{u}(i_*)\circ_{i_e}\underline{v}(j_*)$.
This image is set to be zero otherwise.
This definition is coherent with the definition
of $\K$-structures
since
\begin{equation*}
\kappa^{\vee}(\underline{s})\leq\kappa
\Rightarrow
\kappa^{\vee}(\underline{u})(i_*)\circ_{i_e}\kappa^{\vee}(\underline{v})(j_*)
= \kappa^{\vee}(\underline{s})\leq\kappa.
\end{equation*}

Moreover:

\begin{lemm}\label{SecondStep:OperadKStructure:CobarDifferentialKStructure}
The homomorphisms
\begin{equation*}
\MOp_n(\kappa)\xrightarrow{\partial}\FOp(\MOp_n)(\kappa)
\end{equation*}
defined by the process of~\S\ref{SecondStep:OperadKStructure:BasisDifferentialDefinition}
preserve $\K$-diagram structures,
commute with symmetric group actions,
and make commute the diagrams
\begin{equation*}
\xymatrix{ \MOp_n(\kappa)\ar[d]\ar@{.>}[r]^(0.45){\partial} &
\FOp(\MOp_n)(\kappa)\ar[d] \\
\MOp_n(r)\ar[r]_(0.45){\partial} &
\FOp(\MOp_n)(r) },
\end{equation*}
where $\partial: \MOp_n(r)\rightarrow\FOp(\MOp_n)(r)$ refers to the twisting homomorphism
of the usual cobar construction.
\end{lemm}

\begin{proof}
Immediate from the definition of~\S\ref{SecondStep:OperadKStructure:BasisDifferentialDefinition}.
\end{proof}

Thus,
our definition gives a well-defined homomorphism of symmetric $\K$-diagrams
\begin{equation*}
\MOp_n\xrightarrow{\partial}\FOp(\MOp_n)
\end{equation*}
and we obtain naturally a derivation of the free $\K$-operad underlying the differential of usual cobar construction
when we form the derivation associated to this homomorphism.
But
we still have to check that this $\K$-operad derivation satisfies the equation $\delta(\partial) = \partial^2 = 0$
in order to achieve the construction of the quasi-free $\K$-operad underlying the cobar construction
$\BOp^c(\DOp_n)$.

To ease verifications,
we use the following observation:

\begin{obsv}\label{SecondStep:OperadKStructure:CompositeKStructureEmbedding}
Let $M$ be any symmetric $\K$-diagram.
Let $\tau$ denote a reduced tree (see~\S\ref{Prelude:KoszulDuality:FreeOperad}).
\begin{enumerate}
\item
For an oriented weight system of the form $\kappa = \lambda_*(\beta_*)$, where $\beta_*\in\tau(\K)$,
we have a natural isomorphism
\begin{equation*}
\tau(M,\beta_*)\rightarrow\colim_{\lambda_*(\alpha_*)\leq\kappa}\tau(M,\alpha_*) = \tau(M)(\kappa),
\end{equation*}
because $\beta_*\in\tau(\K)$ is the largest element of $\tau(\K)$
satisfying $\lambda_*(\beta_*)\leq\kappa$.
\item
If the symmetric $\K$-diagram $M$
consists of subobjects of a $\Sigma_*$-object $M(\kappa)\subset M(r)$ (like $\MOp_n$),
then the natural morphism
\begin{equation*}
\tau(M,\beta_*)\rightarrow\tau(M)(r)
\end{equation*}
is an embedding, for every $\beta_*\in\tau(\K)$.
\end{enumerate}
\end{obsv}

\medskip
We check at first:

\begin{lemm}\label{SecondStep:OperadKStructure:InternalDifferentialBasisCobarDerivation}
The homomorphism of Lemma~\ref{SecondStep:OperadKStructure:CobarDifferentialKStructure}
satisfies the commutation relation
$(\delta\partial + \partial\delta)(\underline{s}^{\vee}) = 0$
with respect to the internal differential of $\DOp_n$,
for every basis element $\underline{s}^{\vee}$.
\end{lemm}

\begin{proof}
We prove that the composites of the diagram
\begin{equation}\label{eqn:InternalDifferentialBasisCobarDerivation}
\xymatrix{ \MOp_n(\kappa)\ar[r]^{\partial_{\tau}}\ar[d]_{\delta} &
\tau(\MOp_n)(\kappa)\ar[d]^{\delta} \\
\MOp_n(\kappa)\ar[r]_{\partial_{\tau}} &
\tau(\MOp_n)(\kappa) }
\end{equation}
agree on $\underline{s}^{\vee}$ (up to the minus sign), for each fixed tree with two vertices $\tau$.
We can assume $\kappa = \kappa^{\vee}(\underline{s})$
since each mapping $\partial_{\tau}$ preserves $\K$-diagram structures.

If $\kappa = \kappa^{\vee}(\underline{s})$ has the form $\kappa = \alpha(i_*)\circ_{i_e}\beta(j_*)$,
where $\alpha(i_*)\circ_{i_e}\beta(j_*)$ is the partial composite representation of an element of $\tau(\K)$,
then we are done,
because the dg-module $\tau(\MOp_n)(\kappa)$
embeds into $\tau(\MOp_n)(r)$ by Observation~\ref{SecondStep:OperadKStructure:CompositeKStructureEmbedding}
and the commutation with internal differentials is satisfied in the usual cobar construction
when $\underline{s}^{\vee}$ is viewed as an element of~$\MOp_n(r)$.

Otherwise, we have necessarily $\partial_{\tau}(\underline{s}^{\vee}) = 0$
by Fact~\ref{SecondStep:OperadKStructure:BasisDifferentialWeight}.
The internal differential of $\MOp_n(r)$
is defined on dual basis elements
by a formula of the form
\begin{equation*}
\delta(\underline{s}^{\vee}) = \sum_{\underline{s}\in\delta(\underline{w})} \pm\underline{w}^{\vee},
\end{equation*}
where we use the notation $\underline{s}\in\delta(\underline{w})$
to assert that the term $\underline{s}$ occurs in the expansion of $\delta(\underline{w})$
for any simplex $\underline{w}$.
Note again that $\underline{s}$ occurs only at most once in $\delta(\underline{w})$,
because $\underline{w}$ is assumed to be non-degenerated.

From the derivation relation
\begin{equation*}
\delta(\underline{u}(i_*)\circ_{i_e}\underline{v}(j_*))
= \delta(\underline{u}(i_*))\circ_{i_e}\underline{v}(j_*)
+ \pm\underline{u}(i_*)\circ_{i_e}\delta(\underline{v}(j_*)),
\end{equation*}
we obtain the implication:
\begin{multline*}
\underline{s}\in\delta(\underline{w})
\quad\text{and}
\quad\underline{w}\in\underline{u}(i_*)\circ_{i_e}\underline{v}(j_*) \\
\Rightarrow\quad\text{$\underline{s}\in\underline{t}(i_*)\circ_{i_e}\underline{v}(j_*)$
for some $\underline{t}\in\delta(\underline{u})$}, \\
\text{or $\underline{s}\in\underline{u}(i_*)\circ_{i_e}\underline{t}(j_*)$
for some $\underline{t}\in\delta(\underline{v})$},
\end{multline*}
whose conclusion requires that $\kappa^{\vee}(\underline{s})$
has the form $\kappa^{\vee}(\underline{s}) = \alpha(i_*)\circ_{i_e}\beta(j_*)$
for some $\alpha,\beta$.
Hence, if this is not the case,
then we also have $\partial_{\tau}(\underline{w}^{\vee}) = 0$
for each $\underline{w}$ such that $\underline{s}\in\delta(\underline{w})$.
From this observation,
we conclude that diagram~(\ref{eqn:InternalDifferentialBasisCobarDerivation})
still commutes when we have $\kappa\not=\alpha(i_*)\circ_{i_e}\beta(j_*)$, $\forall\alpha(i_*)\circ_{i_e}\beta(j_*)\in\tau(\K)$.
\end{proof}

\begin{lemm}\label{SecondStep:OperadKStructure:BasisCobarDerivationSquare}
The $\K$-operad derivation $\partial: \FOp(\MOp_n)\rightarrow\FOp(\MOp_n)$
induced by the homomorphism of Lemma~\ref{SecondStep:OperadKStructure:CobarDifferentialKStructure}
satisfies $\partial\partial(\underline{s}^{\vee}) = 0$,
for every generating element $\underline{s}^{\vee}\in\DOp_n$.
\end{lemm}

\begin{proof}
By a straightforward extension of the description of $\partial\partial$
in the usual cobar construction (see for instance~\cite[\S\S 3.3-3.5]{FressePartitions}
or \cite[\S 1.4.3]{FresseCylinder}),
the restriction of the map
\begin{equation*}
\partial\partial: \FOp(\MOp_n)(\kappa)
\rightarrow\FOp(\MOp_n)(\kappa)
\end{equation*}
to generators $\MOp_n(\kappa)$
has a component for each summand $\tau(\MOp_n)$
associated to a tree with $3$ vertices
and this component is defined by the sum of all composites
\begin{equation*}
\MOp_n(\kappa)
\xrightarrow{\partial_{\rho}}\rho(\MOp_n)(\kappa)
\xrightarrow{\partial_{\sigma}}\tau(\MOp_n)(\kappa),
\end{equation*}
where $\partial_{\sigma}$ is applied to a vertex $v\in V(\rho)$
such that the blow-up $v\mapsto\sigma$ in the tree $\rho$
gives the tree $\tau$.

By an easy inspection of definitions,
we check that any such composite cancels $\underline{s}^{\vee}$
when $\kappa^{\vee}(\underline{s})$ is not of the form $\kappa^{\vee}(\underline{s}) = \lambda_*(\alpha_*)$
for some $\alpha_*\in\tau(\K)$.
Hence, in this case, the equation $\partial\partial(\underline{s}^{\vee}) = 0$
holds trivially.
Otherwise,
the summand $\tau(\MOp_n)(\kappa)$
embeds into $\tau(\MOp_n)(r)$
by Observation~\ref{SecondStep:OperadKStructure:CompositeKStructureEmbedding}
and since the identity $\partial\partial(\underline{s}^{\vee}) = 0$
holds in the usual cobar construction,
we still have $\partial\partial(\underline{s}^{\vee}) = 0$
on $\tau(\MOp_n)(\kappa)$.
Thus,
we obtain that $\partial\partial(\underline{s}^{\vee})$
vanishes in all cases.
\end{proof}

\begin{lemm}\label{SecondStep:OperadKStructure:CobarDerivationDifferential}
The $\K$-operad derivation $\partial: \FOp(\MOp_n)\rightarrow\FOp(\MOp_n)$
satisfies the assertions of Lemma~\ref{SecondStep:OperadKStructure:InternalDifferentialBasisCobarDerivation}
and Lemma~\ref{SecondStep:OperadKStructure:BasisCobarDerivationSquare}
for all elements of $\FOp(\MOp_n)$
and not only generating elements.
\end{lemm}

\begin{proof}
Immediate from the derivation relation (straightforward generalization of the usual verification
for the twisting derivation of a quasi-free operad).
\end{proof}

From Observation~\ref{SecondStep:OperadKStructure:FreeOperadColimit}
and lemmas~\ref{SecondStep:OperadKStructure:CobarDifferentialKStructure}-\ref{SecondStep:OperadKStructure:CobarDerivationDifferential},
we conclude:

\begin{thm}\label{SecondStep:OperadKStructure:DefinitionResult}
We have a quasi-free $\K$-operad
\begin{equation*}
\BOp^c(\DOp_n) = (\FOp(\MOp_n),\partial)
\end{equation*}
whose colimit $\colim_{\kappa\in\K(r)}\BOp^c(\DOp_n)(\kappa)$
is isomorphic to the usual cobar construction of the cooperad $\DOp_n$.\qed
\end{thm}

This theorem gives a first part of the assertions of Lemma~\ref{SecondStep:Result}.

\medskip
For our needs,
we prove a strengthened form of Observation~\ref{SecondStep:OperadKStructure:FreeOperadColimit}:

\begin{prop}\label{SecondStep:OperadKStructure:ColimitLayers}
Let $M$ be any symmetric $\K$-diagram.
The natural morphism $\eta: M\rightarrow\FOp(M)$ induces a natural isomorphism of operads
\begin{equation*}
\phi_\eta: \FOp(\colim_{\K_m} M)\xrightarrow{\simeq}\colim_{\K_m}\FOp(M),
\end{equation*}
for every $m$, including $m = \infty$
in which case we recover the identity of Observation~\ref{SecondStep:OperadKStructure:FreeOperadColimit}.
\end{prop}

\begin{proof}
From the definition of the free $\K$-operad $\FOp(M)$,
we obtain by interchange and composition of colimits:
\begin{multline*}
\colim_{\kappa\in\K_m(r)}\bigl\{\FOp(M)(\kappa)\bigr\}
= \colim_{\kappa\in\K_m(r)}\Bigl\{\bigoplus_{\tau\in\Theta(r)}\tau(M)(\kappa)/\equiv\Bigr\} \\
\begin{aligned}
& = \bigoplus_{\tau\in\Theta(r)}\colim_{\kappa\in\K_m(r)}
\Bigl\{\colim_{\lambda_*(\alpha_*)\leq\kappa}\tau(M,\alpha_*)\Bigr\}/\equiv \\
& = \bigoplus_{\tau\in\Theta(r)}
\Bigl\{\colim_{|\alpha_*|<n}\tau(M,\alpha_*)\Bigr\}/\equiv,
\end{aligned}
\end{multline*}
where we use the notation $|(\mu,\sigma)| = \max_{i j}\{\mu_{i j}\}$
for an oriented weight-system $(\mu,\sigma)$
and $|\alpha_*| = \max_{v\in V(\tau)} |\alpha_v|$
for a composite $\alpha_*\in\tau(\K)$.

We have clearly
\begin{equation*}
|\alpha_*|<n\Leftrightarrow\alpha_*\in\tau(\K_m).
\end{equation*}
Hence,
the standard relation $\colim_{I\times J} X(i)\otimes Y(j)\simeq\colim_I X(i)\otimes\colim_J Y(j)$
for tensor products of colimits
gives the identity
\begin{equation*}
\colim_{|\alpha_*|<n}\tau(M,\alpha_*) = \tau(\colim_{\K_n} M)
\end{equation*}
and the conclusion follows.
\end{proof}

In the case $M = \MOp_n$,
we deduce from this lemma and Proposition~\ref{SecondStep:DualModuleKStructure:Colimits}:

\begin{prop}\label{SecondStep:OperadKStructure:CobarColimitLayers}\hspace*{2mm}
For the quasi-free $\K$-operad
\begin{equation*}
\BOp^c(\DOp_n) = (\FOp(\MOp_n),\partial)
\end{equation*}
we have natural isomorphisms
\begin{equation*}
\colim_{\kappa\in\K_n(r)}\BOp^c(\DOp_n)(\kappa)
\xrightarrow{\simeq}\cdots\xrightarrow{\simeq}\colim_{\kappa\in\K(r)}\BOp^c(\DOp_n)(\kappa)\xrightarrow{\simeq}\BOp^c(\DOp_n)(r).
\end{equation*}
\end{prop}

The verification of this proposition achieves the proof of the first assertion
of Theorem~\ref{SecondStep:Result}.
To complete our result,
we still have to study the mapping induced by the morphism of $\K$-diagrams $\sigma^*: \DOp_{n-1}(\kappa)\rightarrow\DOp_n(\kappa)$
at the level of free $\K$-operads.
Our conclusions arise from the following lemmas:

\begin{lemm}\label{SecondStep:OperadKStructure:DualSuspensionBasisCobarDerivation}
Our homomorphisms $\partial: \MOp_n\rightarrow\FOp(\MOp_n)$
satisfy the commutation relation
\begin{equation*}
\partial(\sigma^*\underline{s}^{\vee}) = \sigma^*\partial(\underline{s}^{\vee})
\end{equation*}
with respect to the morphism of symmetric $\K$-diagrams $\sigma^*: \DOp_{n-1}(\kappa)\rightarrow\DOp_n(\kappa)$,
for every basis element $\underline{s}^{\vee}$.
\end{lemm}

\begin{proof}Similar to Lemma~\ref{SecondStep:OperadKStructure:InternalDifferentialBasisCobarDerivation}.\end{proof}

\begin{lemm}\label{SecondStep:OperadKStructure:CobarDerivationDualSuspension}
For the twisting derivation of $\K$-operad
associated to the homomorphisms
$\partial: \MOp_n\rightarrow\FOp(\MOp_n)$,
the assertion of Lemma~\ref{SecondStep:OperadKStructure:DualSuspensionBasisCobarDerivation}
holds for all elements
and not only for generating elements,
so that we have a commutative diagram
\begin{equation*}
\xymatrix{ \FOp(\MOp_n)\ar[r]^{\partial} &
\FOp(\MOp_n) \\
\FOp(\MOp_{n-1})\ar[r]^{\partial}\ar[u]^{\FOp(\sigma^*)} &
\FOp(\MOp_{n-1})\ar[u]_{\FOp(\sigma^*)} },
\end{equation*}
for all $n\geq 2$.
\end{lemm}

\begin{proof}Straightforward:
use the derivation relation
to deduce the claim of this lemma
from the result of Lemma~\ref{SecondStep:OperadKStructure:DualSuspensionBasisCobarDerivation}.\end{proof}

From these lemmas, we conclude:

\begin{prop}\label{SecondStep:OperadKStructure:DualSuspension}
The morphism of symmetric $\K$-diagrams $\sigma^*: \DOp_{n-1}\rightarrow\DOp_n$
of Proposition~\ref{SecondStep:DualModuleKStructure:DualSuspensionRealization}
induces a morphism of quasi-free $\K$-operads
\begin{equation*}
(\FOp(\MOp_{n-1}),\partial)\xrightarrow{\phi_{\sigma^*}}(\FOp(\MOp_n),\partial)
\end{equation*}
which gives a realization at the $\K$-structure level
of the morphism
\begin{equation*}
\sigma^*: \BOp^c(\DOp_{n-1})\rightarrow\BOp^c(\DOp_n)
\end{equation*}
yielded by the suspension morphism of the Barratt-Eccles operad.\qed
\end{prop}

The verification of this proposition achieves the proof of the second assertion of Lemma~\ref{SecondStep:Result}.\qed

\section[Final step: applications of model structures and lifting arguments]{Final step: applications of model structures\\and lifting arguments}\label{FinalStep}

The goal of this section is to prove:

\begin{mainlemm}\label{FinalStep:Result}
We have morphisms
\begin{equation*}
\xymatrix{ \BOp^c(\DOp_1)\ar[r]^(0.5){\sigma^*}\ar[d]_{\kappa} &
\BOp^c(\DOp_2)\ar[r]^(0.65){\sigma^*}\ar@{.>}[d]_{\exists\kappa} &
\cdots\ar[r]^(0.4){\sigma^*} &
\BOp^c(\DOp_n)\ar[r]^(0.65){\sigma^*}\ar@{.>}[d]_{\exists\kappa} & \cdots \\
\EOp_1\ar[r]_{\iota} &
\EOp_2\ar[r]_{\iota} &
\cdots\ar[r]_{\iota} &
\EOp_n\ar[r]_{\iota} & \cdots }
\end{equation*}
lifting the morphisms $\phi_n: \BOp^c(\DOp_n)\rightarrow\COp$
of Lemma~\ref{FirstStep:Result}
and satisfying the requirement of Assertion~(\ref{Prelude:RealizationTheorem:Prescription})
in Theorem~\ref{Prelude:RealizationTheorem}.
\end{mainlemm}

In~\S\ref{SecondStep},
we proved that the morphisms $\phi_n: \BOp^c(\DOp_n)\rightarrow\COp$
of Lemma~\ref{FirstStep:Result}
are realized by morphisms of $\K$-operads.
The rough idea is to apply the lifting argument of~\S\ref{FirstStep:Conclusion}
in order to get morphisms of $\K$-operads $\tilde{\phi}_n: \BOp^c(\DOp_n)\rightarrow\EOp$
lifting $\phi_n: \BOp^c(\DOp_n)\rightarrow\COp$.
Then we simply take the $\K_n$-colimits of these morphisms
and use the identity $\EOp_n = \colim_{\K_n}\EOp$
to produce the morphisms $\kappa: \BOp^c(\DOp_n)\rightarrow\EOp_n$
of Lemma~\ref{SecondStep:Result}.

First of all,
we check in~\S\ref{FinalStep:KModuleCofibration} that the morphism of symmetric $\K$-diagrams
\begin{equation*}
\sigma^*: \DOp_{n-1}\rightarrow\DOp_n
\end{equation*}
which represents the dual of the suspension morphism of the Barratt-Eccles operad
is a cofibration (of symmetric $\K$-diagrams).
In this situation,
we obtain from Proposition~\ref{Interlude:GraphOperadModel:CofibrantKOperads}
that a natural cobase extension of the morphism of quasi-free operads $\sigma^*: \BOp^c(\DOp_{n-1})\rightarrow\BOp^c(\DOp_n)$
is a cofibration of $\K$-operads,
for each $n>1$.

Then
we apply the model category structure of $\K$-operads
in order to define the desired liftings $\tilde{\phi}_n: \BOp^c(\DOp_n)\rightarrow\EOp$
in~\S\ref{FinalStep:KOperadLifting}
and we achieve the proof of Lemma~\ref{FinalStep:Result} in~\S\ref{FinalStep:Conclusion}.

\subsection{The cofibration statement}\label{FinalStep:KModuleCofibration}
The purpose of this subsection is to prove the following proposition:

\begin{mainsecprop}\label{FinalStep:KModuleCofibration:SuspensionCofibration}
The morphism of symmetric $\K$-diagrams
\begin{equation*}
\sigma^*: \DOp_{n-1}\rightarrow\DOp_n
\end{equation*}
is a cofibration, for all $n>1$.
\end{mainsecprop}

According to Proposition~\ref{Interlude:GraphOperadModel:SymmetricKDiagramCofibrations},
the verification of Proposition~\ref{FinalStep:KModuleCofibration:SuspensionCofibration}
reduces to the following result:

\begin{mainseclemm}\label{FinalStep:KModuleCofibration:LatchingCofibration}
The extended latching morphism
\begin{equation*}
\DOp_{n-1}(\kappa)\bigoplus_{L\DOp_{n-1}(\kappa)}L\DOp_n(\kappa)
\xrightarrow{(\sigma^*,\lambda)}\DOp_n(\kappa)
\end{equation*}
is a cofibration of dg-modules
for every $\kappa\in\K(r)$, $r\in\NN$.
\end{mainseclemm}

The verification of this lemmas is deferred to a series of sublemmas.

\begin{lemm}\label{FinalStep:KModuleCofibration:LatchingObject}
The latching morphism $\lambda: L\DOp_n(\kappa)\rightarrow\DOp_n(\kappa)$
induces an isomorphism between the latching object
\begin{equation*}
L\DOp_n(\kappa) = \colim_{\alpha\lneqq\kappa} \DOp_n(\alpha)
\end{equation*}
and the submodule
\begin{equation*}
\Span\{\underline{s}^{\vee}\ \text{such that}\ \kappa^{\vee}(\underline{s})\lneqq\kappa\}\subset\DOp_n(\kappa).
\end{equation*}
\end{lemm}

\begin{proof}
We adapt the argument of Proposition~\ref{SecondStep:DualModuleKStructure:Colimits}.
Set
\begin{equation*}
S = \Span\{\underline{s}^{\vee}\ \text{such that}\ \kappa^{\vee}(\underline{s})\lneqq\kappa\}.
\end{equation*}
We have clearly $\lambda(L\DOp_n(\kappa))\subset S$.
We have a map
\begin{equation*}
\psi: S\rightarrow L\DOp_n(\kappa)
\end{equation*}
sending a basis element $\underline{s}^{\vee}$, $\kappa^{\vee}(\underline{s})\lneqq\kappa$,
to the same element $\underline{s}^{\vee}$ in the summand $\DOp_n(\kappa^{\vee}(\underline{s}))$
of the latching object.
We have clearly $\lambda\psi = \id$.
Conversely, a basis element $\underline{s}^{\vee}\in\DOp_n(\alpha)$
is identified in the latching object
with the same element $\underline{s}^{\vee}$
coming from $\DOp_n(\kappa^{\vee}(\underline{s}))$
since $\underline{s}^{\vee}\in\DOp_n(\alpha)\Rightarrow\kappa^{\vee}(\underline{s})\leq\alpha$.
Hence, we also have $\psi\lambda = \id$.
\end{proof}

\begin{lemm}\label{FinalStep:KModuleCofibration:SuspensionSplitting}
The morphism
\begin{equation*}
\sigma^*: \DOp_{n-1}(r)\rightarrow\DOp_n(r)
\end{equation*}
has a retraction $\tau^*$ mapping the modules
\begin{equation*}
L\DOp_m(\kappa)\subset\DOp_m(\kappa)\subset\DOp_m(r)
\end{equation*}
for $m=n$ into the same modules for $m=n-1$
and preserving the splitting
\begin{equation*}
\DOp_m(\kappa) = L\DOp_m(\kappa)
\oplus\Span\{\underline{s}^{\vee}\ \text{such that}\ \kappa^{\vee}(\underline{s})\not<\kappa\}.
\end{equation*}
\end{lemm}

\begin{proof}
The section of the proof of Lemma~\ref{FirstStep:Equations:SuspensionSurjectivity}
gives by duality a morphism
\begin{equation*}
\tau^*: \DOp_n(r)\rightarrow\DOp_{n-1}(r)
\end{equation*}
mapping a basis element $\underline{s}^{\vee} = (s_0,\dots,s_d)^{\vee}$ with $\kappa(s_0,\dots,s_d) = (\mu,\sigma)$
to another basis element
\begin{equation*}
\tau^*((s_0,\dots,s_d)^{\vee}) = \pm (t_1,\dots,t_{r-1},s_0,\dots,s_d)^{\vee}
\end{equation*}
and an easy inspection of the definition of the permutations $t_1,\dots,t_{r-1}$
shows that $\kappa(t_1,\dots,t_{r-1},s_0,\dots,s_d) = (\mu-1,\sigma)$.
By inversion,
we obtain the identity
\begin{align*}
\kappa^{\vee}(t_1,\dots,t_{r-1},s_0,\dots,s_d) & = ((n-2)-(\mu-1),\sigma) \\
& = (n-1-\mu,\sigma) = \kappa^{\vee}(s_0,\dots,s_d)
\end{align*}
and we conclude that the map $\tau^*$ gives the required retraction.
\end{proof}

Lemmas~\ref{FinalStep:KModuleCofibration:LatchingObject}-\ref{FinalStep:KModuleCofibration:SuspensionSplitting}
imply readily:

\begin{lemm}\label{FinalStep:KModuleCofibration:LatchingSplitting}
The extended latching morphism of Lemma~\ref{FinalStep:KModuleCofibration:LatchingCofibration}
is split injective.\qed
\end{lemm}

By Lemma~\ref{FirstStep:Conclusion:BarrattEcclesDegreeBounds},
we also have:

\begin{fact}\label{FinalStep:KModuleCofibration:DegreeBounds}
The dg-modules
\begin{equation*}
L\DOp_m(\kappa)\subset\DOp_m(\kappa)\subset\DOp_m(r)
\end{equation*}
are bounded, for all $m<\infty$.
\end{fact}

This fact and Lemma~\ref{FinalStep:KModuleCofibration:LatchingSplitting}
imply the conclusion of Lemma~\ref{FinalStep:KModuleCofibration:LatchingCofibration}
and achieves the proof of Proposition~\ref{FinalStep:KModuleCofibration:SuspensionCofibration}.\qed

\subsection{The lifting argument}\label{FinalStep:KOperadLifting}
The goal of this subsection is to define the liftings $\tilde{\phi}_n: \BOp^c(\DOp_n)\rightarrow\EOp$
of the morphisms of $\K$-operads $\phi_n: \BOp^c(\DOp_n)\rightarrow\COp$
equivalent by adjunction to the morphisms of Lemma~\ref{FirstStep:Result}.
To fulfil the requirements of Theorem~\ref{Prelude:RealizationTheorem},
we have to fix $\tilde{\phi}_n: \BOp^c(\DOp_n)\rightarrow\EOp$
on the $2$-ary part of the generating $\K$-diagram $\DOp_n$
and we analyze this issue first.

Throughout this subsection, we use the short notation $\MOp_n = \Sigma^{-1}\DOp_n$
for the desuspension of~$\DOp_n$.

\subsubsection{The $2$-ary components}\label{FinalStep:KOperadLifting:TwoaryTerms}
Recall that $\EOp(2)$
is identified with the standard $\Sigma_2$-free resolution of the trivial representation of $\Sigma_2$:
\begin{equation*}
\ZZ\mu_0\oplus\ZZ\tau\mu_0\xleftarrow{\tau-1}\ZZ\mu_1\oplus\ZZ\tau\mu_1
\xleftarrow{\tau+1}\ZZ\mu_2\oplus\ZZ\tau\mu_2\xleftarrow{\tau-1}\cdots,
\end{equation*}
where $\mu_d$ is the element of $\EOp(2)$
defined by the alternate $d$-simplex $(\id,\tau,\id,\dots)$
and $\tau$ refers to the transposition of $(1,2)$.
For the submodules $\EOp(\kappa)$ associated to oriented weight systems $\kappa = (\mu,\id),(\mu,\tau)\in\K(2)$,
we have clearly:
\begin{align*}
& \EOp(\kappa)_d = \EOp(2)_d = \ZZ\mu_d\oplus\ZZ\tau\mu_d
&& \quad\text{in degree $d<\mu_{1 2}$},\\
& \EOp(\mu,\id)_d = \ZZ\mu_d\quad\text{and}\quad\EOp(\mu,\tau)_d = \ZZ\tau\mu_d
&& \quad\text{in degree $d=\mu_{1 2}$},\\
& \EOp(\kappa)_d = 0
&& \quad\text{in degree $d>\mu_{1 2}$}.
\end{align*}

The submodule $\EOp_n(2)$
is identified with the truncation of $\EOp(2)$
in degree $d<n$.
The dual dg-module $\MOp_n(2) = \Sigma^{-1}(\DOp_n)(2)$
is identified with the chain complex:
\begin{equation*}
\ZZ\mu_{n-1}^{\vee}\oplus\ZZ\tau\mu_{n-1}^{\vee}
\xleftarrow{\tau-1}\ZZ\mu_{n-2}^{\vee}\oplus\ZZ\tau\mu_{n-2}^{\vee}
\xleftarrow{\tau+1}\cdots
\xleftarrow{\tau\pm 1}\ZZ\mu_0^{\vee}\oplus\ZZ\tau\mu_0^{\vee}
\end{equation*}
where $\mu_d^{\vee}$ and $\tau\mu_d^{\vee}$
are dual basis elements, put in degree $n-1-d$,
of the alternate simplices
$\mu_d = (\id,\tau,\id,\dots)$ and $\tau\mu_d = (\tau,\id,\tau,\dots)$.
For an oriented weight systems $\kappa = (\mu,\id),(\mu,\tau)\in\K(2)$
such that $\mu_{1 2}\leq n-1$,
we also obtain:
\begin{align*}
& \MOp_n(\kappa)_d = \ZZ\mu_{n-1-d}^{\vee}\oplus\ZZ\tau\mu_{n-1-d}^{\vee}
&& \quad\text{in degree $d<\mu_{1 2}$},\\
& \MOp_n(\mu,\id)_d = \ZZ\mu_{n-1-d}^{\vee}\quad\text{and}\quad\MOp_n(\mu,\tau)_d = \ZZ\tau\mu_{n-1-d}^{\vee}
&& \quad\text{in degree $d=\mu_{1 2}$},\\
& \MOp_n(\kappa)_d = 0
&& \quad\text{in degree $d>\mu_{1 2}$}.
\end{align*}
In the case $\mu_{1 2}>n-1$,
we have $\MOp_n(\kappa)_d = \MOp_n(2)_d$
for every $d$.

The dual of the suspension morphism of the Barratt-Eccles operad
satisfies:
\begin{equation*}
\sigma^*(\mu_{n-2-d}^{\vee}) = \mu_{n-1-d}^{\vee}
\quad\text{and}\quad\sigma^*(\mu_{n-1-d}^{\vee}) = \mu_{n-1-d}^{\vee},
\end{equation*}
for every $0\leq d\leq n-2$
and hence identifies $\MOp_{n-1}(2)$
with a truncation of $\MOp_n(2)$.

\medskip
For each $\kappa\in\K_n(2)$,
we have a natural morphism
\begin{equation*}
\tilde{\phi}_n: \MOp_n(\kappa)\rightarrow\EOp(\kappa)
\end{equation*}
defined by $\tilde{\phi}_n(\mu_{n-1-d}^{\vee}) = \mu_d$
on basis elements.
From our observations,
we conclude:

\begin{fact}\label{FinalStep:KOperadLifting:TwoaryKModuleLifting}
In arity $r = 2$,
the just defined morphisms form, for each fixed $n\geq 1$, a morphism of $\K(2)$-diagrams
\begin{equation*}
\tilde{\phi}_n: \MOp_n(\kappa)\rightarrow\EOp(\kappa),
\end{equation*}
commuting with the action of permutations $w\in\Sigma_2$,
with augmentations
\begin{equation*}
\xymatrix{ \MOp_n(\kappa)\ar@{.>}[rr]\ar[dr] &&
\EOp(\kappa)\ar[dl] \\
& \ZZ & },
\end{equation*}
and so that the diagram
\begin{equation*}
\xymatrix{ \MOp_1(\kappa)\ar[r]^{\sigma^*}\ar@{.>}[d]_{\tilde{\phi}_1} &
\cdots\ar[r]^(0.3){\sigma^*} &
\MOp_n(\kappa)\ar[r]^{\sigma^*}\ar@{.>}[d]_{\tilde{\phi}_n} &
\cdots \\
\EOp(\kappa)\ar[r]_{=} & \cdots\ar[r]_{=} &
\EOp(\kappa)\ar[r]_{=} & \cdots }
\end{equation*}
commutes, for each $\kappa\in\K(2)$.
\end{fact}

\subsubsection{The bottom filtration layer of a quasi-free operad}\label{FinalStep:KOperadLifting:TwoaryQuasiFreeBottomLayer}
In the remainder of this section,
we use the notation $M^{(2)}$ to represent the $\Sigma_*$-object $M^{(2)}\subset M$
such that
\begin{equation*}
M^{(2)}(r) = \begin{cases} M(2), & \text{if $r=2$}, \\ 0, & \text{otherwise}. \end{cases}
\end{equation*}
If $M$ has a $\K$-structure,
then so does the $\Sigma_*$-object $M^{(2)}\subset M$
and we may also use the notation $M^{(2)}$
to refer to the symmetric $\K$-diagram underlying $M^{(2)}(r)$.
If necessary,
then we use dummy variables to mark the distinction
between the $\Sigma_*$-object $M^{(2)}$
and its underlying $\K$-diagram,
as usual.

The assertion of Fact~\ref{FinalStep:KOperadLifting:TwoaryKModuleLifting}
amounts to the definition of a morphism of symmetric $\K$-diagrams
\begin{equation*}
\MOp_n^{(2)}\xrightarrow{\tilde{\phi}_n}\EOp^{(2)}\subset\EOp,
\end{equation*}
for every $n\geq 1$.

For a quasi-free operad $\POp = (\FOp(M),\partial)$
such that $M(0) = M(1) = 0$,
we have $M_{\leq 1} = 0$ and the submodule $M_{\leq 2}$ of~\S\ref{Interlude:GraphOperadModel:QuasiFreeFiltration}
is reduced to $M_{\leq 2} = M^{(2)}$.
Moreover,
if we assume $\partial(M)\subset\bigoplus_{m\geq 2}\FOp_{m}(M)$,
then the observation of~\S\ref{Interlude:GraphOperadModel:QuasiFreeFiltration}
\begin{equation*}
\partial(M_{\leq r})\subset\FOp(M_{\leq r-1})
\end{equation*}
implies that $\partial$ vanishes on $M_{\leq 2}$
and the $2$nd layer of the quasi-free operad filtration
\begin{equation*}
\POp_{\leq r} = (\FOp(M_{\leq r}),\partial)\subset\POp
\end{equation*}
is identified with the free operad $\POp_{\leq 2} = \FOp(M^{(2)})$.

The morphisms of $\K$-operads $\tilde{\phi}_n: \FOp(\MOp_n^{(2)})\rightarrow\EOp$
associated to the morphisms of Fact~\ref{FinalStep:KOperadLifting:TwoaryKModuleLifting}
define morphisms on the $2$nd layer of the quasi-free operads $\BOp^c(\DOp_n) = (\FOp(\MOp_n),\partial)$.

\medskip
Note that:

\begin{fact}\label{FinalStep:KOperadLifting:TwoaryRestriction}
The restriction of the morphisms of Lemma~\ref{FirstStep:Result}
\begin{equation*}
\phi_n: \BOp^c(\DOp_n)\rightarrow\COp
\end{equation*}
to the generators of arity $2$ of the cobar construction
agrees with the obvious augmentation $\epsilon: \MOp_n(2)\rightarrow\ZZ$
\end{fact}

This assertion follows from the definition of $\phi_n$
in~\S\ref{FirstStep:Conclusion}.
This condition is also automatically fixed by the requirement of Lemma~\ref{FirstStep:Result},
because the morphisms $\phi_n$ of Lemma~\ref{FirstStep:Result}
necessarily vanish in degree $d>0$
and are determined by $\phi_1$ when we take their restriction to $\MOp_n(2)_0 = \Sigma^{-1}\DOp_n(2)_0$.

\medskip
The morphism of Fact~\ref{FinalStep:KOperadLifting:TwoaryKModuleLifting}
induces a morphism of $\K$-operads
\begin{equation*}
\tilde{\phi}_n: \FOp(\MOp_n^{(2)})\rightarrow\EOp,
\end{equation*}
for every $n\geq 1$.
From the assertions of Fact~\ref{FinalStep:KOperadLifting:TwoaryKModuleLifting}
and the observation of Fact~\ref{FinalStep:KOperadLifting:TwoaryRestriction}
we obtain:

\begin{lemm}\label{FinalStep:KOperadLifting:TwoaryFreeKOperadLifting}
The morphisms of $\K$-operads $\tilde{\phi}_n: \FOp(\MOp_n^{(2)})\rightarrow\EOp$
make the diagram
\begin{equation*}
\xymatrix{ \FOp(\MOp_n^{(2)})\ar[d]\ar@{.>}[r]^(0.6){\tilde{\phi}_n} & \EOp\ar[d] \\
\BOp^c(\DOp_n)\ar[r]_(0.6){\phi_n} & \COp }
\end{equation*}
commute, for every $n\geq 1$,
and commute with the morphisms of free operads $\sigma^*: \FOp(\MOp_{n-1}^{(2)})\rightarrow\FOp_n(\MOp_n^{(2)})$
induced by the suspension morphism of the Barratt-Eccles operad.
\qed
\end{lemm}

The lifting argument, which motivates the constructions of~\S\S\ref{Interlude}-\ref{SecondStep},
is given by the following proposition:

\begin{mainsecprop}\label{FinalStep:KOperadLifting:MainConstruction}
The morphisms of $\K$-operads $\phi_n$
deduced from the results of Lemma~\ref{FirstStep:Result} and Lemma~\ref{SecondStep:Result}
admit liftings
\begin{equation*}
\xymatrix{ & \EOp\ar[d] \\
\BOp^c(\DOp_n)\ar[r]_(0.6){\phi_n}\ar@{.>}[ur]^(0.6){\exists\tilde{\phi}_n} & \COp }
\end{equation*}
so that the diagrams
\begin{equation*}
\xymatrix{ \BOp^c(\DOp_{n-1})\ar@{.>}[d]_{\tilde{\phi}_{n-1}}\ar[r]^{\sigma^*} &
\BOp^c(\DOp_n)\ar@{.>}[d]_{\tilde{\phi}_n} &
\FOp(\MOp_n^{(2)})\ar[l]\ar[d]^{\tilde{\phi}_n} \\
\EOp\ar[r]_{=} & \EOp & \EOp\ar[l]^{=} }
\end{equation*}
commute, for all $n>1$.
\end{mainsecprop}

\begin{proof}
The lifting $\tilde{\phi}_n$ is defined by induction on $n$,
starting with the Koszul duality augmentation of the associative operad
\begin{equation*}
\BOp^c(\AOp)\xrightarrow{\epsilon}\AOp\subset\EOp
\end{equation*}
in the case $\EOp_1 = \AOp$.
At the $n$th stage of the induction process,
the already defined morphisms
fit the solid frame of a commutative diagram of $\K$-operads
\begin{equation*}
\xymatrix{ \BOp^c(\DOp_{n-1})
\bigvee_{\FOp(\MOp_{n-1}^{(2)})}\FOp(\MOp_n^{(2)})
\ar[d]\ar[rr]^(0.7){(\tilde{\phi}_{n-1},\tilde{\phi}_{n})} && \EOp\ar@{->>}[d]^{\sim} \\
\BOp^c(\DOp_n)\ar[rr]_(0.7){\phi_n}\ar@{.>}[urr] && \COp }.
\end{equation*}
Proposition~\ref{FinalStep:KModuleCofibration:SuspensionCofibration} implies,
according to Proposition~\ref{Interlude:GraphOperadModel:CofibrantKOperads},
that the left-hand side vertical morphism
is a cofibration of $\K$-operads (recall that $\FOp(\MOp_n^{(2)}) = \sk_2\BOp^c(\DOp_n)$, $\forall n$).
The lifting axiom of the model category of $\K$-operads
implies the existence of a fill-in morphism in this diagram
and this gives the desired morphism $\tilde{\phi}_n: \BOp^c(\DOp_n)\rightarrow\EOp$.
\end{proof}

\subsection{Conclusion of the lifting argument}\label{FinalStep:Conclusion}
The morphisms of Proposition~\ref{FinalStep:KOperadLifting:MainConstruction}
give by colimit a sequence of morphisms:
\begin{equation}\label{eqn:ColimitDiagram}
\xymatrix{ \colim_{\K_1}\BOp^c(\DOp_1)\ar[r]^(0.7){\sigma^*}\ar[d]_{\tilde{\phi}_1} &
\cdots\ar[r]^(0.25){\sigma^*} &
\colim_{\K_n}\BOp^c(\DOp_n)\ar[r]^(0.75){\sigma^*}\ar@{.>}[d]_{\exists\tilde{\phi}_n} & \cdots \\
\colim_{\K_1}\EOp\ar[r] &
\cdots\ar[r] &
\colim_{\K_n}\EOp\ar[r] & \cdots }.
\end{equation}
The upper row of~(\ref{eqn:ColimitDiagram})
is isomorphic to
\begin{equation*}
\BOp^c(\DOp_1)\xrightarrow{\sigma^*}
\BOp^c(\DOp_2)\xrightarrow{\sigma^*}
\cdots\xrightarrow{\sigma^*}
\BOp^c(\DOp_n)\xrightarrow{\sigma^*}\cdots
\end{equation*}
by Theorem~\ref{SecondStep:OperadKStructure:DefinitionResult}
and the lower row to
\begin{equation*}
\EOp_1\hookrightarrow\EOp_2\hookrightarrow\cdots\hookrightarrow\EOp_n\hookrightarrow\cdots
\end{equation*}
by Proposition~\ref{Interlude:CompleteGraphOperad:BarrattEcclesLayers}.
Hence, our construction returns morphisms $\psi_n: \BOp^c(\DOp_n)\rightarrow\EOp_n$.

In arity $r=2$,
the relation $\colim_{\kappa\in\K_n(2)}\MOp_n(\kappa)\simeq\MOp_n(2)$ gives, by Proposition~\ref{SecondStep:OperadKStructure:ColimitLayers},
an isomorphism between the colimit of the free $\K$-operad $\colim_{\K_n}\FOp(\MOp_n^{(2)})$,
where we consider the underlying symmetric $\K$-diagram of~$\MOp_n^{(2)}$,
and the usual free operad $\FOp(\MOp_n^{(2)})$ on the $\Sigma_*$-object
associated to this symmetric $\K$-diagram~$\MOp_n^{(2)}(\kappa)$.
By definition of the lifting $\tilde{\phi}_n$
in Proposition~\ref{FinalStep:KOperadLifting:MainConstruction},
the composites
\begin{equation*}
\xymatrix{ \MOp_n^{(2)}(r)\ar[r]
& \displaystyle\colim_{\kappa\in\K_n(r)}\FOp(\MOp_n^{(2)})(\kappa)\ar[r]
& \displaystyle\colim_{\kappa\in\K_n(r)}\BOp^c(\DOp_n)(\kappa)\ar[d]_{\simeq}\ar[r]^{\tilde{\phi}_n}
& \displaystyle\colim_{\kappa\in\K_n(r)}\EOp(\kappa)\ar[d] & \\
&& \BOp^c(\DOp_n)(r)\ar@{.>}[r]_{\psi_n}
& \EOp_n(r)\ar@{}[r]|{\displaystyle\subset\EOp(r)} & }
\end{equation*}
agree in arity $r=2$ with the natural embedding
\begin{equation*}
\MOp_n(2)\xrightarrow{\simeq}\EOp_n(2)\subset\EOp(2).
\end{equation*}
From this observation,
we obtain:

\begin{fact}\label{FinalStep:Conclusion:TwoaryColimitRestriction}
The restriction of our morphism $\psi_n: \BOp^c(\DOp_n)\rightarrow\EOp_n$
to the generating $\Sigma_*$-object
of $\BOp^c(\DOp_n)$ in arity $r=2$
agrees in homology with the morphism
\begin{equation*}
H_*(\MOp_n)(2)\rightarrow H_*(\EOp_n)(2)
\end{equation*}
mapping $\lambda_{n-1}^{\vee}$
to the representative of the product $\mu$
and $\mu^{\vee}$ to the representative of the bracket $\lambda_{n-1}$,
where we use the identity
\begin{equation*}
H_*(\MOp_n) = \Sigma^{-1} H_*(\DOp_n) = \Sigma^{-1}(\Lambda^{-n}\GOp_n^{\vee})
\end{equation*}
and the notation of Proposition~\ref{Prelude:GerstenhaberOperads:KoszulDualityStatement}
for the dual basis of~$\mu,\lambda_{n-1}\in\GOp_n(2)$
in $\GOp_n^{\vee}(2)$.
\end{fact}

Recall that the edge morphism $\edge: \BOp^c(H_*(\DOp_n))\rightarrow H_*(\BOp^c(\DOp_n))$
in the homological requirement of Assertion~(\ref{Prelude:RealizationTheorem:Prescription})
of Theorem~\ref{Prelude:RealizationTheorem}
is defined on the generating $\Sigma_*$-object $\Sigma^{-1} H_*(\bar{\DOp}_n)$
by the homology of the natural embedding $\Sigma^{-1}\bar{\DOp}_n = \MOp_n\hookrightarrow\BOp^c(\DOp_n)$
in arity $r=2$ and by the null morphism in arity $r\not=2$,
like the Koszul duality equivalence of the $n$-Gerstenhaber operad $\GOp_n$.
Therefore Fact~\ref{FinalStep:Conclusion:TwoaryColimitRestriction} implies immediately:

\begin{fact}\label{FinalStep:Conclusion:HomologicalKoszulRequirement}
The composite of our morphisms $\psi_n: \BOp^c(\DOp_n)\rightarrow\EOp_n$
with the edge morphism $\edge: \BOp^c(H_*(\DOp_n))\rightarrow H_*(\BOp^c(\DOp_n))$
agrees with the Koszul duality equivalence of the $n$-Gerstenhaber operad,
as required by Assertion~(\ref{Prelude:RealizationTheorem:Prescription})
of Theorem~\ref{Prelude:RealizationTheorem}.
\end{fact}

This observation achieves the verification of Lemma~\ref{FinalStep:Result}
and the proof of theorems~\ref{Prelude:RealizationTheorem}-\ref{Prelude:BarDualityTheorem}.\qed

\section*{Epilogue}


In the prologue,
we observe that our main results, theorems~\ref{Prelude:RealizationTheorem}-\ref{Prelude:BarDualityTheorem},
give an interpretation of the suspension morphisms $\sigma: \EOp_n\rightarrow\Lambda^{-1}\EOp_{n-1}$
in terms of the homotopy of $E_n$-operads.
Namely,
the cooperad morphism $\sigma^*: \Lambda^{1-n}\EOp_{n-1}^{\vee}\rightarrow\Lambda^{-n}\EOp_n^{\vee}$
associated to $\sigma$
and the embedding $\iota: \EOp_{n-1}\rightarrow\EOp_n$
are dual to each other with respect to the cobar construction.
This assertion means:
\begin{itemize}
\item
on one side,
we have a commutative square
\begin{equation}\label{eqn:EmbeddingBarDuality}
\xymatrix{ \BOp^c(\Lambda^{1-n}\EOp_{n-1}^{\vee})\ar[r]^{\sigma^*}\ar@{.>}[d]_{\sim} &
\BOp^c(\Lambda^{-n}\EOp_n^{\vee})\ar@{.>}[d]^{\sim} \\
\EOp_{n-1}\ar[r]_{\iota} & \EOp_n }
\end{equation}
in which vertical arrows are weak-equivalences;
\item
on the other side,
we have a commutative square
\begin{equation}\label{eqn:SuspensionBarDuality}
\xymatrix{ \BOp^c(\Lambda^{-n}\EOp_{n}^{\vee})\ar[r]^{\iota^*}\ar@{.>}[d]_{\sim} &
\BOp^c(\Lambda^{-n}\EOp_n^{\vee})\ar[r]^(0.43){\simeq} &
\Lambda^{-1}\BOp^c(\Lambda^{1-n}\EOp_n^{\vee})\ar@{.>}[d]^{\sim} \\
\EOp_n\ar[rr]_{\sigma} && \Lambda^{-1}\EOp_{n-1} },
\end{equation}
\end{itemize}
where $\iota^*$ represents the image of the embedding $\iota$
under the functor $\BOp^c(\Lambda^{-n}(-)^{\vee})$.

The existence of~(\ref{eqn:EmbeddingBarDuality})
is nothing but the assertion of theorems~\ref{Prelude:RealizationTheorem}-\ref{Prelude:BarDualityTheorem}.
Diagram~(\ref{eqn:SuspensionBarDuality})
is obtained from~(\ref{eqn:EmbeddingBarDuality})
by an application of the adjunction relation of the bar duality of operads (see~\cite[Theorem 2.17]{GetzlerJones})
using that the cobar construction commutes with operadic suspensions.

The purpose of this concluding section is give an intrinsic interpretation, in terms of the homotopy of $E_n$-operads,
of the action of the Barratt-Eccles operad
on the reduced normalized cochain complex of spheres~$\bar{N}^*(S^m)$.
The result arises as a consequence of the relation~(\ref{eqn:SuspensionBarDuality})

Recall that this action is defined by a morphism $\nabla_{S^m}: \EOp\rightarrow\End_{\bar{N}^*(S^m)}$,
where $\End_{\bar{N}^*(S^m)}$
is the endomorphism operad of $\bar{N}^*(S^m)$.
Since $\bar{N}^*(S^m) = \ZZ[-m]$ (recall that $\ZZ[d]$ denotes the free $\ZZ$-module of rank $1$ put in lower degree $d$),
we obtain $\End_{\bar{N}^*(S^m)}(r) = \ZZ[rm-m]$,
from which we deduce the identity:
\begin{equation*}
\End_{\bar{N}^*(S^m)} = \Lambda^{-m}\COp,
\end{equation*}
where $\Lambda^{-m}\COp$ is the operadic $m$-fold desuspension of the commutative operad $\COp$.
Hence,
the action of $\EOp$ on $\bar{N}^*(S^m)$
is represented by a morphism $\sigma_m: \EOp\rightarrow\Lambda^{-m}\COp$.

The identity $\End_{\bar{N}^*(S^m)}(r) = \ZZ[rm-m]$
implies that $\nabla_{S^m}$
is defined by a cochain of degree $rm-m$
on each $\EOp(r)$, $r\in\NN$.
According to~\cite{BergerFresse},
these cochains are nothing but the $m$-fold cup products $\sgn^{\cup m}: \EOp(r)\rightarrow\ZZ$
of the cochains $\sgn$
defined in~\S\ref{Prelude:BarrattEccles:CircleCochainMultiplications}.
The associativity relation between cup products and cap products
implies:

\begin{concludingfact}
The morphism $\sigma_m: \EOp\rightarrow\Lambda^{-m}\COp$
giving the action of the Barratt-Eccles operad on $\bar{N}^*(S^m)$
is identified with the composite of the $m$-fold suspension morphism
\begin{equation*}
\EOp\xrightarrow{\sigma}\Lambda^{-1}\EOp\xrightarrow{\sigma}\cdots\xrightarrow{\sigma}\Lambda^{-m}\EOp
\end{equation*}
with the morphism
\begin{equation*}
\Lambda^{-m}\EOp\xrightarrow{\sim}\Lambda^{-m}\COp = \End_{\bar{N}^*(S^m)}
\end{equation*}
induced by the augmentation of the Barratt-Eccles operad.
\end{concludingfact}

Note that the action of $\EOp_n\subset\EOp$
on $\bar{N}^*(S^m)$
vanishes when $n\leq m$.
In the case $n>m$,
we deduce from~(\ref{eqn:SuspensionBarDuality}):

\begin{mainthm}\label{Epilogue:SphereAction}
For every $n>m$,
we have a commutative diagram
\begin{equation*}
\xymatrix{ \BOp^c(\Lambda^{-n}\EOp_{n}^{\vee})\ar[d]_{\iota^*}\ar[rr]^{\sim} &&
\EOp_n\ar[d]^{\sigma}\ar@{^{(}->}[]!R+<4pt,0pt>;[r] & \EOp\ar[d]^{\sigma} \\
\vdots\ar[d]_{\iota^*} && \vdots\ar[d]^{\sigma} & \vdots\ar[d]^{\sigma} \\
\BOp^c(\Lambda^{-n}\EOp_{n-m}^{\vee})\ar[r]_(0.43){\simeq} &
\Lambda^{-m}\BOp^c(\Lambda^{m-n}\EOp_{n-m}^{\vee})\ar[r]_(0.6){\sim}\ar@/_1em/[drr]_{\Lambda^{-m}\phi_{n-m}} &
\Lambda^{-m}\EOp_{n-m}\ar@{^{(}->}[]!R+<4pt,0pt>;[r] &
\Lambda^{-m}\EOp\ar[d]^{\sim} \\
&&& \Lambda^{-m}\COp = \End_{\bar{N}^*(S^m)} }
\end{equation*}
giving the action of the $E_n$-operad $\BOp^c(\Lambda^{-n}\EOp_n^{\vee})$
on $\bar{N}^*(S^m)$.
\end{mainthm}

In a follow up~\cite{FresseOperadMaps},
we prove that each morphism $\phi_n: \BOp^c(\Lambda^{-n}\EOp_n^{\vee})\rightarrow\COp$
in Lemma~\ref{FirstStep:Result}
is uniquely determined up to homotopy.
Hence,
Theorem~\ref{Epilogue:SphereAction} implies that the action of an $E_n$-operad on the cochain complex of a sphere $\bar{N}^*(S^m)$
has an intrinsic characterization in terms of the embeddings $\iota: \EOp_{n-1}\hookrightarrow\EOp_n$
underlying the definition of an $E_n$-operad.

\section*{Glossary of notation}

\begin{trivlist}

\indexspace\item $\AOp$: the associative operad

\item $\BOp^c(\DOp)$: the operadic cobar construction of a cooperad
(\S\ref{Prelude:KoszulDuality:CobarConstruction})

\item $\C$: the base category of dg-modules

\item $\COp$: the commutative operad

\item $\COp_n$: the topological operad of little $n$-cubes

\item $\DOp$: any cooperad

\item $\DOp_n$: a short notation for $\DOp_n = \Lambda^{-n}\EOp_n^{\vee}$

\item $\EOp$: the Barratt-Eccles chain operad
(\S\ref{Prelude:BarrattEccles})

\item $\EOp_n$: the suboperad of $\EOp$ equivalent to the chain operad of little $n$-cubes
(\S\ref{Prelude:BarrattEccles:LittleCubesFiltration})

\item $\FOp(M)$: the free operad on a $\Sigma_*$-object
(\S\ref{Prelude:KoszulDuality:FreeOperad}),
on a symmetric $\K$-diagram
(\S\S\ref{Interlude:GraphOperadModel:FreeKOperad}-\ref{Interlude:GraphOperadModel:FreeOperadStructure})

\item $\GOp_n$: the $n$-Gerstenhaber operad
(\S\ref{Prelude:GerstenhaberOperads})

\item $\K$: the complete graph operad
(\S\ref{Interlude:CompleteGraphOperad})

\item $\K_n$: the filtration layers of the complete graph operad
(\S\ref{Interlude:CompleteGraphOperad})

\item $\kappa$: any oriented weight-system
(\S\ref{Interlude:CompleteGraphOperad:Poset})

\item $\KOp(\POp)$: the Koszul dual of an operad
(\S\ref{Prelude:KoszulDuality:KoszulDualDefinition})

\item $\LOp$: the Lie operad

\item $\Lambda$: the operadic suspension
(\S\ref{Prelude:BarrattEccles:SuspensionMorphisms})

\item $\MOp_n$: a short notation for $\MOp_n = \Sigma^{-1}\DOp_n$

\item $\POp$: any operad

\item $\Sigma$: the suspension of dg-modules

\item $\Sigma_r$: the symmetric group on $r$ letters

\item $\Theta$: the category of trees

\item $(-)^{\vee}$: the complement of oriented weight-systems (\S\ref{SecondStep:DualModuleKStructure:WeightComplement})
or the duality of $\ZZ$-modules

\item $\WOp$: the simplicial Barratt-Eccles operad
(\S\ref{SecondStep:DualModuleKStructure:Definition})

\end{trivlist}


\begin{thebibliography}{99}

\bibitem{BarrattEccles}
M. Barratt, P. Eccles,
\emph{On $\Gamma_+$-structures. I. A free group functor for stable homotopy theory},
Topology \textbf{13} (1974), 25--45.

\bibitem{BergerCell}
C. Berger,
\emph{Op\'erades cellulaires et espaces de lacets it\'er\'es},
Ann. Inst. Fourier \textbf{46} (1996), 1125--1157.

\bibitem{BergerSummary}
\bysame,
\emph{Combinatorial models for real configuration spaces and $E_n$-operads},
\emph{in} ``Operads: Proceedings of Renaissance Conferences (Hartford, CT/Luminy, 1995)'',
Contemp. Math. \textbf{202}, Amer. Math. Soc. (1997), 37--52.

\bibitem{BergerFressePrismatic}
C. Berger, B. Fresse,
\emph{Une d\'ecomposition prismatique de l'op\'erade de Barratt-Eccles},
C. R. Acad. Sci. Paris S\'er. I Math. \textbf{335} (2002), 365-370.

\bibitem{BergerFresse}
\bysame,
\emph{Combinatorial operad actions on cochains},
Math. Proc. Camb. Philos. Soc. \textbf{137} (2004), 135--174.

\bibitem{BergerMoerdijk}
C. Berger, I. Moerdijk,
\emph{Axiomatic homotopy theory for operads},
Comment. Math. Helv. \textbf{78} (2003), 805--831.

\bibitem{BoardmanVogt}
J. Boardman, R. Vogt,
\emph{Homotopy invariant algebraic structures on topological spaces},
Lecture Notes in Mathematics \textbf{347}, Springer-Verlag, 1973.

\bibitem{Bourbaki}
N. Bourbaki,
\emph{Groupes et alg\`ebres de Lie, Chapitres 2 et 3},
Masson, 1972.

\bibitem{BrunFiedVogt}
M. Brun, Z. Fiedorowicz, R. Vogt,
\emph{On  the  multiplicative  structure of topological Hochschild homology},
Algebr. Geom. Topol. \textbf{7} (2007), 1633--1650.

\bibitem{Cohen}
F. Cohen,
\emph{The homology of $\mathcal{C}_{n+1}$-spaces, $n\geq 0$},
\emph{in} ``The homology of iterated loop spaces'',
Lecture Notes in Mathematics \textbf{533}, Springer-Verlag (1976), 207--351.

\bibitem{FoxMarkl}
T. Fox, M. Markl,
\emph{Distributive laws, bialgebras, and cohomology},
\emph{in} ``Operads: proceedings of renaissance conferences'',
Contemp. Math. \textbf{202}, Amer. Math. Soc. (1997), 167--206.

\bibitem{FresseSimpAlg}
B. Fresse,
\emph{On the homotopy of simplicial algebras over an operad},
Trans. Amer. Math. Soc. \textbf{352} (2000), 4113--4141.

\bibitem{FressePartitions}
\bysame,
\emph{Koszul duality of operads and homology of partition posets},
\emph{in} ``Homotopy theory: relations with algebraic geometry, group cohomology, and algebraic $K$-theory'',
Contemp. Math. \textbf{346}, Amer. Math. Soc. (2004), 115--215.

\bibitem{FresseCylinder}
\bysame, \emph{Operadic cobar constructions, cylinder objects and homotopy morphisms of algebras over operads},
\emph{in} ``Alpine perspectives on algebraic topology (Arolla, 2008)'', Contemp. Math. (to appear).
Preprint available at \href{http://arxiv.org/0902.0177}{\texttt{arXiv:0902.0177}} (2009).

\bibitem{FresseOperadMaps}
\bysame,
\emph{On mapping spaces of differential graded operads with the commutative operad as target},
preprint \href{http://arxiv.org/0909.3020}{\texttt{arXiv:0909.3020}} (2009).

\bibitem{GetzlerJones}
E. Getzler, J. Jones,
\emph{Operads, homotopy algebra and iterated integrals for double loop spaces},
preprint \texttt{arXiv:hep-th/9403055} (1994).

\bibitem{GinzburgKapranov}
V. Ginzburg, M. Kapranov,
\emph{Koszul duality for operads},
Duke Math. J. \textbf{76} (1995), 203-272.

\bibitem{Goerss}
P.G. Goerss,
\emph{On the Andr\'e-Quillen cohomology of commutative $\FF_2$-algebras},
Ast\'erisque \textbf{186}, Soci\'et\'e Math\'ematique de France, 1990.

\bibitem{HinichHomotopy}
V. Hinich,
\emph{Homological algebra of homotopy algebras},
Comm. Algebra \textbf{25} (1997), 3291--3323.

\bibitem{HinichSchechtman}
V. Hinich, V. Schechtman,
\emph{On homotopy limit of homotopy algebras},
\emph{in} ``$K$-theory, arithmetic and geometry (Moscow, 1984-1986)'',
Lecture Notes in Mathematics \textbf{1289},
Springer-Verlag (1987), 240-264.

\bibitem{Hirschhorn}
P. Hirschhorn,
\emph{Model categories and their localizations},
Mathematical Surveys and Monographs \textbf{99}, American Mathematical Society, 2003.

\bibitem{Hoffbeck}
E. Hoffbeck,
\emph{A Poincar\'e-Birkhoff-Witt criterion for Koszul operads},
Manuscripta Math. (to appear).
Preprint available at \href{http://arxiv.org/abs/0709.2286}{\texttt{arXiv:0709.2286}} (2007).

\bibitem{Hovey}
M. Hovey,
\emph{Model categories},
Mathematical Surveys and Monographs \textbf{63}, American Mathematical Society, 1999.

\bibitem{PoHu}
P. Hu,
\emph{Higher string topology on general spaces},
Proc. London Math. Soc. \textbf{93} (2006), 515--544.

\bibitem{Kontsevich}
M. Kontsevich,
\emph{Operads and motives in deformation quantization},
Lett. Math. Phys. \textbf{48} (1999), 35--72.

\bibitem{KrizMay}
I. Kriz, P. May,
\emph{Operads, algebras, modules and motives},
Ast\'erisque \textbf{233}, Soci\'et\'e Math\'ematique de France, 1995.

\bibitem{Mandell}
M. Mandell,
\emph{$E_\infty$-algebras and $p$-adic homotopy theory},
Topology \textbf{40} (2001), 43--94.

\bibitem{MandellIntegral}
\bysame,
\emph{Cochains and homotopy type},
Publ. Math. Inst. Hautes \'Etudes Sci. \textbf{103} (2006), 213--246.

\bibitem{Markl}
M. Markl, \emph{Distributive laws and Koszulness},
Ann. Inst. Fourier \textbf{46} (1996), 307--323.

\bibitem{MayOperations}
P. May,
\emph{A general algebraic approach to Steenrod operations},
\emph{in}
``The Steenrod Algebra and its Applications (Columbus, Ohio, 1970)'',
Lecture Notes in Mathematics \textbf{168}, Springer-Verlag (1970), 153--231

\bibitem{May}
\bysame,
\emph{The geometry of iterated loop spaces},
Lecture Notes in Mathematics \textbf{271}, Springer-Verlag, 1972.

\bibitem{McClureSmith}
J. McClure, J.H. Smith,
\emph{Multivariable cochain operations and little $n$-cubes},
J. Amer. Math. Soc. \textbf{16} (2003), 681--704.

\bibitem{Reutenauer}
C. Reutenauer,
\emph{Free Lie algebras},
London Mathematical Society Monographs \textbf{7},
Oxford University Press, 1993.

\bibitem{SmirnovChain}
V. Smirnov,
\emph{Homotopy theory of coalgebras},
Izv. Akad. Nauk SSSR Ser. Mat. \textbf{49} (1985), 1302--1321, 1343.
Translation in Math. USSR-Izv.\textbf{27} (1986), 575--592.

\bibitem{Smith}
J.H. Smith,
\emph{Simplicial group models for $\Omega^n\Sigma^n X$},
Israel J. Math. \textbf{66} (1989), 330--350.

\bibitem{Sinha}
D. Sinha,
\emph{The homology of the little disks operad},
preprint \href{http://arXiv:math/0610236}{\texttt{arXiv:math/0610236}} (2006).

\end{thebibliography}
\end{document}